%% file: main.tex
\documentclass{art}
\usepackage{mypreamble}

\begin{document}
\begin{frontmatter}

  \title{State-based Modal Logics for Free Choice}
  \runtitle{State-based Modal Logics for Free Choice}

  \author{\fnms{Maria Aloni, Aleksi Anttila, and Fan Yang}
    \ead[label=e1]{m.d.aloni@uva.nl}
    \ead[label=e2]{aleksi.i.anttila@helsinki.fi}
    \ead[label=e3]{fan.yang.c@gmail.com}
  }
  \address{Aloni\\
  Institute for Logic, Language and Computation \& Philosophy\\
    University of Amsterdam\\
    Science Park 107\\
    1098 XG Amsterdam\\
    NETHERLANDS\\
    \href{mailto:m.d.aloni@uva.nl}{m.d.aloni@uva.nl}\\
    \\
    Anttila\\
    Institute for Logic, Language and Computation\\
    University of Amsterdam\\
    Science Park 107\\
    1098 XG Amsterdam\\
    NETHERLANDS\\
    \href{mailto:aleksi.ilari.anttila@gmail.com}{aleksi.ilari.anttila@gmail.com}\\
    \\
    Yang\\
    Department of Philosophy and Religious Studies\\
    Utrecht University\\
    Janskerkhof 13\\
    3512 BL Utrecht \\
    NETHERLANDS\\
    \href{mailto:fan.yang.c@gmail.com}{fan.yang.c@gmail.com}\\
  }
  \runauthor{M.~Aloni, A.~Anttila, and F.~Yang}

\begin{abstract}
We study the mathematical properties of bilateral state-based modal logic (BSML), a modal logic employing state-based semantics (also known as team semantics), which has been used to account for free choice inferences and related linguistic phenomena. This logic extends classical modal logic with a nonemptiness atom which is true in a state if and only if the state is nonempty. We introduce two extensions of BSML and show that the extensions are expressively complete, and develop natural deduction axiomatizations for the three logics.
\end{abstract}



\end{frontmatter}

\input{sections/introduction}
\input{sections/preliminaries}
\input{sections/expressive_power}
\input{sections/axiomatization/axiomatization}
\input{sections/conclusion}

\input{sections/bibliography}

\end{document}

%% file: sections/introduction.tex
\section{Introduction} \label{section:introduction}
In this paper, we study \emph{bilateral state-based modal logic} ($\BSML$), a modal logic employing \emph{team semantics}. \BSML was introduced in \cite{aloni2022} to account for \emph{Free Choice} ($\FC$) inferences. In $\FC$-inferences, conjunctive meanings  are unexpectedly derived from  disjunctive sentences:
\ex. {\bf Free choice} \label{ex:fc}  \hfill \cite{vwright1968,kamp1973}
\a. You may go to the beach or   to the
      cinema. \\ 
    $\implicates$ You may go to the beach
      and you may go to the cinema.
      \b.  $\Diamond (b \vee c) \implicates \Diamond b \wedge \Diamond c$

The novel hypothesis at the core of the account in \cite{aloni2022} is that $\FC$ and other related inferences are a consequence of a tendency in human cognition to disregard structures that verify sentences by virtue of some empty configuration ({\emph{neglect-zero tendency}) \cite{Bott-Schlotterbeck-Klein:19}. Models that verify sentences by virtue of an empty witness-set are called {\it zero-models}---see Figure \ref{fig:zero-models} for an illustration.

 \begin{figure}[t]
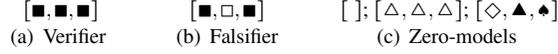

     \centering
 \subfigure[\footnotesize Verifier]{ 
 $\quad[\blacksquare, \blacksquare, \blacksquare]\quad$}
 \qquad \subfigure[\footnotesize Falsifier]{ 
 \ \ \ $[\blacksquare, \square, \blacksquare]$ \ \ \ }
 \qquad  \subfigure[\footnotesize Zero-models]{ 
$[ \  ]$; $[\triangle, \triangle, \triangle]$; $[\Diamond, \blacktriangle, \spadesuit]$}
    \caption{Models for the sentence {\it Every square is black.}}
    \label{fig:zero-models}
 \end{figure}

\BSML is introduced as a framework allowing for the formalization of the neglect-zero tendency and the rigorous study of its impact on linguistic interpretation. It extends classical modal logic ($\ML$) with a special \emph{nonemptiness atom} (\NE) requiring the adoption of \emph{team semantics}. In team semantics (introduced originally in \cite{hodges1997,hodges1997b} and developed further in the literature on \emph{dependence logic} \cite{vaananen2007,vaananen2008}), formulas are interpreted with respect to sets of evaluation points called \emph{teams}, rather than single points. In the type of modal team semantics employed by \BSML, teams are {\em sets of possible worlds}. In \cite{aloni2022}, teams are also called \emph{states} since they represent speakers' information states; we will use these terms interchangeably. The atom $\NE$ (introduced in \cite{vaananen2014,yang2017}) is true (or \emph{supported}) in a state iff the state is not empty: $s\vDash \NE$ iff $s\neq \emptyset$. This atom is used, in \cite{aloni2022}, to define a \emph{pragmatic enrichment function} $[\ ]^+$ whose core effect is to disallow zero-models; it is then shown that pragmatic enrichment yields non-trivial effects including the prediction (given certain preconditions) of both narrow-scope ($\Diamond (b \vee c) \implicates \Diamond b \wedge \Diamond c$) and wide-scope ($\Diamond b \vee \Diamond c \implicates \Diamond b \wedge \Diamond c$) $\FC$-inferences, as well as their cancellation under negation (see \ref{ex:negation} below). See Figure \ref{fig:split} for illustrations of zero-models and the narrow-scope $\FC$-prediction in $\BSML$.\footnote{The state $\{w_q\}$ in Figure \ref{fig:zero_model} supports  $p \vee q$. This is because we can find substates of $\{w_q\}$ supporting each disjunct: $\{w_q\}$ itself supports $q$, and the empty state vacuously supports $p$. But $\{w_q\}$ is a zero-model for $p \vee q$ because no \emph{nonempty} substate supports the first disjunct. Pragmatic enrichment has the effect of ruling out this zero-model---$\{w_q\}$ does not support $[p \vee q]^+$. The state $\{w_p,w_q\}$ in \ref{fig:non-zero} is a non-zero verifier for $p\vee q$, so it supports $[p \vee q]^+$. The state $\{w_\emptyset\}$ in \ref{fig:fc} is a zero-model for $\Di(p\lor q)$, and it does not support $\Di p\land \Di q$. The state $\{w_{pq}\}$, on the other hand, is a non-zero verifier for $\Di(p \lor q)$. We will see that such a verifier must also verify $\Di p\land \Di q$---in other words, narrow-scope $\FC$ is predicted by the fact that $[ \Di (p\lor q)]^+\vDash \Di p\land \Di q$. See Section \ref{section:preliminaries} for the definition of the pragmatic enrichment function and more details.} See \cite{aloni2022} for more discussion and for comparisons with other accounts of $\FC$.

\ex. {\bf Cancellation of $\FC$ under negation} \label{ex:negation}  \hfill  \cite{alonso2006} \a.  You are not allowed to eat the cake or the  ice-cream. \\   
   $\implicates$   You are not allowed to eat   either one. \b.   $\neg \Diamond  (\alpha \vee \beta)$ $\implicates$   $ \neg  \Diamond  \alpha \wedge \neg \Diamond  \beta $

{ 
  \begin{figure}[t]  
\centering
  \subfigure[$\{w_q\}\nvDash \lbrack p \vee q \rbrack^{+}$; zero-model for $p\lor q$]{\hspace{0.1cm}
\begin{tikzpicture}[>=latex,scale=.85]
\label{fig:zero_model}
 \draw[opaque,rounded corners,fill=gray!10] (0.1,1.4) rectangle (1.4, .1);

 % Frame
 \draw[draw=gray] (-1.5,1.5) rectangle (1.5, -1.5);

 % Indices
\draw (-0.75,.75) node[index gray, scale=0.9] (yy) {$w_{p}$};
 \draw (0.75,.75) node[index gray, scale=0.9] (yn) {$w_q$};
 \draw (-0.75,-0.75) node[index gray, scale=0.9] (ny) {$w_{pq}$};
 \draw (0.75,-0.75) node[index gray, scale=0.9] (nn) {$w_{\emptyset}$};

\end{tikzpicture}\hspace{0.1cm} }
\hspace{0.2cm}
 \subfigure[\footnotesize $\{w_p,w_q\}\vDash \lbrack p \vee q \rbrack^{+}$; non-zero verifier for $p\lor q$]{ 
 \hspace{0.1cm}
 
 \begin{tikzpicture}[>=latex,scale=.85]
 \label{fig:non-zero}

 % Possibilities
 \draw[opaque,rounded corners,fill=gray!10] (-1.4,1.4) rectangle (1.4, .1);
  
 % Frame
 \draw[draw=gray] (-1.5,1.5) rectangle (1.5, -1.5);

 % Indices
\draw (-0.75,.75) node[index gray, scale=0.9] (yy) {$w_{p}$};
 \draw (0.75,.75) node[index gray, scale=0.9] (yn) {$w_q$};
 \draw (-0.75,-0.75) node[index gray, scale=0.9] (ny) {$w_{pq}$};
 \draw (0.75,-0.75) node[index gray, scale=0.9] (nn) {$w_{\emptyset}$};

\end{tikzpicture}
\hspace{0.1cm}
}
\hspace{0.2cm}
 \subfigure[\footnotesize $\{w_\emptyset\}\nvDash \lbrack \Di( p \vee q) \rbrack^{+}$; $\{w_{pq}\}\vDash \lbrack \Di( p \vee q) \rbrack^+$]{  \hspace{0.1cm}
 \begin{tikzpicture}[>=latex,scale=.85]
 \label{fig:fc}

 \draw[opaque,rounded corners,fill=gray!10] (0.1,-0.1) rectangle (1.4, -1.4);
 \draw[opaque,rounded corners,fill=gray!10] (-1.4,-0.1) rectangle (-0.1, -1.4);
 
 % Frame
 \draw[draw=gray] (-1.5,1.5) rectangle (1.5, -1.5);

 % Indices
\draw (-0.75,.75) node[index gray, scale=0.9] (yy) {$w_{p}$};
 \draw (0.75,.75) node[index gray, scale=0.9] (yn) {$w_q$};
 \draw (-0.75,-0.75) node[index gray, scale=0.9] (ny) {$w_{pq}$};
 \draw (0.75,-0.75) node[index gray, scale=0.9] (nn) {$w_{\emptyset}$};

\path[->] (nn) edge[](yn);
\path[->] (ny) edge[](yn);
\path[->] (ny) edge[](yy);

\end{tikzpicture}\hspace{0.1cm}
 }
\caption{Zero-models and $\FC$ in $\BSML$. Worlds are labelled according to what is true in them: $w_p$ stands for a world where only $p$ is true, $w_{pq}$ for one in which only $p$ and $q$ are, etc. The accessibility relation is depicted using arrows.} \label{fig:split}
\end{figure}
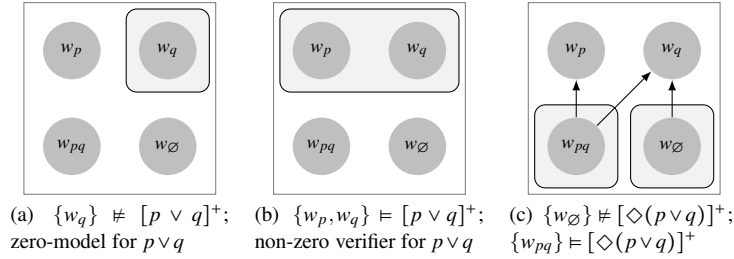  
}

$\BSML$ is closely related to modal logics in the lineage of dependence logic \cite{vaananen2007}, such as modal dependence \cite{vaananen2008,hella2014,yang20172}, inclusion \cite{kontinen2014,hella2015,anttila2023axiomatizingmodalinclusionlogic}, independence \cite{kontinen20142}, and team \cite{muller2014,kontinen2014,luck2020} logic. It is also similar to modal inquisitive logic \cite{ciardelli2016,ciardellibook}. However, the specific combination of connectives ($\NE,\lor,\Di,\lnot$) on which the account of $\FC$ in \cite{aloni2022} depends is unique to $\BSML$. Let us briefly discuss each of these in turn.

The most crucial component in the account of $\FC$ is the nonemptiness atom $\NE$. This atom also sets $\BSML$ apart from dependence and inquisitive logic in that formulas containing $\NE$ might not be \emph{downward closed} and might not have the \emph{empty state} (or empty team) property. Logics which violate these properties are commonly known as \emph{team logics} (see, for example, \cite{vaananen2007,yang2017,luck2020}). Most team logics incorporate the \emph{Boolean negation} $\sim$ ($s\vDash\sim \phi$ iff $s\nvDash \phi$). This negation is very strong in terms of expressive power: the extension of classical modal logic $\ML$ with $\sim$ (known as \emph{modal team logic} ($\MTL$)) can express all modally definable state properties \cite{kontinen2014}. $\BSML$, being essentially $\ML$ together with $\NE$, is a more modest extension. The propositional fragment of $\BSML$ (that is, classical propositional logic extended with $\NE$) was studied already in \cite{yang2017}.

The disjunction $\dis$ in $\BSML$ is the standard disjunction in the dependence logic lineage, commonly referred to as the \emph{tensor disjunction} or \emph{local disjunction}. Inquisitive logic makes use of a different disjunction, the \emph{global disjunction} $\intd$ (also known as the \emph{inquisitive disjunction}).
\begin{align*}
    &s\vDash \phi \lor \psi &&\text{iff} &&\text{there are }t,u\subseteq s \text{ s.t. }s=t\cup u \text{ and }t\vDash\phi \text{ and }u\vDash \psi  \\
    &s\vDash \phi \intd \psi &&\text{iff} &&s\vDash \phi \text{ or }s\vDash \psi
\end{align*}
The predictions in \cite{aloni2022} rely on the use of $\dis$, but it is interesting to note that alternative accounts of $\FC$ have been given in inquisitive semantics, making use of $\intd$ instead of $\dis$---see, for example, \cite{aloni2007,AloniCiardelli2013,nygren}, and see \cite{aloni2022} for a comparison.

The modalities $\Di$ and $\Bo$ of $\BSML$ (essential for the $\FC$-predictions) are equivalent to ones considered in an early version of modal inquisitive logic \cite{ciardelli2016}, and are closely related to the modalities in possibility semantics \cite{humberstone1981}. They are distinct from the standard modalities of modal team semantics---we refer the reader to \cite{anttila2021} for a comparison.

For its predictions involving negation (such as \ref{ex:negation}), $\BSML$ relies on a \emph{bilateral} semantics: in addition to support/assertion conditions, each formula is also provided with \emph{anti-support}/rejection conditions. The semantics of the \emph{bilateral negation} $\lnot$ are then defined using the anti-support conditions.\footnote{Bilateralism is typically associated with inferentialism in the proof-theoretic tradition \cite{price1983,smiley1996,rumfitt2000,restall2005,wansing2010,incurvati2017,wansingayhan23}; see \cite{sep-proof-theoretic-semantics} for an overview. Other systems employing bilateral semantics include first-degree entailment logic (FDE) \cite{anderson1975,dunn1976,belnap1977,belnap19772} and truthmaker semantics \cite{vfraassen1969,fine2017}. Recently, \cite{aher2011} and \cite{willer2018} presented bilateral accounts of $\FC$ in an inquisitive and a dynamic setting, respectively. The bilateral negation in $\BSML$ is essentially the same notion as the \emph{dual} or \emph{game-theoretic negation} of \emph{independence-friendly logic} \cite{hintikka1989,hintikka1996} and some formulations of dependence logic \cite{vaananen2007,vaananen2008}.}

In this paper, we study the logical properties of \BSML, as well as two new extensions of $\BSML$: $\BSML$ with the global disjunction $\intd$ ($\BSMLI$), and $\BSML$ with the novel \emph{emptiness operator} $\OC$ ($\BSMLO$). The emptiness operator is a natural counterpart to $\NE$ which can be used to cancel out its effects: a state $s$ supports $\OC\phi$ iff $s$ supports $\phi$ or $s$ is empty ($s\vDash \OC\phi$ iff $s\vDash \phi$ or $s =\emptyset$), meaning that, for instance, $\OC\NE$ is always supported. While these extensions may prove to have interesting applications of their own (we return to this in the conclusion), our introduction of them is motivated primarily by technical considerations: as we show in the first part of the paper, each extension is expressively complete for a natural class of state properties. Moreover, in the second part of the paper we axiomatize each of the three logics, and the expressive completeness of the extensions plays a crucial role in our proof of the completeness of these axiomatizations.

In the first part, then, we focus on expressive power. We characterize the expressiveness of our logics in terms of the notion of \emph{state (or team) bisimulation} as introduced in the literature on modal team semantics \cite{hella2014,kontinen2014} (cf. \emph{inquisitive bisimulation} \cite{ciardelli2021}). We show that $\BSMLI$ is expressively complete for the class of all state properties invariant under bounded bisimulation (meaning that $\BSMLI$ and $\MTL$ are equivalent in expressive power \cite{kontinen2014}), and that $\BSMLO$ is expressively complete for the class of all \emph{union-closed} state properties invariant under bounded bisimulation. These results build on similar results in, for instance, \cite{sevenster2009,hella2014,kontinen2014,hella2015}. $\BSML$, as we will demonstrate, is union closed but cannot express all union-closed properties.

In the second part, we develop natural deduction axiomatizations for each of the three logics. These systems build on systems presented in \cite{yang2017} for logics which are essentially the propositional fragments of $\BSML$ and $\BSMLI$. For other similar natural deduction axiomatizations for modal state-based logics, see, for instance, \cite{yang20172,anttila2023axiomatizingmodalinclusionlogic}. 

The paper is structured as follows. In Section \ref{section:preliminaries}, we define the syntax and semantics of $\BSML$ and the extensions, and discuss some basic properties of these logics. In Section \ref{section:expressive_power}, we show the expressive power results described above. We also derive some simple consequences of these results such as the finite model property (for each of the logics) and uniform interpolation (for $\BSMLI$ and $\BSMLO$). In Section \ref{section:axiomatizations}, we present our natural deduction axiomatizations, and in Section \ref{section:conclusion} we conclude by noting some open problems. The paper is partly based on the MSc thesis \cite{anttila2021} of the second author (supervised by the first author and the third author), which contains preliminary versions of some of the results.

%% file: sections/preliminaries.tex
\section{Bilateral State-based Modal Logic} \label{section:preliminaries}

In this section we define the syntax and semantics of $\BSML$ and its extensions, discuss the basic properties of these logics, and recall standard notions and results from team semantics.

\begin{definition}[Syntax] \label{def:syntax}
Fix a (countably infinite) set $\mathsf{Prop}$ of propositional variables. The set of formulas of \emph{bilateral state-based modal logic} $\BSML$ is generated by:
        \begin{align*}
            \phi ::= p \sepp \bot \sepp \bnot \phi \sepp ( \phi \land \phi )  \sepp ( \phi \dis \phi ) \sepp \Di \phi   \sepp \NE
        \end{align*}
        where $p\in \mathsf{Prop}$.   
        \emph{Classical modal logic} $\ML$ is the $\NE$-free fragment of $\BSML$.
        
        $\BSMLI$ is $\BSML$ extended with the \emph{global disjunction} $\intd$: the set of formulas of $\BSMLI$ is generated by the definition for $\BSML$ augmented with the case $\phi \intd \phi$.
        
        $\BSMLO$ is $\BSML$ extended with the \emph{emptiness operator} $\OC$: the set of formulas of $\BSMLO$ is generated by the definition for $\BSML$ augmented with the case $\OC\phi$.
\end{definition}

Throughout the paper, we use the first Greek letters $\alpha,\beta,\gamma,\dots$ to stand for formulas of \ML (also called {\em classical formulas}).
We write $\mathsf{P}(\phi)$ for the set of propositional variables in $\phi$ and $\phi (p_{1},\ldots ,p_{n})$ if $\mathsf{P}(\phi)\subseteq \{p_{1},\ldots ,p_{n}\}$. We write $\phi (\psi/p)$ for the result of replacing all occurrences of $p$ in $\phi$ by $\psi$.

\begin{definition} \label{def:models}
A \emph{(Kripke) model} (over $\mathsf{X}\subseteq \mathsf{Prop}$) is a triple $M = (W,R,V)$, where
\begin{itemize}
    \item[--] $W$ is a nonempty set, whose elements are called \emph{(possible) worlds};
    \item[--] $R \subseteq W \times W$ is a binary relation, called the \emph{accessibility relation};
    \item[--] 
    $V: \mathsf{X} \to \wp(W)$
    is a function, called the \emph{valuation}.
\end{itemize}  

We call a subset $s\subseteq W$ of $W$ a \emph{state} on $M$.
\end{definition} 
For any world $w$ in $M$, define, as usual, $R[w]:=\{v \in W\sepp wRv\}$. Similarly, for any state $s$ on $M$, define $R[s]:=\bigcup_{w\in s}R[w]$.

In the standard world-based Kripke semantics for modal logic, formulas are evaluated with respect to worlds: one writes $M,w\vDash \phi$ if $\phi$ is true in the world $w$ in the model $M$. In our state-based semantics, formulas are instead evaluated with respect to states. We will also make use of two fundamental semantic notions---\emph{support} and \emph{anti-support}---rather than one. As noted in Section \ref{section:introduction}, states in $\BSML$ represent speakers' information states. Support of a formula by (or in) a state represents that what the formula expresses is assertible given the information in the state; anti-support, similarly, represents rejectability in a state.

\begin{definition}[Semantics] \label{def:semantics}
    For a model $M=(W,R,V)$ over $\mathsf{X}$, a
    state $s$ on $M$, and formula $\phi$ with $\mathsf{P}(\phi)\subseteq \mathsf{X}$, the notions of $\phi$ being \emph{supported}/\emph{anti-supported} by $s$ in $M$, written $M,s \vDash \phi$/$M,s\Dashv \phi$ (or simply $s \vDash \phi$/$s\Dashv \phi$), are defined recursively as follows:
    \begin{center}
    \begin{tabular}{p{2cm} p{1cm} p{8.5cm}}
        $M,s \vDash   p$  & $:\iff$&  for all $ w \in s: w \in V(p)$ \\
        $M,s \Dashv  p  $& $:\iff$& for all $ w \in s: w \notin V(p)$\\
        &&\\
        $M,s \vDash   \bot$  & $:\iff$&  $s=\emptyset$ \\
        $M,s \Dashv  \bot  $& & always the case\\
        &&\\
                $M,s \vDash    \NE$ &$ :\iff$&$ s \neq \emptyset$ \\
        $M,s \Dashv    \NE $&$ :\iff$&$ s =\emptyset$\\
        &&\\
        $M,s \vDash    \bnot \phi$ & $:\iff$ & $M,s\Dashv \phi$ \\
        $M,s \Dashv    \bnot \phi $&$ :\iff $&$ M,s \vDash \phi $\\
        &&\\
        $M,s \vDash    \phi \land \psi  $&$ :\iff $&$ M,s \vDash  \phi\text{ and } M,s \vDash \psi$\\
        $M,s \Dashv    \phi \land \psi  $& $:\iff$& there exist $t,u$ s.t. $ s=t \cup u\text{ and } M,t \Dashv  \phi\text{ and } M,u \Dashv  \psi$ \\
        &&\\
        $M,s \vDash    \phi \dis \psi  $& $:\iff$& there exist $t,u $ s.t. $ s=t\cup u\text{ and } M,t \vDash  \phi\text{ and } M,u \vDash  \psi$ \\
        $M,s \Dashv    \phi \dis \psi  $&$ :\iff $&$ M,s \Dashv  \phi\text{ and } M,s \Dashv  \psi$\\
        &&\\
                $M,s \vDash    \phi \intd \psi  $&$ :\iff$&$ M,s \vDash  \phi \text{ or } M,s \vDash  \psi$\\
        $M,s \Dashv    \phi \intd \psi  $&$ :\iff$&$  M,s \Dashv  \phi\text{ and } M,s \Dashv  \psi$\\
        &&\\
        $M,s \vDash    \Di \phi  $& $:\iff$& for all $ w \in s$ there exists $  t \subseteq R[w]$ s.t. $ t\neq \emptyset \text{ and }M,t \vDash \phi$ \\
        $M,s \Dashv    \Di \phi  $&$ :\iff$& for all $ w \in s : M,R[w]\Dashv \phi$\\
        &&\\
        $M,s \vDash    \OC \phi  $& $:\iff$& $M,s\vDash \phi \text{ or }s=\emptyset$ \\
        $M,s \Dashv    \OC \phi  $& $:\iff$& $M,s\Dashv \phi $
    \end{tabular}
    \end{center}
\end{definition}   
For a set $\Phi$ of formulas, we write $M,s\vDash\Phi$ if $M,s \vDash \phi$ for all $\phi \in \Phi$. We say that $\Phi$ \emph{entails} $\psi$, written $\Phi \vDash \psi$, if for all models $M$ and states $s$ on $M$:  $M,s \vDash \Phi$  implies $M,s\vDash \psi$. We also write simply $\phi_1,\dots,\phi_n\vDash\psi$ for $\{\phi_1,\dots,\phi_n\}\vDash\psi$, and $\Phi,\phi\vDash \psi$ for $\Phi\cup\{\phi\}\vDash \psi$. We write $\vDash\phi$ for $\emptyset \vDash \phi$ and in this case say that $\phi$ is \emph{valid}. If both $\phi\vDash\psi$ and $\psi\vDash\phi$, then we write $\phi\equiv\psi$ and say that $\phi$ and $\psi$ are \emph{equivalent}. If both $\phi \equiv \psi$ and $\lnot \phi \equiv \lnot \psi$, then $\phi$ and $\psi$ are said to be \emph{strongly equivalent}, written $\phi \equivclosed \psi$.

The box modality $\Bo$ is defined as the dual of the diamond: $\Box \phi:=\lnot \Di\lnot \phi$; the resulting support/antisupport clauses are:
\begin{center}
    \begin{tabular}{p{2cm} p{1cm} p{8.5cm}}
        $M,s \vDash    \Bo \phi  $&$ \iff$&for all $ w \in s:  M,R[w]\vDash \phi$\\
        $M,s \Dashv    \Bo \phi  $& $\iff$&for all $ w \in s$ there exists $ t \subseteq R[w]$ s.t. $ t\neq \emptyset \text{ and }M,t \Dashv \phi$ \\
        &&\\
    \end{tabular}
\end{center}
    
We refer to the atom $\bot$ as the \emph{weak contradiction}. The original syntax for \BSML in \cite{aloni2022} does not include $\bot$, but rather defines it as $\bot:=p\land \lnot p$ for some fixed $p$. Including $\bot$ in our syntax allows us to simplify parts of our exposition. We also define the \emph{strong contradiction} $\Bot:=\bot \land \NE$, and the \emph{strong tautology} $\Top:=\lnot \bot$. The weak contradiction is supported only in the empty state, the strong contradiction in no state, and the strong tautology in all states. The atom $\NE$ (supported in all nonempty states) can  also be viewed as the \emph{weak tautology} $\top$; accordingly, we let $\top:=\NE$. Note that we have the following equivalences:
\begin{align*}
    &\lnot \bot \equiv \Top &&\lnot \Top \equiv \bot && \lnot \Bot \equiv \Top &&\lnot \top \equiv \bot 
\end{align*}
We use these contradictions and tautologies to interpret the empty disjunctions and conjunction:
\begin{align*}
    &\bigdis \emptyset := \bot &&\bigintdd \hspace{-2pt} \emptyset := \Bot && \bigwedge \emptyset := \Top
\end{align*}

We now give some examples to illustrate the semantics---consider Figure \ref{fig:examples}. A (local) disjunction $\phi \lor \psi$ is supported in $s$ if $s$ can be split into two states $t$ and $u$ such that $t\vDash \phi$ and $u\vDash \psi$. Note that one or both of these substates might be empty; for instance, in \ref{fig:examples_a} we have $\{w_q\}\vDash p\lor q$. One can force these substates to be nonempty using $\NE$: we have $\{w_q\}\nvDash (p\land \NE) \lor (q\land \NE)$, whereas $\{w_p,w_q\}\vDash (p\land \NE) \lor (q\land \NE)$. A global disjunction, on the other hand, is supported if either disjunct is supported: $\{w_q\}\vDash  (p\land \NE) \intd (q\land \NE)$, whereas $\{w_p,w_q\}\nvDash  (p\land \NE) \intd (q\land \NE)$. The emptiness operator $\OC$ is essentially a restricted variant of $\intd$ in that $\OC\phi \equiv \phi \intd \bot$: $\{w_q\}\vDash \OC (p\land \NE)\lor \OC (q\land \NE)$ and $\{w_p,w_q\}\vDash \OC (p\land \NE)\lor \OC (q\land \NE)$. A diamond formula $\Di \phi$ is supported in $s$ if for each world $w$ in $s$ there is a nonempty substate $t$ of the state of all worlds $R[w]$ accessible from $w$ that supports $\phi$; for example, in \ref{fig:examples_b} we have $s_b\vDash \Di q$ and $s_b\nvDash \Di p$. A box formula $\Bo\phi$ is supported in $s$ if for each $w\in s$, $R[w]$ as a whole supports $\phi$: $s_b\nvDash \Bo q$ and $s_b\vDash \Bo p \lor \Bo q$. It is easy to verify that the conjunction and (local) disjunction distribute over the global disjunction: $\phi\wedge(\psi\intd\chi)\equiv(\phi\wedge\psi)\intd(\phi\wedge\chi)$ and $\phi\vee(\psi\intd\chi)\equiv(\phi\vee\psi)\intd(\phi\vee\chi)$; whereas the modalities do not distribute over $\intd$, as, for example, $s_b \vDash \Bo (p \intd q)$ but $s_b \nvDash \Bo p \intd \Bo q$. We have instead $\Di (\phi \intd \psi)\equiv \Di \phi \lor \Di \psi$ and $\Bo (\phi \intd \psi)\equiv \Bo \phi \lor \Bo \psi$.

    { 
  \begin{figure}[t]  
\centering
  \subfigure[$\{w_q\}\vDash p\lor q$; $\{w_{p},w_q\}\vDash (p\land \NE)\lor (q\land \NE)$]{
\begin{tikzpicture}[>=latex,scale=.85]
\label{fig:examples_a}
 % Possibilities

 \draw[opaque,rounded corners,fill=gray!10] (-1.4,1.4) rectangle (1.4, .1);
 \draw[opaque,rounded corners,fill=gray!10] (0.2,1.3) rectangle (1.3, .2);

 % Frame
 \draw[draw=gray] (-1.5,1.5) rectangle (1.5, -1.5);

 % Indices
\draw (-0.75,.75) node[index gray, scale=0.9] (yy) {$w_{p}$};
 \draw (0.75,.75) node[index gray, scale=0.9] (yn) {$w_q$};
 \draw (-0.75,-0.75) node[index gray, scale=0.9] (ny) {$w_{pq}$};
 \draw (0.75,-0.75) node[index gray, scale=0.9] (nn) {$w_{\emptyset}$};

\end{tikzpicture} }
\hspace{0.9cm}
 \subfigure[\footnotesize $s_b\vDash \Di q$]{ 
 \begin{tikzpicture}[>=latex,scale=.85]
 \label{fig:examples_b}

 % Possibilities

 \draw[opaque,rounded corners,fill=gray!10] (0.-1.4,-0.1) rectangle (1.4, -1.4);
 
 % Frame
 \draw[draw=gray] (-1.5,1.5) rectangle (1.5, -1.5);

 % Indices
\draw (-0.75,.75) node[index gray, scale=0.9] (yy) {$w_{p}$};
 \draw (0.75,.75) node[index gray, scale=0.9] (yn) {$w_q$};
 \draw (-0.75,-0.75) node[index gray, scale=0.9] (ny) {$w_{pq}$};
 \draw (0.75,-0.75) node[index gray, scale=0.9] (nn) {$w_{\emptyset}$};

\path[->] (nn) edge[](yn);
\path[->] (ny) edge[](yy);
\path[->,loop,looseness=4,in=30,out=330] (ny) edge[] (ny); 

\end{tikzpicture}
}
\hspace{0.9cm}
 \subfigure[\footnotesize $s_c\vDash \lbrack \Di( p \vee q) \rbrack^{+}$]{ 
 \begin{tikzpicture}[>=latex,scale=.85]
 \label{fig:fcc}

 \draw[opaque,rounded corners,fill=gray!10] (-1.4,-0.1) rectangle (-0.1, -1.4);
 
 % Frame
 \draw[draw=gray] (-1.5,1.5) rectangle (1.5, -1.5);

 % Indices
\draw (-0.75,.75) node[index gray, scale=0.9] (yy) {$w_{p}$};
 \draw (0.75,.75) node[index gray, scale=0.9] (yn) {$w_q$};
 \draw (-0.75,-0.75) node[index gray, scale=0.9] (ny) {$w_{pq}$};
 \draw (0.75,-0.75) node[index gray, scale=0.9] (nn) {$w_{\emptyset}$};

\path[->] (ny) edge[](yn);
\path[->] (ny) edge[](yy);

\end{tikzpicture}
}
\caption{Examples of the semantics} \label{fig:examples}
\end{figure}
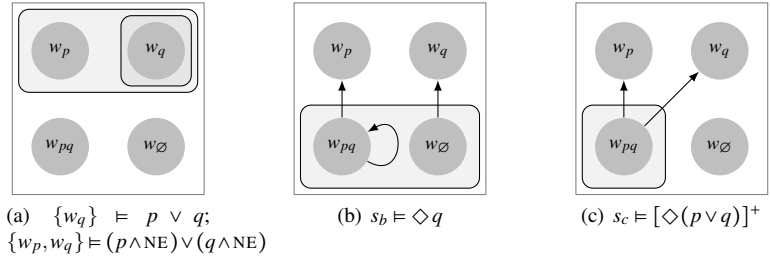  
}

Let us also revisit the pragmatic enrichment function $[\ ]^+:\ML\to \BSML$ described in Section \ref{section:introduction}, which can now be formally and recursively defined as:
\begin{align*}
    &[p]^+ &&:= &&p \land \NE&&\\
    &[\medcircle \alpha]^+ &&:= && \medcircle([\alpha]^+) \land \NE&&\text{for }\medcircle\in \{\lnot,\Di,\Bo\}\\
    &[\alpha \triangle \beta]^+ &&:= && ([\alpha]^+ \triangle [\beta]^+) \land \NE&&\text{for }\triangle\in \{\land,\lor\}
\end{align*}
It is easy to see that $[\ ]^+$ disallows zero-models as described in Section \ref{section:introduction}. Consider, for instance, $[p\lor q]^+=((p\land \NE)\lor (q\land \NE)) \land \NE\equiv (p\land \NE)\lor (q\land \NE)$. The state $\{w_q\}$ is a zero-model for $p\lor q$ since $\{w_q\}\vDash p\lor q$ but $\{w_q\}\nvDash (p\land \NE)\lor (q\land \NE)$, whereas $\{w_p,w_q\}$ is a non-zero verifier since $\{w_p,w_q\}\vDash (p\land \NE)\lor (q\land \NE)$. For a formula $\Di(\alpha \lor \beta)$ as in the antecedent of a narrow-scope $\FC$-inference, we have $[\Di(\alpha \lor \beta)]^+\equiv \Di ((\alpha \land \NE) \lor (\beta \land \NE))\land \NE$. It is then also easy to see that $[\Di(\alpha \lor \beta)]^+\vDash \Di \alpha \land \Di \beta$---a non-zero verifier for $\Di (\alpha \lor \beta)$ must verify $\Di \alpha \land \Di \beta$. In Figure \ref{fig:examples}, $s_b$ is a zero-model for $\Di(p\lor q)$ and $s_c$ is a non-zero verifier: we have $s_b\vDash p\lor q$, $s_b\nvDash \Di((p\land \NE)\lor (q\land \NE))$, and $s_b\nvDash \Di p\land \Di q$; whereas $s_c\vDash \Di((p\land \NE)\lor (q\land \NE))$ and $s_c\vDash \Di p\land \Di q$.

Similarly to many other logics employing team semantics (see, for example, \cite{yang2017,CiardelliRoelofsen2011}), $\BSML$ and its extensions do not admit uniform substitution: $\phi(p)\vDash \psi(p)$ need not imply $\phi(\chi/p)\vDash \psi(\chi/p)$. For instance, we have $\vDash p\lor \lnot p$ and $p\lor p\vDash p$, whereas $\nvDash \NE\lor \lnot \NE$ and $(p\intd \lnot p)\lor (p\intd \lnot p)\nvDash p\intd \lnot p$.  

Closely related is the failure of preservation of equivalence under replacement: $\phi\equiv \psi$ need not imply $\theta(\phi/p)\equiv\theta(\psi/p)$. For instance, we have $\lnot (\phi \lor \psi)\equiv \lnot (\phi \intd \psi)$ but $\lnot\lnot (\phi \lor \psi) \equiv\phi \lor \psi \notequiv  \phi \intd \psi \equiv\lnot \lnot (\phi \intd \psi)$, and $ \bot\equiv \lnot \NE$ but $ \lnot \bot\equiv \Top \notequiv \NE \equiv \lnot \lnot\NE$. This failure is occasioned by the bilateral negation: it is easy to show that if $p$ is not within the scope of $\lnot$, it does follow that $\theta(\phi/p)\equiv\theta(\psi/p)$.
%\footnote{Another way of formulating the failure of replacement under negation is as follows: the property $||\phi||=\{(M,s)\sepp M,s\vDash\phi \}$ defined by $\phi$ (see Section \ref{section:expressive_power}) fails to determine $||\lnot \phi||$ since $||\phi||=||\psi||$ need not imply $||\lnot \phi||=||\lnot\psi||$. Adapting similar results for Henkin sentences and their contraries \cite{burgess2003} and for dependence logic with the dual negation \cite{kontinen2011}, it is possible to show \cite{anttila2024} that this lack of determination is extreme in the sense that for each of our logics, the following are equivalent:
%  (i) If $M,s\vDash \phi$ and $M,t\vDash \psi$, then $s\cap t=\emptyset$. \hspace{0.5cm} (ii) There is a $\theta$ s.t. $\phi \equiv \theta$ and $\psi \equiv \lnot \theta$.
%In other words, not only does $||\phi||$ fail to determine $||\lnot \phi||$, but, moreover, given only $||\phi||$, all we can know about $||\lnot \phi||$ is that it is some property $\mathcal{P}$ definable in the logic, and that $(M,s)\in \mathcal{P} \cap ||\phi||$ implies $s=\emptyset$.}
Strong equivalence, on the other hand, \emph{is} preserved under replacement, as we prove in the following proposition (cf. a similar replacement theorem in \cite[Lemma 3.25]{vaananen2007}). We write $\phi[\chi]$ (and sometimes simply $[\chi]$ or $\chi$) to refer to a specific occurrence of the subformula $\chi$ in $\phi$, and $\phi[\psi/\chi]$ for the result of replacing this occurrence of $\chi$ in $\phi$ with $\psi$. These notations are also extended to sequences of occurrences $\phi[\chi_1,\ldots ,\chi_n]$ and sequences or replacements $\phi[\psi_1/\chi_1,\ldots ,\psi_n/\chi_n]$ in the obvious way.

\begin{proposition}[Replacement w.r.t. strong equivalence]\label{Replacement_thm_semantic}
For any formulas $\theta,\phi$, and $\psi$, if $\phi \equivclosed \psi$, then $ \theta[\phi/p] \equivclosed \theta[\psi/p]$.
\end{proposition}
\begin{proof}
By induction on $\theta$. We only give a detailed proof for a few  interesting cases.
        
        $\theta=\lnot \eta$.  By  the induction hypothesis, $\eta[\phi/p]\equivclosed \eta[\psi/p]$, so $\lnot \eta[\phi/p]\equiv \lnot \eta[\psi/p]$ and $\lnot \lnot \eta[\phi/p]\equiv \eta[\phi/p]\equiv \eta[\psi/p] \equiv \lnot \lnot \eta[\psi/p]$. Therefore, $\lnot \eta[\phi/p
        ]\equivclosed \lnot \eta[\psi/p]$.
        
        $\theta=\Di\eta$. By the induction hypothesis, we have $\eta[\phi/p]\equivclosed\eta[\psi/p]$. To show $\Di \eta[\phi/p]\vDash \Di\eta[\psi/p]$, assume $s\vDash \Di \eta[\phi/p]$. Then for all $w\in s$ there is a nonempty $ t\subseteq R[w]$ such that $t\vDash \eta[\phi/p]$. But then also for each such $t$ we have $t\vDash \eta[\psi/p]$, so $s\vDash \Di \eta[\psi/p]$. The proof of $\Di \eta[\psi/p]\vDash \Di\eta[\phi/p]$ is analogous. To show $\lnot \Di \eta[\phi/p]\vDash \lnot \Di\eta[\psi/p]$, assume $s\vDash \lnot \Di \eta[\phi/p]$. Then $R[w]\vDash \lnot \eta[\phi/p]$ for all $w\in s$. But then also $ R[w]\vDash \lnot \eta[\psi/p]$ for all $w\in s$, so $s\vDash \lnot \Di \eta[\psi/p]$. The proof of $\lnot \Di \eta[\psi/p]\vDash \lnot \Di\eta[\phi/p]$ is analogous.
        
        $\theta=\OC\eta$.  By the induction hypothesis, we have $\eta[\phi/p]\equivclosed\eta[\psi/p]$. Then $\OC \eta[\phi/p]\equiv \eta[\phi/p]\intd \bot\equiv \eta[\psi/p] \intd \bot\equiv \OC\eta[\psi/p]$ and $\lnot \OC \eta[\phi/p]\equiv \lnot\eta[\phi/p]\equiv\\ \lnot\eta[\psi/p] \equiv \lnot\OC\eta[\psi/p]$.
        \end{proof}

        An easy induction shows that a formula $\phi$ and its negation $\lnot \phi$ are \emph{incompatible} in the sense that if $M,s\vDash \phi$ and $M,t\vDash \lnot \phi$, then $s\cap t =\emptyset$ (for a proof, see \cite[Proposition 3.3.9]{anttila2021}). It also appears \cite{anttila2024} that one can prove the converse of (a reformulation of) this fact: if $M,s\vDash\phi$ and $M,t\vDash\psi$ implies $s \cap t=\emptyset$, then there is a formula $\theta$ such that $\phi \equiv \theta$ and $\psi \equiv \lnot \theta$ (cf. the similar results for Henkin sentences and their contraries \cite{burgess2003} and for dependence logic with the dual negation \cite{kontinen2011}).
        
As usual, each formula can be transformed into an equivalent formula in \emph{negation normal form}, that is, a formula in each occurrence of negation directly precedes a propositional variable $p$ or the atom $\NE$, or is part of the defined connective $\Bo$.

\begin{fact}[Negation normal form] \label{fact:negation_normal_form}
     For any formula $\phi$, there exists a formula 
    $\phi'$ in negation normal form such that  $\phi \equiv \phi'$. Moreover, if $\phi \in \BSML$, then $\phi\equivclosed\phi'$.
\end{fact}
\begin{proof}
Follows easily from the following (strong) equivalences:
    \begin{align*}
    \lnot \lnot \phi&\equivclosed \phi    & \lnot (\phi\land \psi)&\equivclosed \lnot \phi \lor\lnot \psi &\lnot (\phi\intd \psi) &\equiv \lnot \phi \land\lnot \psi&\\
     \lnot \Di \phi&\equivclosed \Bo \lnot \phi&\lnot (\phi\lor \psi)&\equivclosed \lnot \phi \land\lnot \psi & \lnot \OC \phi&\equiv \lnot \phi&\tag*{\qedhere}
    \end{align*} 
\end{proof}
    
Throughout the paper, we make extensive use of the following properties:

\begin{definition} \label{def:closure_properties}
We say that a formula $\phi$ 
\begin{itemize}
    \item[--] is \emph{downward closed}, provided   $[M,s \vDash \phi \text{ and } t \subseteq s] \implies M,t \vDash \phi$;
     \item[--] is \emph{union closed}, provided $[M,s \vDash \phi \text{ for all } s\in S\neq \emptyset] \implies M, \bigcup S  \vDash \phi$;
    \item[--] has the \emph{empty state property}, provided $M, \emptyset \vDash \phi \text{ for all }M$;
     \item[--] is \emph{flat}, provided $M, s\vDash\phi  \iff M,\{w\}\vDash \phi\text{ for all }w\in s$.
\end{itemize}
\end{definition}
  It is easy to check that a formula is flat if and only if it is downwards closed, union closed, and has the empty state property. A simple induction shows:
  
\begin{fact} \label{fact:NE_intd_closure_properties}
    If a formula $\phi$ does not contain $\NE$, then it is downward closed and has the empty state property. If $\phi$ does not contain $\intd$, then it is union closed. 
    
    In particular, formulas of \BSML and \BSMLO are union closed; formulas of \ML are flat (as they are downward  and union closed, and have the empty state property).
\end{fact}
It is also easy to verify that $\Di\phi$ and $\Bo \phi$ are always flat. It follows from the flatness of \ML-formulas that the state-based semantics for $\ML$ correspond with the standard (single-world) Kripke semantics in the sense of the following fact:

\begin{fact} \label{fact:ML_team_classical_correspondence}
For any formula $\alpha\in \ML$:
\begin{align*}
    &M,s \vDash \alpha &&\iff &&M,\{w\}\vDash\alpha\text{ for all }w\in s&&\iff &&M,w \vDash \alpha \text{ for all }w \in s.
\end{align*}
\end{fact}
\begin{proof}
    The first equivalence follows from flatness, and the second from $\{w\}\vDash \alpha$ iff $w\vDash \alpha$ (proved by an easy induction).
\end{proof}
As an immediate consequence of Fact \ref{fact:ML_team_classical_correspondence}, we have that for any set $\Delta \cup \{\alpha\}$ of $\ML$-formulas, $\Delta \vDash \alpha$ in the state-semantics sense iff $\Delta \vDash \alpha$ in the usual single-world-semantics sense. For this reason, when discussing $\ML$-formulas, we simply use the notation $\vDash$ to refer to both state-based and world-based entailment (similarly for $\equiv$).

%% file: sections/expressive_power.tex
\section{Expressive Power and Normal Forms} \label{section:expressive_power}

In this section, we study the expressive power of \BSML and its extensions. We measure the expressive power of the logics in terms of the \emph{properties}---sets of pointed models---they can express. It is well-known that the expressive power of classical modal logic $\ML$ is characterized by \emph{bisimulation-invariance}: the properties expressible in $\ML$ are invariant (closed) under bisimulation. This also holds in the state-based modal setting: our logics are invariant under \emph{state bisimulation}. In Section \ref{section:bisimulation}, we introduce state bisimulation and prove bisimulation-invariance; in Section \ref{section:expressive_completeness}, we apply this notion to show expressive completeness results for our logics.

\subsection{Bisimulation} \label{section:bisimulation}
We recall briefly some standard results concerning bisimulation for classical modal logic and standard Kripke semantics; for more comprehensive discussion, we refer the reader to, for instance, \cite{blackburn2001,goranko2007}.

Throughout this section, we fix a finite set $\mathsf{X}\subseteq\mathsf{Prop}$ of propositional variables. We omit mention of $\mathsf{X}$ whenever doing so does not result in confusion in order to keep our notation light. A \emph{pointed model} (over $\mathsf{X}$) is a pair $(M,w)$ where $M$ is a model over $\mathsf{X}$ and $w\in W$.   
\begin{definition}[Bisimilarity] \label{def:bisimilarity}
    For any $k \in \mathbb{N}$ and any $(M,w)$ (over $\mathsf{Y}\supseteq \mathsf{X}$) and $(M',w')$ (over $\mathsf{Y}'\supseteq \mathsf{X}$), we write $M,w \bisim^{\mathsf{X}}_k M',w'$ (or simply $w\bisim_k w'$) and say $(M,w)$ and $(M',w')$ are $\mathsf{X},k$-\emph{bisimilar} if the following recursively defined relation holds:
    \begin{itemize}
        \item[--] $M,w \bisim^{\mathsf{X}}_0 M,w :\iff $ for all $p \in \mathsf{X}$ we have $M,w \vDash p$ iff $M',w' \vDash p$.
        \item[--] $M,w \bisim^{\mathsf{X}}_{k+1} M',w' :\iff M,w \bisim_0 M',w'$ and
        \begin{itemize}
            \item {[forth]} for all $v \in R[w]$ there is a $v' \in R' [w']$ such that $M,v \bisim_k M',v'$;
            \item {[back]} for all $v' \in R'[w']$ there is a $v \in R [w]$ such that $M,v \bisim_k M',v'$.
        \end{itemize}
    \end{itemize}
\end{definition}

The \emph{modal depth} $md(\phi)$ of a formula $\phi$ is defined recursively as:
    \begin{itemize}
        \item[--] $md(p)=md(\NE)=0$
        \item[--] $md(\bnot \phi)=md(\OC\phi)=md(\phi)$
        \item[--] $md(\phi \land \psi)=md(\phi \dis \psi)=md(\phi \intd \psi)=max\{md(\phi),md(\psi)\}$
        \item[--] $md(\Di \phi)= md(\phi)+1$
    \end{itemize}

    We say that $(M,w)$ and $(M,w')$ are \emph{${\mathsf{X}},k$-equivalent} ($k\in \mathbb{N}$), written $M,w \equiv^{\mathsf{X}}_k M',w'$ (or simply $w\equiv_k w'$), if for all $\alpha({\mathsf{X}}) \in \ML$ with $md(\alpha)\leq k$: $M,w\vDash \alpha \iff  M',w'\vDash \alpha$.

\begin{definition}[Hintikka formulas] \label{def:hintikka}
    Let $k \in \N$ and let $(M,w)$ be a pointed model over $\mathsf{Y}\supseteq \mathsf{X}$. We define the \emph{$k$-th Hintikka formula} $\chi^{\mathsf{X},k}_{M,w} $ (or simply $\chi^k_{w} $) of $(M,w)$ recursively as:
    \begin{align*}
        &\chi^{\mathsf{X},0}_{M,w}   &&:= &&\bigwedge \{p \sepp p \in \mathsf{X} , w \in V(p)\} \land \bigwedge \{\lnot p \sepp p \in \mathsf{X} , w \notin V(p)\}\\
        &\chi^{{\mathsf{X}},k+1}_{M,w} &&:= &&\chi^{k}_{w}
            \land \underset{v \in R[w]}{\bigwedge}\Di \chi^{k}_{v}
            \land \Bo \underset{v \in R[w]}{\bigdis}\chi^{k}_{v}
    \end{align*}
\end{definition}

It is easy to see that there are only finitely many non-equivalent $k$-th Hintikka formulas for a given finite $\mathsf{X}$---this is why we may assume that the conjunction and the disjunction in $\chi^{\mathsf{X},k+1}_{M,w}$ are finite and hence that $\chi^{\mathsf{X},k+1}_{M,w}$ is well-defined.

\begin{theorem} (See, for instance, \cite{blackburn2001,goranko2007}.
)\label{theorem:hintikka_bisimulation} 
    \begin{align*}
        &w \equiv_k w'  &&\iff &&w \bisim_k w' &&\iff   &&w'\vDash \chi^{k}_{w} && \iff && \chi^k_w \equiv \chi^k_{w'}
    \end{align*}
\end{theorem}

In the context of standard Kripke semantics, we measure the expressive power of a logic in terms of the classes of pointed models expressible in the logic:

\begin{definition} \label{def:world_properties} 
A \emph{property} (over $\mathsf{X}$) is a class of pointed models over $\mathsf{X}$. Each formula $\alpha$ \emph{expresses} a property (over $\mathsf{X}\supseteq \mathsf{P}(\phi)$)
$$\llbracket \alpha \rrbracket_{\mathsf{X}}:=\{(M,w)\text{ over $\mathsf{X}$}\sepp M,w\vDash \alpha\}.$$
We say that a logic $L$ is \emph{expressively complete} for a class of properties $\mathbb{P}$, written $\llbracket L\rrbracket=\mathbb{P}$, if for each finite $\mathsf{X}$, the class $\mathbb{P}_\mathsf{X}$ of properties over $\mathsf{X}$ in $\mathbb{P}$ is precisely the class of properties over $\mathsf{X}$ expressible by formulas of $L$ with , that is, if
$$\mathbb{P}_\mathsf{X}=\llbracket L \rrbracket_\mathsf{X}:=\{\llbracket \alpha \rrbracket_\mathsf{X}\sepp \alpha \in L\}.$$
\end{definition}

A property $\PPP$ is \emph{invariant under} $\mathsf{X},k$-bisimulation if $(M,w)\in \PPP$ and $M,w\bisim^\mathsf{X}_k M',w'$ imply $(M',w')\in \PPP$. A property $\PPP$ over $\mathsf{X}$ is \emph{invariant under bounded bisimulation} if $\PPP$ is invariant under $\mathsf{X},k$-bisimulation for some $k\in \mathbb{N}$. Theorem \ref{theorem:hintikka_bisimulation} allows us to prove a world-based expressive completeness theorem for $\ML$: $\ML$ can express all world-based properties that are invariant under bounded bisimulation.

\begin{theorem}[World-based expressive completeness of $\MLBF$] \label{theorem:world-based_expressive_completeness} 
$$\llbracket \ML\rrbracket =\{\PPP\sepp \PPP \text{ is invariant under bounded bisimulation}\}.$$
\end{theorem}
\begin{proof}
    For the inclusion $\subseteq$, for any $\alpha\in \ML$, we have that $\llbracket\alpha \rrbracket$ is invariant under $k$-bisimulation for $k=md(\alpha)$ by Theorem \ref{theorem:hintikka_bisimulation}. For the converse inclusion $\supseteq$, for any $\PPP$ invariant under $k$-bisimulation, it follows easily from Theorem \ref{theorem:hintikka_bisimulation} that 
    $$M',w'\vDash \bigdis_{(M,w)\in \PPP} \chi^k_w \iff (M',w')\in \PPP \quad \text{(i.e., }\PPP=\llbracket \bigdis_{(M,w)\in \PPP} \chi^k_w\rrbracket\in \llbracket \ML \rrbracket).$$
    Note that since there are only finitely many non-equivalent Hintikka formulas $\chi^k_w$ for a given finite $\mathsf{X}$, we may assume that the disjunction in the above formula $\bigdis_{(M,w)\in \PPP} \chi^k_w$ is finite and hence that the formula is well-defined.
\end{proof}

The formula $\bigdis_{(M,w)\in \PPP} \chi^k_w$ in the above proof can be viewed as a characteristic formula for $\PPP$. The proof also yields a disjunctive normal form for formulas of $\ML$: given that $\llbracket\alpha\rrbracket$ is invariant under $k$-bisimulation for $k=md(\alpha)$, $\llbracket\alpha \rrbracket=\llbracket \bigdis_{(M,w)\in \llbracket\alpha\rrbracket} \chi^{md(\alpha)}_w\rrbracket$, that is, $\alpha \equiv \bigdis_{(M,w)\in \llbracket\alpha\rrbracket} \chi^{md(\alpha)}_w$.

We now introduce state-based analogues of the preceding notions and results. A \emph{pointed (state) model} (over $\mathsf{X}$) is a pair $(M,s)$ where $M$ is a model over $\mathsf{X}$ and $s$ is a state on $M$. For a given $k\in \mathbb{N}$, $(M,s)$ and $(M,s')$ are \emph{$\mathsf{X},k$-equivalent} (in a logic $L$), written $M,s \equiv^\mathsf{X}_k M',s'$, if for all $\phi(\mathsf{X})\in L$ with $md(\phi)\leq k: M,s\vDash \phi \iff M',s'\vDash \phi$.

State-based bisimulation (first introduced in \cite{hella2014,kontinen2014}) is a natural generalization of the world-based notion: two pointed state models are bisimilar if for each world in one there is a bisimilar world in the other, and vice versa.
    
\begin{definition}[State bisimilarity] \label{def:state_bisimilarity}
    For any $k \in \mathbb{N}$, $(M,s)$ (over $\mathsf{Y}\supseteq \mathsf{X}$) and $(M',s')$ (over $\mathsf{Y}'\supseteq \mathsf{X}$) are \emph{$\mathsf{X},k$-bisimilar}, written $M,s \bisim^\mathsf{X}_k M',s'$ (or simply $s\bisim_k s'$), if
    \begin{itemize}
        \item[--] {[forth]} for each $w \in s$ there is some $w' \in s'$ such that $M,w \bisim_k M',w'$;
        \item[--] {[back]} for each $w' \in s'$ there is some $w \in s$ such that $M,w \bisim_k M',w'$.
    \end{itemize}
\end{definition}
State bisimulation is clearly a conservative extension of world-based bisimulation in that $\{w\}\bisim_k\{w'\}$ iff $w\bisim_kw'$. We thus use the same symbol $\bisim_k$  for both world-based and state based bisimulation, and we sometimes write simply $w\bisim_ks'$ for $\{w\}\bisim_ks'$. It is easy to see that for any $k \in \N$, if $s \bisim_k s'$, then $s \bisim_n s'$ for all $n < k$ (for $w\bisim_k w'$ similarly implies $w \bisim_n w'$ for all $n< k$).

We list below some useful properties of state bisimulation;  see \cite{hella2014} for a proof.

\begin{lemma} \label{lemma:bisimulation_results}
    If $s \bisim_{k+1} s'$, then:
    \begin{enumerate}
        \item[(i)] 
        \makeatletter\def\@currentlabel{(i)}\makeatother\label{lemma:bisimulation_results_R[s]} $R[s] \bisim_k R'[s'] $;
        \item[(ii)] 
        \makeatletter\def\@currentlabel{(ii)}\makeatother \label{lemma:bisimulation_results_union} for all $t,u \subseteq s$ such that $s=t\cup u$ there are $t',u' \subseteq s'$ such that $s'= t'\cup u'$, $t\bisim_{k+1} t'$, and $u\bisim_{k+1} u'$.
    \end{enumerate}
\end{lemma}

As expected, state $k$-bisimilarity implies $k$-equivalence in each of the three logics:
\begin{theorem}[Bisimulation invariance]\label{theorem:bisimulation_invariance}
If $s\bisim_k s'$, then $s\equiv_k s'$.
\end{theorem}
\begin{proof}
    For any $\phi$, we show by induction on  $\phi$ that  if $s\bisim_k s'$ for $k=md(\phi)$, then $s \vDash \phi$ iff $s'\vDash \phi$. By Fact \ref{fact:negation_normal_form}, we may assume that $\phi$ is in negation normal form. Most cases can be found in \cite{hella2014}; note that the $\dis$-case follows by Lemma \ref{lemma:bisimulation_results} \ref{lemma:bisimulation_results_union}. 
    The case $\phi=\OC \psi$ follows by $\OC\psi\equiv \psi \intd \bnot \NE$ and the other cases. We now show the cases for the modalities. We only give the detailed proof for one direction of each implication.

    $\phi=\Di \psi$. Suppose $s \bisim_k s'$ for $k= md(\Di\psi)=md(\psi)+1$, and  $s \vDash \Di \psi$. Let $w'\in s'$. By $s \bisim_k s'$, there is some $w\in s$ such that $w \bisim_k w'$. Since $s\vDash \Di \psi$, there is a nonempty $t \subseteq R[w]$ such that $t \vDash \psi$. By $w \bisim_k w'$ we obtain $R[w]\bisim_{k-1} R'[w']$ by Lemma \ref{lemma:bisimulation_results} \ref{lemma:bisimulation_results_R[s]}. Then by Lemma \ref{lemma:bisimulation_results} \ref{lemma:bisimulation_results_union}, there is some $t'\subseteq R'[w']$ such that $t\bisim_{k-1} t'$. By induction hypothesis, $t'\vDash \psi$. Since $t\neq \emptyset$ and $t \bisim_{k-1}t'$, we also have $t'\neq \emptyset$. Hence, $s' \vDash \Di \psi$.
        
    $\phi=\Bo \psi$. Suppose $s \bisim_k s'$ for $k= md(\Bo\psi)=md(\psi)+1$, and $s \vDash \Bo \psi$. Let $w' \in s'$. By $s \bisim_k s'$, there is a $w\in s$ such that $w \bisim_k w'$. Since $s\vDash \Bo \psi$, we have $R[w]\vDash \psi$. By $w \bisim_k w'$ we obtain $R[w]\bisim_{k-1} R'[w']$ by Lemma \ref{lemma:bisimulation_results} \ref{lemma:bisimulation_results_R[s]}. Then by induction hypothesis, $R'[w']\vDash \psi$, so $s' \vDash \Bo \psi$.
 \end{proof}       

\subsection{Expressive completeness} \label{section:expressive_completeness}

We move on to the main results concerning the expressive power of our logics. In the classical world-based setting we measured the expressive power of a logic in terms of the classes of pointed models expressible in the logic; we now use classes of pointed state models.

\begin{definition} \label{def:state_properties}
A \emph{state property} (over $\mathsf{X}$) is a class of pointed state models over $\mathsf{X}$. Each formula $\phi$ \emph{expresses} a property (over $\mathsf{X}\supseteq \mathsf{P}(\phi)$)
$$||\phi||_{\mathsf{X}}:=\{(M,s)\text{ over }\mathsf{X}\sepp M,s \vDash \phi\}.$$
We say that a logic $L$ is \emph{expressively complete} for a class of properties $\mathbb{P}$, written $||L||=\mathbb{P}$, if for each finite $\mathsf{X}$, the class $\mathbb{P}_\mathsf{X}$ of properties over $\mathsf{X}$ in $\mathbb{P}$ is precisely the class of properties over $\mathsf{X}$ expressible by formulas of $L$, that is, if
$$\mathbb{P}_\mathsf{X}=||L ||_\mathsf{X}:=\{|| \phi ||_\mathsf{X}\sepp \phi \in L\}.$$
\end{definition}

We say that a state property $\PPP$  is \emph{invariant under $\mathsf{X},k$-bisimulation ($k\in \N$)} if 
\[[M,s\bisim^\mathsf{X}_k M',s'] \implies (M',s')\in \PPP ,\]
and that $\PPP$ over $\mathsf{X}$ is \emph{invariant under bounded bisimulation} if $\PPP$  is invariant under $\mathsf{X},k$-bisimulation for some $k\in \N$. Similarly, $\PPP$  is \emph{union closed} if  \[[(M,s)\in \PPP \text{ for all } s\in S\neq \emptyset] \implies (M, \bigcup S  )\in \PPP,\]
and $\PPP$ is \emph{flat} if
\[(M, s)\in \PPP  \iff (M,\{w\})\in \PPP \text{ for all }w\in s .\]
Clearly, $\PPP$ is flat iff $\PPP$ is union closed and \emph{downward closed} (i.e., $(M,s)\in \PPP$ and $t \subseteq s$ imply $(M,t)\in \PPP$), and has the \emph{empty state property} (i.e., $(M,\emptyset)\in \PPP$ for all $M$).

    We aim to show three results. First, \BSMLI is expressively complete for
        $$\mathbb{B}:=\{\PPP\sepp \PPP\text{ is invariant under bounded bisimulation}\}.$$
        Second, \BSMLO is complete for 
        $$\mathbb{U}:=\{\PPP\sepp \PPP\text{ is union closed and invariant under bounded bisimulation}\}.$$
        Third, although the properties expressed by formulas of \BSML are union closed and invariant under bounded bisimulation, \BSML is not complete for $\mathbb{U}$. Using the notions we have introduced, it is also easy to show that $\ML$ is complete for $$\mathbb{F}:=\{\PPP\sepp \PPP\text{ is flat and invariant under bounded bisimulation}\}.$$ 
        We include below a brief proof of this folklore result as seeing it alongside the other expressive completeness theorems is instructive. Writing $||L_1||\subset ||L_2||$ if $||L_1||_\mathsf{X}\subseteq ||L_2||_\mathsf{X}$ for all finite $\mathsf{X}$ and $||L_2||_\mathsf{Y}\not\subseteq ||L_1||_\mathsf{Y}$ for some finite $\mathsf{Y}$, we therefore have:
        $$|| \ML || \subset ||\BSML  ||\subset  ||\BSMLO  ||\subset ||\BSMLI ||.$$

We begin by using Hintikka formulas $\chi^k_w \in \ML$ for pointed models to construct Hintikka formulas $\chi^k_s\in \ML$ and $\theta^k_s\in \BSML$ for pointed state models.

\begin{definition}[Hintikka formulas for states]\label{def:characteristic_formulas_states}
   For any pointed state model $(M,s)$ over $\mathsf{Y}\supseteq \mathsf{X}$ and any $k \in \N$, the \emph{$k$-th Hintikka formula} $\chi^{\mathsf{X},k}_{M,s}$ (or simply $\chi_s^k$) and the \emph{$k$-th strong Hintikka formula} $\theta^{\mathsf{X},k}_{M,s}$ (or simply $\theta_s^k$) of $(M,s)$ are defined as follows:
    \begin{align*}
        \chi^{\mathsf{X},k}_{M,s}&:=\displaystyle\bigvee_{w\in s}\chi^k_{w};\\
        \theta^{\mathsf{X},k}_{M,s}&:=\displaystyle\bigvee_{w\in s}(\chi^k_{w}\land \NE).
    \end{align*}
\end{definition}
Recall that we stipulate that $\bigvee\emptyset=\bot$. Since there are only finitely many non-equivalent $\chi^k_w$ for a given finite $\mathsf{X}$, we may assume that the disjunctions in $\theta^k_s$ and $\chi^k_w$ are finite, and hence that the formulas are well-defined. These formulas function as characteristic formulas for states in the sense of the following proposition.

\begin{proposition} \label{prop:characteristic_formulas_characterize_states}\
    \begin{enumerate}
        \item[(i)] 
        \makeatletter\def\@currentlabel{(i)}\makeatother \label{prop:characteristic_formulas_characterize_states_classical} $M',s' \vDash \chi^{k}_{s} \iff $ there is a state $t\subseteq s$ such that $M,t\bisim_k M',s'$.
        \item[(ii)] 
        \makeatletter\def\@currentlabel{(ii)}\makeatother \label{prop:characteristic_formulas_characterize_states_theta} $M',s' \vDash \theta^{k}_{s} \iff M,s \bisim_k M',s'$.
    \end{enumerate}
\end{proposition}
\begin{proof}
    We prove the two items simultaneously. If $s=\emptyset$, then $\theta^{k}_{\emptyset}=\chi^{k}_{\emptyset}=\bot$, and 
    $s'\vDash \bot \iff s'=\emptyset \iff s\bisim_k s' \iff \exists t\subseteq s:t\bisim_k s' $. Now assume that $s \neq \emptyset$.
        
    $\Longleftarrow$: We first prove \ref{prop:characteristic_formulas_characterize_states_classical}. Suppose $t\bisim_k s'$ for some $t\subseteq s$. Then for any $w'\in s'$, there exists $w\in t\subseteq s$ such that $w \bisim_k w'$, which, by Theorem \ref{theorem:hintikka_bisimulation}, implies that $w' \vDash \chi^{k}_{w}$ and so $w'\vDash \chi_s^k$. Since the classical formula $\chi_s^k$ is flat (Fact \ref{fact:NE_intd_closure_properties}), we conclude that $s'\vDash\chi_s^k$.
      
    For \ref{prop:characteristic_formulas_characterize_states_theta}, suppose $s \bisim_k s'$. First observe that the back condition for  $s \bisim_k s'$ implies that there is some $t\subseteq s$ such that $t\bisim_k s'$. Thus, by  \ref{prop:characteristic_formulas_characterize_states_classical}, we have $s'\vDash\chi_s^k$, which means that for every $w\in s$, there exists $s'_w\subseteq s'$ such that $s'=\bigcup_{w\in s}s'_w$ and $s'_w\vDash\chi_w^k$. Meanwhile, by the forth condition for $s \bisim_k s'$, we know that for every $w\in s$, there exists $w'\in s'$ such that $w\bisim_kw'$, which implies $w'\vDash \chi_w^k$ by Theorem \ref{theorem:hintikka_bisimulation}.  Thus, we may w.l.o.g. assume that each $s_w'\neq \emptyset$ (for we can simply include the world $w'$ in $s_w'$), giving that $s_w'\vDash\chi_w^k\wedge \NE$. Hence, $s'\vDash\theta_s^k$.

   $\Longrightarrow$: For \ref{prop:characteristic_formulas_characterize_states_classical}, suppose $s'\vDash\chi_s^k$. Then, for each $w\in s$, there exists a subset $s_w'\subseteq s'$ such that $s'=\bigcup_{w \in s}s_w'$ and $s_w' \vDash \chi^{k}_{w}$. By the empty state property of classical formulas and Theorem \ref{theorem:hintikka_bisimulation}, we have that either $s_w'=\emptyset$ or $w\bisim_ks_w'$. Let $t=\{w\in s\mid s_w'\neq \emptyset\}\subseteq s$. Clearly, the forth condition for $t\bisim_ks'$ is satisfied. To verify the back condition, observe that for any $v'\in s'=\bigcup_{w \in s}s_w'$, there exists $w\in s$ such that $v'\in s_w'\neq \emptyset$. Thus, $w\in t$ and $w\bisim_ks'_w$, whereby $w\bisim_kv'$.
            
    For \ref{prop:characteristic_formulas_characterize_states_theta}, suppose $s'\vDash\theta_s^k$. Then, for every $w\in s$, there exists $s_w'\subseteq s'$ such that $s'=\bigcup_{w \in s}s_w'$ and  $s_w' \vDash \chi^{k}_{w}\wedge\NE$. We then have $s_w'\neq \emptyset$ and $w\bisim_ks'$ (by Theorem \ref{theorem:hintikka_bisimulation} again). Analogous to the above proof for \ref{prop:characteristic_formulas_characterize_states_classical}, we define $t=\{w\in s\mid s_w'\neq \emptyset\}\subseteq s$ and obtain $t\bisim_ks'$ by the same argument. But since now each $s_w'\neq \emptyset$, we have actually $t=s$ and thus $s\bisim_ks'$.
\end{proof}
            
      The formulas $\chi^k_s$ are discussed also in \cite{hella2014}, and the formulas $\chi^k_s$ and $\theta^k_s$ are modal versions of the propositional characteristic formulas defined in \cite{yangvaananen2016,yang2017}. Table \ref{table:characteristic_formulas} presents a summary of the different characteristic formulas we consider. 
      
    \begin{table}[t] 
    \begin{center}
        \begin{tabular}{ l l l l } 
            Hintikka formula & Logic & Characterization effect & Defined in \\ 
            \hline
            $\chi^k_w$ & $\ML$ & $w'\vDash \chi^k_w \iff w\bisim_k w'$ & Definition \ref{def:hintikka}\\
            $\chi^k_s=\bigdis_{w\in s}\chi^k_w$ & $\ML$ & $s'\vDash \chi^k_s \iff t\bisim_k s'$ for some $t\subseteq s$ & Definition \ref{def:characteristic_formulas_states}\\
            $\theta^k_s=\bigdis_{w\in s}(\chi^k_w\land \NE)$ & $\BSML$ & $s'\vDash \theta^k_s \iff s\bisim_k s'$ & Definition \ref{def:characteristic_formulas_states}\\
            &&&\\
            Normal form & Logic & Characterization effect & Defined in \\
            \hline
            $\nu^k_\PPP=\bigdis_{(M,s)\in \PPP}\chi^k_s$ & $\ML$ & $M,s\vDash \nu^k_\PPP \iff (M,s)\in \PPP \text{ (}\PPP \in \mathbb{F})$ & Definition \ref{def:characteristic_formulas_F}\\
            $\xi^k_\PPP=\bigintd_{(M,s)\in \PPP}\theta^k_s$ & $\BSMLI$ & $M,s\vDash \xi^k_\PPP \iff (M,s)\in \PPP \text{ (}\PPP \in \mathbb{B})$ & Definition \ref{def:characteristic_formulas_B}\\
            $\zeta^k_\PPP=\bigdis_{(M,s)\in \PPP} \OC \theta^k_s$ & $\BSMLO$ & $M,s\vDash \zeta^k_\PPP \iff (M,s)\in \PPP \text{ (}\PPP \in \mathbb{U}^\emptyset\subseteq \mathbb{U})$ & Definition \ref{def:characteristic_formulas_U} \\
            $\NE\land \zeta^k_\PPP$ & $\BSMLO$ & $M,s\vDash \NE \land  \zeta^k_\PPP\iff (M,s)\in \PPP \text{ (}\PPP \in \mathbb{U}^{\NE}\subseteq \mathbb{U})$ & Definition \ref{def:characteristic_formulas_U}
        \end{tabular}
    \end{center}
    \caption{Characteristic formulas}
       \label{table:characteristic_formulas} 
    \end{table}

      Table \ref{table:characteristic_formulas} also illustrates the main results of the remaining part of this section: we will use Hintikka formulas $\chi^k_s$ and $\theta^k_s$ for states to construct characteristic formulas for different types of state properties, and the formula for a given type of property will, in turn, be used to show that the logic it belongs to is expressively complete with respect to that type of property. Each of the characteristic formulas for a type of property also exemplifies the disjunctive normal form for the logic in question. Similar disjunctive normal forms  are studied in the modal context in, for example, \cite{hella2014,kontinen2014,hella2015}; and in the propositional context in, for example, \cite{yangvaananen2016,yang2017,CiardelliRoelofsen2011}.

We now give the proofs of our expressive completeness results in order. We start with the folklore result that $\ML$ is complete for $\mathbb{F}$. Hereafter, we often abbreviate a pointed state model $(M,s)$ in a property $\PPP$ simply as $s$.

\begin{definition}[Characteristic formulas for properties in $\mathbb{F}$] \label{def:characteristic_formulas_F}
    For any $\PPP\in \mathbb{F}_\mathsf{X}$ and $k\in \mathbb{N}$, the \emph{$k$-th characteristic formula} $\nu^{\mathsf{X},k}_\PPP\in \ML$ of $\PPP$ is defined as follows:
    $$\nu^{\mathsf{X},k}_\PPP:=\bigdis_{s\in \PPP}\chi^k_s$$
\end{definition}
As before, the formula is clearly well-defined given a specific finite $\mathsf{X}$.

\begin{proposition}[$\MLBF$ is expressively complete for $\mathbf{\mathbb{F}}$] \label{prop:ML_expressive_completeness}
            $$||\ML||=\mathbb{F}:=\{\mathcal{P}\sepp \mathcal{P}\text{ is flat and invariant under bounded bisimulation}\}.$$ 
        \end{proposition}
\begin{proof}
The inclusion $\subseteq$ follows from Theorem \ref{theorem:bisimulation_invariance} and Fact \ref{fact:NE_intd_closure_properties}. For the converse inclusion $\supseteq$, we show that for any $\PPP\in \mathbb{F}$ invariant under $k$-bisimulation:
        $$M',s'\vDash \bigdis_{(M,s)\in \PPP}\chi^k_s\iff (M',s')\in \PPP\quad \text{(i.e., $\PPP=||\nu^k_\PPP||\in || \ML ||$).}$$
        
            $\Longleftarrow$: Suppose $s' \in \PPP$. Since $s'\bisim_k s'$, we have $s' \vDash \chi^k_{s'}$ by Proposition \ref{prop:characteristic_formulas_characterize_states}\ref{prop:characteristic_formulas_characterize_states_classical}, which implies $s' \vDash \bigdis_{s\in \PPP}\chi^k_s$ by the empty state property of $\chi_s^k$.
    
    $\Longrightarrow$: Suppose $s' \vDash \bigdis_{s\in \PPP}\chi^k_s$. Then for each $s\in \PPP$, there exists $t_s'\subseteq s'$ such that $s'=\bigcup_{s\in \PPP}t_s'$ and $t_s'\vDash \chi_s^k$. The latter means, by Proposition \ref{prop:characteristic_formulas_characterize_states}\ref{prop:characteristic_formulas_characterize_states_classical}, that there exists $t\subseteq s$ such that $t_s'\bisim_kt$. Since $\mathbb{F}$ is downward closed and invariant under $k$-bisimulation, we have $t_s'\in \PPP$ for each $s\in \mathcal{P}$. Finally, since $\mathcal{P}$ is union closed, and also using the fact that since $\mathcal{P}$ has the empty state property, it is not the empty property, we conclude $s'=\bigcup_{s\in \PPP}t_s'\in \PPP$.
\end{proof}

The theorem also yields an alternative normal form for formulas in $\ML$: For an arbitrary $\ML$-formula $\alpha$, since $||\alpha||$ is invariant under $k$-bisimulation for $k=md(\alpha)$, we have $||\alpha||=||\nu_{||\alpha||}^{md(\alpha)}||$, that is, $\alpha \equiv \bigdis_{s\in ||\alpha||}\chi^{md(\alpha)}_s$.

Next, we  show the first of our main results: $\BSMLI$ is complete for $\mathbb{B}$, and is thus the expressively strongest logic. The proof is similar to the one given for $\ML$ and $\mathbb{F}$, but the addition of $\NE$ and $\intd$ allow us to fully characterize states and to capture all properties invariant under bounded bisimulation.

\begin{definition}[Characteristic formulas for properties in $\mathbb{B}$] \label{def:characteristic_formulas_B}
    For any $\PPP\in \mathbb{B}_\mathsf{X}$ and $k\in \mathbb{N}$, the \emph{$k$-th characteristic formula} $\xi^{\mathsf{X},k}_\PPP\in \BSMLI$ of $\PPP$ is defined as follows:
    $$\xi^{\mathsf{X},k}_\PPP:=\bigintdd_{s\in \PPP}\theta^k_s.$$
\end{definition}
Recall that we stipulate that $\bigintd\emptyset=\Bot$.

\begin{theorem}[$\BSMLIBF$ is expressively complete for $\mathbb{B}$] \label{theorem:characterization_BSMLI}
    $$||\BSMLI||=\mathbb{B}=\{\PPP\sepp \PPP\text{ is invariant under bounded bisimulation}\}$$
\end{theorem}
\begin{proof}
The inclusion $\subseteq$ follows from Theorem \ref{theorem:bisimulation_invariance}. For the nontrivial inclusion $\supseteq $, we show that for any $\PPP\in\mathbb{B}$ invariant under $k$-bisimulation: $$M',s'\vDash \bigintdd_{(M,s)\in \PPP}\theta^k_s\iff (M',s')\in \PPP\quad\text{(i.e., }\PPP=||\xi^k_\PPP||\in ||\BSMLI||).$$

If $\PPP=\emptyset$, then $\xi^k_\PPP=\Bot$, and thus $\PPP=||\Bot||=||\xi^k_\PPP||$. Now, suppose $\PPP\neq \emptyset$. If $s' \vDash \bigvvee_{s\in \PPP}\theta^k_s$, then $s' \vDash \theta^{k}_{s}$ for some $s \in \PPP$. By Proposition \ref{prop:characteristic_formulas_characterize_states}\ref{prop:characteristic_formulas_characterize_states_theta}, we have $s \bisim_k s'$, which implies $s'\in \PPP$ by invariance under $k$-bisimulation. Conversely, if $s'\in \PPP$, since $s'\bisim_k s'$, we have $s' \vDash \theta^k_{s'}$ by Proposition \ref{prop:characteristic_formulas_characterize_states}\ref{prop:characteristic_formulas_characterize_states_theta} and so $s' \vDash \bigvvee_{s\in \PPP}\theta^k_s$.
\end{proof}

As above, this also provides us with a disjunctive normal form: for each $\BSMLI$-formula $\phi$: $||\phi||=||\xi^{md(\phi)}_{||\phi||}||$, that is, $\phi \equiv \bigintd_{s\in ||\phi||}\theta^{md(\phi)}_s$. 

The expressive completeness of $\BSMLO$ for $\mathbb{U}$ is proved in a similar way, but there is a small twist in this case. First define: 
\begin{align*}
    \mathbb{U}^\emptyset&:=\{\PPP\in \mathbb{U}\sepp \PPP\text{ has the empty state property}\};\\
    \mathbb{U}^{\NE}&:=\{\PPP\in \mathbb{U}\sepp \PPP\text{ does not have the empty state property}\}.
\end{align*}
So that $\mathbb{U}=\mathbb{U}^\emptyset\cup \mathbb{U}^{\NE}$. Note that if $\PPP\in \mathbb{U}^{\NE}$ so that for some $M$ we have $(M,\emptyset)\notin \PPP$, then in fact by invariance under $k$-bisimulation $(M',\emptyset)\notin \PPP$ for all $M'$.

\begin{definition}[Characteristic formulas for properties in $\mathbb{U}$] \label{def:characteristic_formulas_U}
    For any $\PPP\in \mathbb{U}_\mathsf{X}$ and $k\in \N$, let 
    \begin{align*}
        &\zeta^{\mathsf{X},k}_\PPP:=\bigdis_{s\in \PPP}\OC\theta^{k}_{s}.
    \end{align*}
    If $\PPP\in \mathbb{U}^\emptyset$, the \emph{characteristic formula} of $\PPP$ is $\zeta^{k}_\PPP$. If $\PPP\in \mathbb{U}^{\NE}$, the \emph{characteristic formula} of $\PPP$ is $\NE\land \zeta^{k}_\PPP$.
\end{definition}

\begin{theorem}[$\BSMLOBF$ is expressively complete for $\mathbb{U}$] \label{theorem:BSMLO_expressive_completeness}
    $$||\BSMLO||=\mathbb{U}=\{\PPP\sepp \PPP\text{ is union closed and invariant under bounded bisimulation}\}.$$
\end{theorem}
\begin{proof}
The inclusion $\subseteq$ follows from Fact \ref{fact:NE_intd_closure_properties} and Theorem \ref{theorem:bisimulation_invariance}. For the nontrivial inclusion $\supseteq$, we first show that for any $\PPP\in\mathbb{U}^\emptyset$ invariant under $k$-bisimulation, 
$$M',s'\vDash \bigdis_{(M,s)\in \PPP}\OC\theta^{k}_{s}\iff (M',s')\in \PPP\quad\text{(i.e., $\PPP=||\zeta^{k}_\PPP||\in ||\BSMLO||$).}$$

$\Longleftarrow$: Suppose $s'\in \PPP\in \mathbb{U}^\emptyset$. Since $s'\bisim_k s'$, by Proposition \ref{prop:characteristic_formulas_characterize_states} \ref{prop:characteristic_formulas_characterize_states_theta}, we have $s'\vDash \theta^k_{s'}$, whereby $s'\vDash\oslash \theta^k_{s'}$. Hence, $s'\vDash \bigdis_{s\in \PPP}\OC\theta^{k}_{s}$ by the empty state property of $\OC\theta^{k}_{s}$.

$\Longrightarrow$: Suppose $s'\vDash \bigdis_{s\in \PPP}\oslash\theta^k_{s}$. Then for each $s\in \PPP$, there exists $t_s'\subseteq s'$ such that $s'=\bigcup_{s\in \PPP}t_s'$ and $t_s'\vDash \oslash\theta^k_{s}$. The latter implies, by Proposition \ref{prop:characteristic_formulas_characterize_states} \ref{prop:characteristic_formulas_characterize_states_theta}, that $t_s'=\emptyset$ or $t_s'\bisim_k s$. Since $\PPP \in\mathbb{U}^\emptyset$ has the empty state property, and since it is invariant under $k$-bisimulation, we have $t_s'\in \PPP$ for each $s\in \PPP$. Finally, since $\PPP$ is union closed, and also using the fact that since $\PPP$ has the empty state property, it is not the empty property, we conclude $s'=\bigcup_{s\in \PPP}t_s'\in \PPP$.

Next, we show that for any $\PPP\in\mathbb{U}^{\NE}$ invariant under $k$-bisimulation,  
$$M',s'\vDash \NE \land \bigdis_{(M,s)\in \PPP}\OC\theta^{k}_{s}\iff (M',s')\in \PPP\quad\text{(i.e., $\PPP=||\NE\land \zeta^{k}_\PPP||\in ||\BSMLO||$).}$$ 
If $\PPP=\emptyset$, then $\zeta^{k}_\PPP=\bot$, and thus $\PPP=||\Bot||=||\NE \land \bot||=||\NE \land \zeta^{k}_\PPP||$. Now, suppose $\PPP\neq \emptyset$.

$\Longrightarrow$: Suppose $s'\in \PPP\in \mathbb{U}^{\NE}$. By the same argument as in the $\mathbb{U}^{\emptyset}$-case, we obtain $s'\vDash \bigdis_{s\in \PPP}\OC\theta^{k}_{s}$. Moreover, since $\PPP\in \mathbb{U}^{\NE}$, we have $s'\neq \emptyset$ so $s'\vDash \NE$.

$\Longleftarrow$: Suppose $s'\vDash\NE\land  \bigdis_{s\in P}\oslash\theta^k_{s}$. As in the $\mathbb{U}^{\emptyset}$-case, for each $s\in \PPP$, there exists $t_s'\subseteq s'$ such that $s'=\bigcup_{s\in \PPP}t_s'$ and $t_s'=\emptyset$ or $t_s'\bisim_k s$. Let $\QQQ:=\{s\in \PPP\sepp t_s'\bisim_k s\}$ so that $s'=\bigcup_{s\in \QQQ}t_s'$. Since $\PPP$ is invariant under $k$-bisimulation, we have $t'_s\in \PPP$ for each $s\in \QQQ$, and since $s' \vDash \NE$, we have $s'\neq \emptyset$ so that $\QQQ\neq \emptyset$. Finally, since $\PPP$ is union closed and $\QQQ\neq \emptyset$, we conclude $s'=\bigcup_{s\in \QQQ}t_s'\in \PPP$.
\end{proof}

Again the proof yields a disjunctive normal form: for any $\BSMLO$-formula $\phi$, either $\phi\equiv \zeta^{md(\phi)}_{||\phi||}$ or $\phi\equiv\NE\land \zeta^{md(\phi)}_{||\phi||}$.

Finally, we remark that while $\BSML$ is union closed and invariant under bisimulations, it is not complete for $\mathbb{U}$. To prove this, we show first that for \BSML-formulas, the empty state property implies the downward closure property. Note that while converse might also appear to hold, it does not---for instance, the formula $\Bot$ is downward closed but does not have the empty state property.

\begin{lemma} \label{lemma:BSML_downward_closure_empty_state}
    For any $\phi\in\BSML$, if $\phi$ has the empty state property, then $\phi $ is downward closed.
\end{lemma}
\begin{proof}
    By induction on $\phi$ (assumed to be in negation normal form). We only show the nontrivial cases.
  
        Let $\phi = \psi \land \chi$. If $M,\emptyset \vDash \psi \land \chi$, then $M,\emptyset\vDash \psi$ and $M,\emptyset\vDash \chi$. Therefore, by the induction hypothesis, $\psi$ and $\chi$ are downward closed. So if $s\vDash \psi \land \chi$ and $t\subseteq s$, then $t\vDash \psi$ and $t \vDash \chi$, and therefore $t\vDash \psi \land \chi$.
        
        Let $\phi = \psi \dis \chi$. If $M,\emptyset \vDash \psi \dis \chi$, then clearly $M, \emptyset\vDash \psi$ and $M,\emptyset\vDash \chi$. So by the induction hypothesis, $\psi$ and $\chi$ are downward closed. If $s\vDash \psi \dis \chi$, and $t\subseteq s$, then there are $s_1,s_2$ such that $s=s_1\cup s_2$, $s_1\vDash \psi$ and $s_2\vDash \chi$. By downward closure we have $t\cap s_1\vDash \psi$ and $s\cap t_2 \vDash \chi$ and therefore $t=(t\cap s_1)\cup (t\cap s_2)\vDash \psi \dis \chi$.
\end{proof}

\begin{fact}[$\BSMLBF$ is not expressively complete for $\mathbb{U}$] \label{fact:BSML_expressive_power}
    $$||\BSML||\subset \mathbb{U}=\{\PPP\sepp \PPP\text{ is invariant under bounded bisimulation and union closed}\}.$$
\end{fact}
\begin{proof}
The inclusion $\subseteq$ follows from Fact \ref{fact:NE_intd_closure_properties} and Theorem \ref{theorem:bisimulation_invariance}. For inequality, let $$\PPP:=||(p\land \NE)\lor (\lnot p \land \NE)||\cup ||\bot||.$$ Clearly $\PPP\in \mathbb{U}$. Assume for contradiction that  $||\phi ||=\PPP$ for some $\phi$ in $\BSML$. Then $\phi$ has the empty state property so it is downward closed by Lemma \ref{lemma:BSML_downward_closure_empty_state}. Let $(M,s)\in || (p \land \NE) \dis (\bnot p \land \NE)||\subseteq \PPP$. Then there are $t,u$ such that $s=t\cup u$, $t\vDash p \land \NE$, and $u\vDash \bnot p \land \NE$. By $||\phi ||=\PPP$, we have $s\vDash \phi$. By downward closure, $t\vDash \phi$, so $(M,t)\in \PPP$. But  
$t\nvDash (p \land \NE) \dis (\bnot p \land \NE) $ and $t\nvDash \bot$; a contradiction.
\end{proof}

The property $\PPP\in \mathbb{U}$ in the above proof is clearly expressible in the expressively complete logic \BSMLO as $\PPP=||\OC((p \land \NE)\lor (\lnot p \land \NE))||$.

We conclude this section by noting some consequences of the expressive power results. First, it is shown in \cite{dagostino} that any state-based modal logic with the locality property and which is also \emph{forgetting} enjoys \emph{uniform interpolation}. Each of our logics has the locality property, and it follows from our expressive completeness results that $\BSMLI$ and $\BSMLO$ are forgetting. Therefore $\BSMLI$ and $\BSMLO$ enjoy uniform interpolation. (See \cite{dagostino} for details; for the locality requirement, see \cite{yang2022}.)

Second, the finite model property for our logics follows from that
for $\ML$, by a simple argument that makes use of disjoint unions of models. The \emph{disjoint union} $\biguplus_{i\in I}M_i$ of the models $\{M_i\sepp i\in I\neq \emptyset\}$ is defined as usual. Recall in particular that the domain of the disjoint union is defined as $\bigcup_{i\in I}(W_i\times \{i\})$. We also extend the notion to state properties in a natural way: the \emph{disjoint union} of a nonempty property $\PPP=\{(M_i,s_i)\sepp i\in I\}$ is $\biguplus \PPP:=(\biguplus_{i\in I} M_i, \biguplus_{i\in I} s_i)$, where $\biguplus_{i\in I} s_i:=\bigcup_{i\in I}(s_i \times\{i\})$. To simplify notation, we define the disjoint union $\biguplus \emptyset$ of the empty property to be $(M,\emptyset)$ for some fixed $M$. We refer to an element $(w,i)$ in the disjoint union as simply $w$, and similarly for states. The following is standard (see, for example, \cite{blackburn2001,goranko2007}):

\begin{proposition} \label{prop:disjoint_union}
    For all $i\in I,w\in W_i,k\in \mathbb{N}, \mathsf{X}\subseteq\mathsf{Prop}: M_i,w \bisim^\mathsf{X}_k \biguplus_{i\in I} M_i, w $. Consequently, for all $s\subseteq W_i: M_i,s\bisim_k \biguplus_{i\in I} M_i, s $.
\end{proposition}

\begin{proposition}[Finite model property] \label{prop:finite_model_property}
If $\nvDash\phi$, then there is some finite model $M$ and state $s$ such that $M,s\nvDash \phi$.
\end{proposition}
\begin{proof}
 Let $(M,s)$ be such that $M,s\nvDash \phi$ and let $\mathsf{X}=\mathsf{P}(\phi)$ and $k=md(\phi)$. If $s=\emptyset$, we clearly have $M',s\nvDash \phi$ for all finite $M'$. Now assume $s\neq \emptyset$. Since there are only finitely many $\mathsf{X},k$-bisimilarity types (or Hintikka formulas), we can pick a finite $t\subseteq s$ such that $t \bisim^\mathsf{X}_k s$. Clearly for any $w\in t$ we have $\chi^k_w\notequiv \bot$; therefore, by the finite model property for $\ML$, there is some finite $(M_w,v_w)$ such that $M_w,v_w\vDash \chi^k_w$. Then by Theorem \ref{theorem:hintikka_bisimulation}, $M_w,v_w \bisim_k M,w$. By Proposition \ref{prop:disjoint_union}, $\biguplus_{w\in t}M_w,v_w\bisim_k M_w,v_w$ for all $w\in t$; then also $\biguplus_{w\in t}M_w,v_w\bisim_k M,w$ for all $w\in t$, so $\biguplus_{w\in t}M_w,\{v_w\sepp w\in t\}\bisim_k M,t$, whence $\biguplus_{w\in t}M_w,\{v_w\sepp w\in t\}\nvDash \phi$ by Theorem \ref{theorem:bisimulation_invariance}, where $\biguplus_{w\in t}M_w$ is clearly finite.
\end{proof}

Together with the completeness results to be proved in Section \ref{section:axiomatizations}, this implies that our logics $\BSML$, $\BSMLO$, and $\BSMLI$ are decidable.

%% file: sections/axiomatization/axiomatization.tex
\section{Axiomatizations} \label{section:axiomatizations}

In this section we introduce sound and complete natural deduction systems for $\BSML$ and its extensions. The strategy we employ in proving completeness makes essential use of the expressive power results in Section \ref{section:expressive_power}: for the expressively complete logics $\BSMLI$ and $\BSMLO$, we show that each formula is provably equivalent to some formula in disjunctive normal form---completeness then reduces to showing it for formulas in normal form. As for $\BSML$, for which no similar normal form is available, our strategy involves simulating the $\BSMLI$-normal form by using sets of $\BSML$-formulas. For this reason we first axiomatize the extensions.

\input{sections/axiomatization/BSMLI/bsmli_rules}
\input{sections/axiomatization/BSMLI/bsmli_completeness}
\input{sections/axiomatization/BSMLI/bsmli_normal_form}
\input{sections/axiomatization/BSMLO/bsmlo_rules}
\input{sections/axiomatization/BSMLO/bsmlo_completeness}
\input{sections/axiomatization/BSMLO/bsmlo_normal_form}
\input{sections/axiomatization/BSML/bsml_rules}
\input{sections/axiomatization/BSML/bsml_completeness.tex}

%% file: sections/axiomatization/BSMLI/bsmli_rules.tex
\subsection{$\BSMLIBF$} \label{section:BSMLI}

We first focus on the strongest logic $\BSMLI$. The propositional fragment of $\BSMLI$ corresponds to the logic $\PT$ studied in \cite{yang2017} (with the caveat that the negation in $\PT$ is not the bilateral negation). Our system for $\BSMLI$ is essentially an extension (and simplification) of the propositional system for $\PT$ presented in \cite{yang2017}, with additional rules for the bilateral negation and the modalities.

Before we present the system, let us issue a word of caution. Recall from Section \ref{section:preliminaries} that $\BSML$ and its extensions are not closed under uniform substitution. Due to this our systems will \emph{not} admit the usual {\em Uniform Substitution} rule. Note in particular that in the presentation of our rules, the metavariables $\alpha$ and $\beta$ range, as before, exclusively over formulas of $\ML$. 

Most rules in the system for $\BSMLI$ are also rules in the systems for $\BSML$ and $\BSMLO$. For ease of reference we first present these shared rules; all rules involving $\intd$ are grouped together at the end.

\begin{definition}[Natural deduction system for $\BSMLIBF$] \label{def:BSMLI_system}
    The following rules comprise a natural deduction system for $\BSMLI$.
    
    \noindent 
    (a) Rules for $\land$:
        \begin{proofbox}[]
        {\small
            
            \begin{minipage}{.3\textwidth}
                \begin{prooftree}
                    \AxiomC{$D_1 $}
                    \noLine
                    \UnaryInfC{$\phi $}
                    \AxiomC{$D_2 $}
                    \noLine
                    \UnaryInfC{$\psi$}
                    \RightLabel{$\land \R{I}$}
                    \BinaryInfC{$ \phi \land\psi$}
                \end{prooftree}
            \end{minipage}
            \begin{minipage}{.3\textwidth}
                \begin{prooftree}
                    \AxiomC{$D$}
                    \noLine
                    \UnaryInfC{$\phi \land \psi$}
                    \RightLabel{$\land \R{E}$}
                    \UnaryInfC{$ \phi $}
                \end{prooftree}
            \end{minipage}
            \begin{minipage}{.3\textwidth}
                \begin{prooftree}
                    \AxiomC{$D$}
                    \noLine
                    \UnaryInfC{$\phi \land \psi$}
                    \RightLabel{$\land \R{E}$}
                    \UnaryInfC{$ \psi $}
                \end{prooftree}
            \end{minipage}
            }
        \end{proofbox}

    \noindent 
    (b) Rules for $\lnot$:

        \begin{proofbox}[]
        {\small
        
            \begin{minipage}{.33\textwidth}
                \begin{prooftree}
                    \AxiomC{$[\alpha] $}
                    \noLine
                    \UnaryInfC{$D$}
                    \noLine
                    \UnaryInfC{$\bot$}
                    \RightLabel{$\bnot \R{I}(*)$}
                    \UnaryInfC{$ \bnot \alpha$}
                \end{prooftree}
            \end{minipage}
            \begin{minipage}{.33\textwidth}
                \begin{prooftree}
                    \AxiomC{$D_1 $}
                    \noLine
                    \UnaryInfC{$\alpha $}
                    \AxiomC{$D_2 $}
                    \noLine
                    \UnaryInfC{$\bnot \alpha$}
                    \RightLabel{$\bnot \R{E}$}
                    \BinaryInfC{$ \beta $}
                \end{prooftree}
            \end{minipage}
            \begin{minipage}{.33\textwidth}
                \begin{prooftree}
                    \AxiomC{$D  $}
                    \noLine
                    \UnaryInfC{$\bnot \bnot \phi $}
                    \doubleLine
                    \RightLabel{$\lnot \lnot\R{E}$}
                    \UnaryInfC{$ \phi $}
                \end{prooftree}
            \end{minipage}
            \vspace{0.2cm}
            
            \begin{minipage}{.33\textwidth}
                \begin{prooftree}
                    \AxiomC{$D  $}
                    \noLine
                    \UnaryInfC{$\bnot (\phi \land \psi) $}
                    \doubleLine
                    \RightLabel{$\R{DM}_\land$}
                    \UnaryInfC{$ \bnot \phi \lor \bnot \psi $}
                \end{prooftree}
            \end{minipage}
            \begin{minipage}{.33\textwidth}
                \begin{prooftree}
                    \AxiomC{$D  $}
                    \noLine
                    \UnaryInfC{$\bnot (\phi \dis \psi) $}
                    \doubleLine
                    \RightLabel{$\R{DM}_\lor$}
                    \UnaryInfC{$ \bnot \phi \land \bnot \psi $}
                \end{prooftree}
            \end{minipage}
            \begin{minipage}{.33\textwidth}
                \begin{prooftree}
                    \AxiomC{$D  $}
                    \noLine
                    \UnaryInfC{$\bnot \NE $}
                    \doubleLine
                    \RightLabel{$\bnot \NE\R{E}$}
                    \UnaryInfC{$ \bot $}
                \end{prooftree}
            \end{minipage}
            \vspace{0.2cm}}
            
            {\footnotesize
            $(*)$ The undischarged assumptions in $D$ do not contain $\NE$.}
        \end{proofbox}

    \noindent 
    (c) Rules for $\dis$:
        \begin{proofbox}[]
        {\small

            \begin{minipage}{.22\textwidth}
                \begin{prooftree}
                    \AxiomC{$D$}
                    \noLine
                    \UnaryInfC{$\phi $}
                    \RightLabel{$\dis \R{I}(*)$}
                    \UnaryInfC{$ \phi \dis\psi$}
                \end{prooftree}
            \end{minipage}
            \begin{minipage}{.22\textwidth}
                \begin{prooftree}
                    \AxiomC{$D$}
                    \noLine
                    \UnaryInfC{$\phi $}
                    \RightLabel{$\dis \R{W}$}
                    \UnaryInfC{$ \phi \dis\phi$}
                \end{prooftree}
            \end{minipage}
                        \begin{minipage}{.22\textwidth}
                \begin{prooftree}
                    \AxiomC{$D$}
                    \noLine
                    \UnaryInfC{$\phi \dis \psi$}
                    \RightLabel{$\R{Com}\dis$}
                    \UnaryInfC{$ \psi \dis \phi$}
                \end{prooftree}
            \end{minipage}
            \begin{minipage}{.3\textwidth}
                \begin{prooftree}
                    \AxiomC{$D$}
                    \noLine
                    \UnaryInfC{$ \phi \dis (\psi\vee\chi)$}
                    \RightLabel{$\R{Ass}\dis$}
                    \UnaryInfC{$ (\phi \dis \psi)\vee\chi$}
                \end{prooftree}
            \end{minipage}
            \vspace{0.2cm}
            
                        \begin{minipage}{.5\textwidth}
                \begin{prooftree}
                    \AxiomC{$D$}
                    \noLine
                    \UnaryInfC{$\phi \dis \psi$}
                    \AxiomC{$[\phi]$}
                    \noLine
                    \UnaryInfC{$D_1$}
                    \noLine
                    \UnaryInfC{$\chi$}
                    \AxiomC{$[\psi]$}
                    \noLine
                    \UnaryInfC{$D_2$}
                    \noLine
                    \UnaryInfC{$\chi$}
                    \RightLabel{$\dis \R{E}(\dagger,\ddagger)$}
                    \TrinaryInfC{$ \chi $}
                \end{prooftree}
            \end{minipage}
                        \begin{minipage}{.5\textwidth}
                \begin{prooftree}
                    \AxiomC{$D$}
                    \noLine
                    \UnaryInfC{$\phi \dis \psi$}
                    \AxiomC{$[\psi]$}
                    \noLine
                    \UnaryInfC{$D_1$}
                    \noLine
                    \UnaryInfC{$\chi$}
                    \RightLabel{$\dis \R{Mon}(\dagger)$}
                    \BinaryInfC{$ \phi \dis \chi$}
                \end{prooftree}
            \end{minipage}
    
            \vspace{0.2cm}}
    
            {\footnotesize
            $(*)$ $\psi$ does not contain $\NE$.
            
            $(\dagger)$ The undischarged assumptions in $D_1,D_2$ do not contain $\NE$.

            $(\ddagger)$ $\chi$ does not contain $\intd$.} 
        \end{proofbox}

        \noindent
        (d) Rules for $\bot$ and $\NE$:
        \begin{proofbox}[]
        {\small

            \begin{minipage}{0.5\textwidth}
                \begin{prooftree}
                    \AxiomC{$D$}
                    \noLine
                    \UnaryInfC{$\bot\lor \phi$}
                    \RightLabel{$\bot \R{E}$}
                    \UnaryInfC{$ \phi$}
                \end{prooftree}
            \end{minipage}
            \begin{minipage}{0.5\textwidth}
                \begin{prooftree}
                    \AxiomC{$D$}
                    \noLine
                    \UnaryInfC{$\Bot \dis \phi $}
                    \RightLabel{$\Bot \R{Ctr}$}
                    \UnaryInfC{$ \psi$}
                \end{prooftree}
            \end{minipage}}
        \end{proofbox}

    \noindent
    (e) Basic modal rules:
        \begin{proofbox}[]
        {\small
            
            \begin{minipage}{.35\textwidth}
                \begin{prooftree}
                    \AxiomC{$[\phi]  $}
                    \noLine
                    \UnaryInfC{$ D $}
                    \noLine
                    \UnaryInfC{$ \psi $}
                    \AxiomC{$D_1  $}
                    \noLine
                    \UnaryInfC{$\Di \phi $}
                    \RightLabel{$\Di \R{Mon}(*)$}
                    \BinaryInfC{$ \Di\psi $}
                \end{prooftree}
            \end{minipage}
            \begin{minipage}{.65\textwidth}
                \begin{prooftree}
                    \AxiomC{$[\phi_1] \ldots [\phi_n] $}
                    \noLine
                    \UnaryInfC{$ D $}
                    \noLine
                    \UnaryInfC{$ \psi $}
                    \AxiomC{$D_1$}
                    \noLine
                    \UnaryInfC{$\Bo \phi_1$}
                    \AxiomC{$\ldots$}
                    \AxiomC{$D_n$}
                    \noLine
                    \UnaryInfC{$\Bo \phi_n $}
                    \RightLabel{$\Bo \R{Mon}(*)$}
                    \QuaternaryInfC{$ \Bo\psi $}
                \end{prooftree}
            \end{minipage}
                
                       \vspace{0.2cm}

                       \begin{minipage}{.6\textwidth}
                       \vspace{1cm}
            {\footnotesize
            $(*)$ $D$ does not contain undischarged assumptions.}
            \end{minipage}
                    \begin{minipage}{.4\textwidth}
                        \begin{prooftree}
                \AxiomC{$D  $}
                \noLine
                \UnaryInfC{$\bnot \Di \phi  $}
                \RightLabel{$\R{Inter}\Di \Bo$}
                \doubleLine
                \UnaryInfC{$ \Bo \bnot \phi$}
            \end{prooftree}
            \end{minipage}
            }
        \end{proofbox}
          \noindent
    (f) Rules governing the interaction of the modalities and connectives:
        \begin{proofbox}[]
        {\small

            \begin{minipage}{.5\textwidth}
                \begin{prooftree}
                    \AxiomC{$D  $}
                    \noLine
                    \UnaryInfC{$\Di (\phi  \dis (\psi\land \NE)) $}
                    \RightLabel{$\Di\R{Sep}$}
                    \UnaryInfC{$\Di \psi $}
                \end{prooftree}
            \end{minipage}
            \begin{minipage}{.5\textwidth}
                \begin{prooftree}
                    \AxiomC{$D_1  $}
                    \noLine
                    \UnaryInfC{$\Di \phi  $}
                    \AxiomC{$D_2  $}
                    \noLine
                    \UnaryInfC{$ \Di \psi $}
                    \RightLabel{$\Di\R{Join}$}
                    \BinaryInfC{$\Di (\phi  \dis \psi) $}
                \end{prooftree}
            \end{minipage}
            \vspace{0.2cm}

            \begin{minipage}{.5\textwidth}
                \begin{prooftree}
                    \AxiomC{$D  $}
                    \noLine
                    \UnaryInfC{$\Bo (\phi  \land \NE) $}
                    \RightLabel{$\Bo\R{Inst}$}
                    \UnaryInfC{$\Di \phi $}
                \end{prooftree}
            \end{minipage}
            \begin{minipage}{.5\textwidth}
            \begin{prooftree}
                \AxiomC{$D_1  $}
                \noLine
                \UnaryInfC{$\Bo\phi$}
                \AxiomC{$D_2  $}
                \noLine
                \UnaryInfC{$ \Di \psi$}
                \RightLabel{$\Bo\Di\R{Join}$}
                \BinaryInfC{$\Bo (\phi\dis \psi) $}
            \end{prooftree}
            \end{minipage}}
        \end{proofbox}
        
        \noindent
        (g) Propositional rules involving $\intd$ ($\BSMLI$-specific):
        
        \begin{proofbox}[]
        {\small
            \begin{minipage}{.25\textwidth}
                \begin{prooftree}
                    \AxiomC{$D$}
                    \noLine
                    \UnaryInfC{$\phi $}
                    \RightLabel{$\intd \R{I}$}
                    \UnaryInfC{$ \phi \intd\psi$}
                \end{prooftree}
            \end{minipage}
            \begin{minipage}{.25\textwidth}
                \begin{prooftree}
                    \AxiomC{$D$}
                    \noLine
                    \UnaryInfC{$\psi $}
                    \RightLabel{$\intd \R{I}$}
                    \UnaryInfC{$ \phi \intd\psi$}
                \end{prooftree}
            \end{minipage}
            \begin{minipage}{.4\textwidth}
                \begin{prooftree}
                    \AxiomC{$D$}
                    \noLine
                    \UnaryInfC{$\phi \intd \psi$}
                    \AxiomC{$[\phi]$}
                    \noLine
                    \UnaryInfC{$D_1$}
                    \noLine
                    \UnaryInfC{$\chi$}
                    \AxiomC{$[\psi]$}
                    \noLine
                    \UnaryInfC{$D_2$}
                    \noLine
                    \UnaryInfC{$\chi$}
                    \RightLabel{$\intd \R{E}$}
                    \TrinaryInfC{$ \chi $}
                \end{prooftree}
            \end{minipage}
            \vspace{0.2cm}

            \begin{minipage}{.4\textwidth}
                \begin{prooftree}
                    \AxiomC{$D$}
                    \noLine
                    \UnaryInfC{$\phi \dis (\psi \intd \chi)$}
                    \RightLabel{$\R{Distr}\dis \intd $}
                    \UnaryInfC{$( \phi \dis \psi )\intd (\phi \dis \chi)$}
                \end{prooftree}
            \end{minipage}
            \begin{minipage}{.25\textwidth}
                \begin{prooftree}
                    \AxiomC{$D  $}
                    \noLine
                    \UnaryInfC{$\bnot (\phi \intd \psi) $}
                    \doubleLine
                    \RightLabel{$\R{DM}_{\intd}$}
                    \UnaryInfC{$ \bnot \phi \land \bnot \psi $}
                \end{prooftree}
            \end{minipage}
            \begin{minipage}{.25\textwidth}
                \begin{prooftree}
                    \AxiomC{}
                    \RightLabel{$\NE \R{I}$}
                    \UnaryInfC{$ \bot \intd\NE $}
                \end{prooftree}
            \end{minipage}}
        \end{proofbox}

\noindent
    (h) Modal rules for $\intd$ ($\BSMLI$-specific):
        \begin{proofbox}[]
        {\small

            \begin{minipage}{.5\textwidth}
                \begin{prooftree}
                    \AxiomC{$D  $}
                    \noLine
                    \UnaryInfC{$\Di(\phi\intd \psi)$}
                    \RightLabel{$\R{Conv}\Di\intd\dis$}
                    \doubleLine
                    \UnaryInfC{$\Di \phi \dis \Di \psi$}
                \end{prooftree}
            \end{minipage}
            \begin{minipage}{.5\textwidth}
                \begin{prooftree}
                    \AxiomC{$D  $}
                    \noLine
                    \UnaryInfC{$\Bo(\phi\intd \psi)$}
                    \RightLabel{$\R{Conv}\Bo\intd\dis$}
                    \doubleLine
                    \UnaryInfC{$\Bo \phi \dis \Bo \psi$}
                \end{prooftree}
            \end{minipage}}
         \end{proofbox}
        We write $\Phi \vdash_{\BSMLI} \psi$ (or simply $\Phi\vdash \psi$) if $\psi$ is derivable from formulas in $\Phi$ in the system. We also write $\Phi,\phi_1,\dots,\phi_n\vdash \psi$ for $\Phi\cup\{\phi_1,\dots,\phi_n\}\vdash \psi$. We say that $\phi$ and $\psi$ are \emph{provably equivalent}, written $\phi \proveq \psi$, if $\phi \vdash \psi $ and $\psi \vdash \phi$.
    \end{definition}
        
  The conjunction  $\land$ and global disjunction $\intd$ have the standard introduction and elimination rules. The rules involving $\dis$ must be constrained due to the failure of downward/union closure occasioned by the presence of $\NE$ and $\intd$. For instance, $p\nvDash p \dis \NE$, so unconstrained $\dis$-introduction is not sound. To ensure soundness, the introduced disjunct must have the empty state property, and so in $\dis\R{I}$ we require that the introduced disjunct does not contain $\NE$. Alternatively, the introduced disjunct may simply be the premise again, as in $\dis$-weakening $\dis \R{W}$. $\dis$-elimination is similarly constrained to ensure that the subderivations $D_1$ and $D_2$ do not depend on formulas which are not downward-closed, and to ensure that the consequent $\chi$ of the rule is union-closed. The commutativity and associativity of $\dis$, and the distributivity of $\vee$ over $\intd$ need to be included as rules; for the derivable algebraic properties of the connectives, see Proposition \ref{prop:commutativity_associativity_distributivity} below. The rules for $\bnot$ include both constrained standard rules as well as rules corresponding to the equivalences noted in Section \ref{section:preliminaries}. The propositional rules involving $\bot$ and $\NE$---$\bot \R{E}$, $\Bot\R{Ctr}$ ($\Bot$-contraction), $\NE\R{I}$ and $\lnot \NE\R{E}$---are self-explanatory. Note that with $\bot$ we have \emph{ex falso} with respect to classical formulas, and with $\Bot$, with respect to all formulas:
        \begin{lemma} \label{lemma:Bot_E} $\Bot\vdash \phi$.
    \end{lemma}
    \begin{proof}
        $\Bot\vdash \Bot\lor \bot$ by $\dis\R{I}$, and then $\Bot\lor \bot\vdash \phi$ by $\Bot\R{Ctr}$.
    \end{proof}

    The basic modal rules (e) are standard. While the modalities do not distribute over $\intd$, a weaker distributivity in which $\intd$ is switched with $\dis$ --- the conversion rules $\R{Conv}\Di\intd\dis$ and  $\R{Conv}\Bo\intd\dis$---does hold. The $\Di$-separation rule $\Di\R{Sep}$  corresponds to $\FC$-entailment for pragmatically enriched formulas as described in Sections \ref{section:introduction} and \ref{section:preliminaries}. The $\Bo$-instantiation rule $\Bo\R{Inst}$ characterizes the fact that $\Bo\phi$ implies $\Di\phi$ in case accessible worlds exist. The two join rules $\Di\R{Join}$ and $\Bo\Di\R{Join}$ allow one to graft together witnessing accessible states.
    
    \begin{theorem}[Soundness of $\BSMLIBF$] \label{theorem:soundness_BSMLI} If $\Phi \vdash \psi$, then $ \Phi \vDash \psi$.
\end{theorem}
\begin{proof}
By induction on the length $k$ of possible derivations of $\Phi \vdash\psi$. We only include the cases for the novel rules involving modalities, as well as for $\dis\R{E}$; most of the other cases can be found in \cite{yang2017}. The base case ($k=1$) is simple; see \cite{yang2017}. For the inductive case, assume the result holds for all derivations of length $\leq k$. We consider the different possibilities for what the final rule in the derivation $\Phi\vdash \psi$ can be.

$\dis \R{E}$: Assume we have derivations of length $\leq k$ witnessing $\Phi \vdash \phi \dis \psi$; $\Phi_1,\phi \vdash \chi $; and $\Phi_2, \psi \vdash \chi $; that for all $\eta\in \Phi_1\cup \Phi_2$, $\eta$ does not contain $\NE$, and that $\chi$ does not contain $\intd$. By the induction hypothesis, $\Phi \vDash \phi \dis \psi$; $\Phi_1,\phi \vDash \chi $; and $\Phi_2,\psi \vDash \chi $. We show $\Phi, \Phi_1, \Phi_2 \vDash \chi$. Assume $s \vDash \Phi\cup \Phi_1\cup \Phi_2$. Then $s \vDash \phi \dis \psi$, so there are some $t_1,t_2$ such that $s=t_1\cup t_2$; $t_1 \vDash \phi$; and $t_2\vDash \psi$. Since no $\eta \in \Phi_1\cup \Phi_2$ contains $\NE$, each such $\eta$ is downward closed. So $t_1 \vDash \eta_1$ for all $\eta_1\in \Phi_1$, and $t_2 \vDash \eta_2$ for all $\eta_2\in \Phi_2$. Thus $t_1 \vDash \chi$ and $t_2 \vDash \chi$. Since $\chi$ does not contain $\intd$, it is union closed; therefore $s \vDash \chi$.
            
  $\R{Conv} \Di\intd\dis$: It suffices to show $\Di(\phi \intd \psi)\equiv \Di  \phi \dis \Di \psi$. $\Dashv$ is easy. For $\vDash$, let $s\vDash \Di(\phi \intd \psi)$. If $s = \emptyset$, then clearly $ s \vDash \Di \phi \dis \Di \psi$. Otherwise if $s \neq \emptyset$, let $w \in s$. Then there is a nonempty $t\subseteq R[w]$ such that $t \vDash \phi \intd \psi$, so that $t \vDash \phi$ or $t \vDash\psi$. Letting
                \begin{align*}
                    s_1:=\{w \in s\sepp \exists t \subseteq R[w]:t\neq \emptyset \text{ and }t \vDash \phi \}~\text{and}~
                    s_2:=\{w \in s\sepp \exists t \subseteq R[w]:t\neq \emptyset\text{ and }t \vDash \psi \}
                \end{align*}
                we have $s=s_1\cup s_2$. Clearly $s_1\vDash \Di \phi$ and $s_2\vDash \Di \psi$, so $s \vDash \Di \phi \dis \Di \psi$.
        
             $\R{Conv}\Bo \intd \dis$: Analogous to $\R{Conv}\Di \intd \dis$.
        
    $\Di\R{Sep}$: It suffices to show $\Di (\phi  \dis (\psi\land \NE))\vDash \Di \psi$; let $s\vDash \Di (\phi  \dis (\psi\land \NE))$. If $s=\emptyset$, clearly $s \vDash \Di \psi$. If $s\neq \emptyset$, let $w \in s$. Then there is a nonempty $t \subseteq R[w]$ such that $ t\vDash \phi \dis (\psi \land \NE)$. Therefore, there are $t_1,t_2$ such that $t=t_1\cup t_2$; $t_1\vDash \phi $; and $t_2\vDash \psi \land \NE$. Then $t_2\neq \emptyset$ and $ t_2\vDash \psi$, and clearly $t_2\subseteq R[w]$. Thus $s \vDash \Di \psi$.
         
   $\Di\R{Join}$: It suffices to show $\Di  \phi,\Di \psi\vDash \Di (\phi  \dis \psi)$; let $s \vDash \Di  \phi$ and $s\vDash \Di \psi$. If $s=\emptyset$, then clearly $s \vDash \Di (\phi\dis \psi )$. Otherwise if $s \neq \emptyset$, let $w \in s$. Since $ s \vDash \Di  \phi$ and $s\vDash \Di \psi$, there are nonempty $t_1,t_2\subseteq R[w]$ such that $ t_1\vDash \phi$ and $ t_2\vDash \psi$. Therefore $t_1\cup t_2\vDash \phi \dis \psi $. Clearly $t_1\cup t_2$ is nonempty and $t_1\cup t_2\subseteq R[w]$, so $s \vDash \Di (\phi\dis \psi )$.

     $\Bo\R{Inst}$: It suffices to show $\Bo (\phi \land \NE)\vDash \Di \phi$; let $s\vDash \Bo (\phi \land \NE)$. If $s =\emptyset$, 
     clearly $s \vDash \Di \phi$.  
     If $s \neq \emptyset$, let $w\in s$. Then $ R[w]\vDash \phi \land \NE$, so $R[w]\neq \emptyset$, and therefore $s \vDash \Di\phi$.
            
      $\Bo\Di \R{Join}$: It suffices to show $\Bo\phi, \Di \psi\vDash  \Bo (\phi \dis \psi )$; let $s\vDash \Bo\phi$ and $s\vDash \Di \psi$. If $s=\emptyset$, clearly $s\vDash \Bo(\phi \dis \psi )$. If $s \neq \emptyset$, let $w \in s$. Since $s \vDash \Bo\phi$, we have $R[w]\vDash \phi $, and since $s\vDash \Di \psi$, there is some nonempty $t\subseteq R[w]$ such that $ t\vDash \psi$. Since $R[w]=R[w]\cup t$, we have $R[w]\vDash \phi \dis \psi$, and so $s\vDash \Bo (\phi \dis \psi)$.
\end{proof}

%% file: sections/axiomatization/BSMLI/bsmli_completeness.tex
As expected, a limited replacement lemma holds for our system; cf. Proposition \ref{Replacement_thm_semantic}, the semantic replacement lemma w.r.t. strong equivalence. 

\begin{lemma}[Replacement] \label{lemma:provable_subformula_replacement}
 Suppose $\theta$ contains a specific occurrence $\theta[p]$ of $p$ which is not in the scope of $\lnot$ (unless the $\lnot$ forms part of $\Bo=\lnot \Di\lnot$). Then $\phi\vdash \psi$ implies $\theta[\phi/p]\vdash\theta[\psi/p]$. In particular, if $\phi\proveq \psi$, then $\theta[\phi/p]\proveq\theta[\psi/p]$.
\end{lemma}
\begin{proof} 
By a routine induction on $\theta$, where the modality cases are proved by applying $\Di \R{Mon}$ and $\Bo \R{Mon}$.
\end{proof}

We showed in Fact \ref{fact:negation_normal_form} that every formula is equivalent semantically to one in negation normal form. This fact also holds in the proof system.

\begin{proposition}[Negation normal form] \label{prop:negation_normal_form_provable_BSMLI}
    Every formula $\phi$ is provably equivalent to a formula in negation normal form.
\end{proposition}
\begin{proof}
    The result follows by repeated applications of $\lnot \lnot \phi\proveq \phi$ (by $\lnot \lnot \R{E}$); $\lnot (\phi \land \psi) \proveq \lnot \phi \lor \lnot \psi$, $\lnot (\phi \lor \psi) \proveq \lnot \phi \land \lnot \psi$ and $\lnot (\phi \intd \psi) \proveq \lnot \phi \land \lnot \psi$ (by the $\mathsf{DM}$-rules); and $\lnot \Di\phi \proveq \Bo\lnot \phi$ (by $\R{Inter}\Di\Bo$).
\end{proof}

 Standard commutativity, associativity and distributivity laws for $\land$ and $\intd$ can be derived in the usual way in our system. Laws involving $\vee$ hold with some restrictions:  

\begin{proposition} \label{prop:commutativity_associativity_distributivity}\ 
\begin{enumerate}
    \item[(i)]
        \makeatletter\def\@currentlabel{(i)}\makeatother \label{prop:commutativity_associativity_distributivity_i}
        $\phi  \land(\psi \dis \chi)\vdash (\phi \land \psi ) \dis (\phi \land \chi) $ if $\phi$ does not contain $\NE$,\\ and $(\phi \land \psi ) \dis (\phi \land \chi)\vdash \phi  \land(\psi \dis \chi)$  if $\phi$ does not contain $\intd$;
    \item[(ii)] 
        \makeatletter\def\@currentlabel{(ii)}\makeatother \label{prop:commutativity_associativity_distributivity_ii} $\phi \dis (\psi \land \chi)\vdash (\phi \dis \psi )\land (\phi \dis \chi) $;
    \item[(iii)] 
        \makeatletter\def\@currentlabel{(iii)}\makeatother $\phi  \intd(\psi \dis \chi)\vdash (\phi \intd \psi ) \dis (\phi \intd \chi) $;
    \item[(iv)] 
        \makeatletter\def\@currentlabel{(iv)}\makeatother $\phi \dis (\psi \intd \chi) \proveq (\phi \dis \psi )\intd (\phi \dis \chi)$.
\end{enumerate}
\end{proposition}
\begin{proof}
Easy.
\end{proof}

We often apply the above three basic lemmas (Lemmas \ref{lemma:provable_subformula_replacement}--\ref{prop:commutativity_associativity_distributivity}) without explicit reference to them. Note that all results in this section which do not involve $\intd$ (such as \ref{prop:commutativity_associativity_distributivity_i} and \ref{prop:commutativity_associativity_distributivity_ii} of Proposition \ref{prop:commutativity_associativity_distributivity}) also hold for $\BSMLO$ and $\BSML$. Some such results are given $\BSMLI$-specific derivations in this section. The $\BSMLO$- and $\BSML$-derivations can be found in the sequel.
    
    We now move on to prove the completeness theorem for the system. Our strategy consists in showing that each formula is provably equivalent to one in disjunctive normal form---as stated in the following lemma---and then making use of the semantic and proof-theoretic properties of formulas in this form.
    
\begin{lemma} \label{lemma:normal_form_provable_equivalence_BSMLI}
    For each $\phi\in \BSMLI$, each $k \geq md(\phi)$, and each finite $\mathsf{X}\supseteq \mathsf{P}(\phi)$, there is some property $\PPP$ over $\mathsf{X}$ such that:
    $$\phi \proveq \bigintdd_{s\in \PPP}\theta^{\mathsf{X},k}_s.$$
\end{lemma}

The proof of the above lemma is involved. We withhold this proof for now, and first present the main completeness argument.

 As a first step, let us note that our system is a conservative extension of the smallest normal logic $\mathbf{K}$---it is easy to see by inspecting our rules, that, for instance, the axioms and rules of the Hilbert-style system for $\mathbf{K}$ are derivable (see also \cite{yang20172} for a detailed proof of the analogous fact for a similar system). Therefore, since $\mathbf{K}$ is complete with respect to the class of all Kripke models, the same holds for the classical fragment of our system:

\begin{proposition}[Classical completeness] \label{prop:classical_completeness}
    For any $\Delta\cup \{\alpha\}\subseteq \ML:\Delta\vDash \alpha$ iff $\Delta\vdash \alpha$.
\end{proposition}

As an easy corollary, the $k$-th Hintikka formulas of $k$-bisimilar pointed models are provably equivalent.

\begin{lemma} \label{lemma:hintikka_provable_equivalence}\
\begin{enumerate}
    \item[(i)] 
        \makeatletter\def\@currentlabel{(i)}\makeatother \label{lemma:hintikka_provable_equivalence_world} If $w \bisim_k w'$, then $ \chi^{k}_{w}\proveq\chi^{k}_{w'}$.
    
    \item[(ii)] 
        \makeatletter\def\@currentlabel{(ii)}\makeatother \label{lemma:hintikka_provable_equivalence_state} If $s \bisim_k s'$, then $\chi^k_{s} \proveq \chi^k_{s'}$.
\end{enumerate}
\end{lemma}
\begin{proof} For \ref{lemma:hintikka_provable_equivalence_world}, we have $\chi^k_w\equiv \chi^k_{w'}$ by Theorem \ref{theorem:hintikka_bisimulation}. Thus $\chi^k_w\proveq \chi^k_{w'}$ follows from Proposition \ref{prop:classical_completeness}. Item \ref{lemma:hintikka_provable_equivalence_state} follows from \ref{lemma:hintikka_provable_equivalence_world} and Proposition \ref{prop:classical_completeness}.
\end{proof}

Note that, strictly speaking, Hintikka formulas are only guaranteed to be defined for finite pointed models; for an infinite pointed model, we in effect choose some finite $k$-bisimilar pointed model and treat the $k$-th Hintikka formula of this finite model as that of the infinite model. The lemma above ensures that our choice of finite representative does not matter proof-theoretically and hence that our use of these representatives in this section is admissible. It follows from results we show that similar provable equivalence results hold for all the characteristic formulas we make use of, and hence that in all cases the use of these representatives is admissible. We now show that the strong Hintikka formulas of two bisimilar pointed models are likewise provably equivalent.

\begin{lemma} \label{lemma:charf_provable_equivalence}
   If  $s \bisim_k s'$, then  $\theta^{k}_{s}\proveq\theta^{k}_{s'}$.
\end{lemma}
\begin{proof}
Suppose $s\bisim_k s'$. The two directions are symmetric; we only give the detailed proof for $\theta_s^k\vdash\theta_{s'}^k$, i.e., $\bigvee_{w\in s}(\chi_w^k\wedge \NE)\vdash \bigvee_{w'\in s'}(\chi_{w'}^k\wedge \NE)$.  
If $s=\emptyset$, then clearly $s'=\emptyset$, so that $\theta^k_s=\theta^k_{s'}=\bot$ and $\theta^k_{s}\vdash\theta^k_{s'}$.
 Otherwise if $s\neq \emptyset$, then for any $w\in s$, there exists by assumption a (nonempty) substate $s'_w\subseteq s'$ such that $w\bisim_k w'$ for all $w'\in s'_w$. We may assume that $s'_w$ is the maximal such substate. By Lemma \ref{lemma:hintikka_provable_equivalence}, we have that for any $w'\in s'_w$, $\chi_w^k\vdash \chi_{w'}^k$ implying $\chi_w^k\wedge\NE\vdash \chi_{w'}^k\wedge\NE$. Now, by repeatedly applying the rule $\dis \R{W}$, we obtain $\chi_{w'}^k\wedge\NE\vdash \bigvee_{u\in s'_w}(\chi_{w'}^k\wedge\NE)$. Since $w'$ is also $k$-bisimilar to every element $u\in s'_w$, we have by Lemma \ref{lemma:hintikka_provable_equivalence} again that $\chi_{w'}^k\wedge\NE\vdash \chi_{u}^k\wedge\NE$. Thus, by $\dis \R{Mon}$, we obtain $ \bigvee_{u\in s'_w}(\chi_{w'}^k\wedge\NE)\vdash \bigvee_{u\in s'_w}(\chi_{u}^k\wedge\NE)$. Putting all these together, we derive $\chi_w^k\wedge \NE\vdash \bigvee_{u\in s'_w}(\chi_{u}^k\wedge\NE)$; that is, $\chi_w^k\wedge \NE\vdash\theta_{s'_w}$.

Hence, by repeatedly applying $\dis \R{Mon}$ and $\dis \R{E}$, we obtain $\bigvee_{w\in s}(\chi_w^k\wedge \NE)\vdash \bigvee_{w\in s}\theta_{s'_w}\\\vdash \theta_{\bigcup_{w\in s}s'_w}$. Observe that since  $s\bisim_k s'$, we have $s'=\bigcup_{w\in s}s'_w$, whereby $\theta_{\bigcup_{w\in s}s'_w}=\theta_{s'}$. Finally, we conclude that $\theta_s\vdash \theta_{s'}$.
\end{proof}

On the other hand, if two states are not $k$-bisimilar, their strong Hintikka formulas are contradictory:

\begin{lemma} \label{lemma:not_bisimilar_implies_charfs_incompatible} If $s\nbisim_k s'$, then $ \theta^{k}_{s},\theta^{k}_{s'}\vdash \Bot$.
\end{lemma}

To prove this fact, we need two lemmas. The first is an analogous result for Hintikka formulas for worlds:

\begin{lemma} \label{lemma:not_bisimilar_implies_hintikka_incompatible}
    If $w\nbisim_k w'$ for all $w'\in s'$, then $\bigdis_{w'\in s'} \chi^{k}_{w'} \vdash \bnot \chi^{k}_{w}$.
\end{lemma}
\begin{proof}
For any $w'\in s'$ we have $w'\nvDash \chi^k_w$ by Theorem \ref{theorem:hintikka_bisimulation} and therefore $w'\vdash \lnot \chi^k_w$ by Proposition \ref{prop:classical_completeness}. So $\bigdis_{w'\in s'} \chi^k_{w'} \vdash \bnot \chi^k_w$ by $\lor \R{E}$.
 \end{proof}
 
The second lemma states a basic fact concerning the strong contradiction.
 
\begin{lemma}\label{lemma:provability_results_Bot}
$(\alpha\land \NE)\dis \phi ,\bnot \alpha \vdash \Bot$.
\end{lemma}
\begin{proof}
We have that
        \begin{align*}
            &((\alpha\land \NE)\dis \phi) \land \bnot \alpha&&\vdash&&((\alpha\land \NE)\land \bnot \alpha) \dis (\phi\land \bnot \alpha)\tag{Prop. \ref{prop:commutativity_associativity_distributivity}} \\
            &&&\vdash&&(\bot\land \NE) \dis (\phi\land \bnot \alpha)\tag{$\lnot \R{E}$}\\
            &&&\vdash&&\Bot\tag{$\Bot \R{Ctr}$}
        \end{align*}
\end{proof}

\begin{proof}[Proof of Lemma \ref{lemma:not_bisimilar_implies_charfs_incompatible}]
    If $s\nbisim_k s'$, then either there is a $w\in s$ such that $w\nbisim_k w'$ for all $w'\in s'$, or there is a $w'\in s'$ such that $w\nbisim_k w'$ for all $w\in s$. We may w.l.o.g. assume the former.
    Then:
    \begin{align*}
            &\theta^k_s\land\theta^k_{s'}&&=&&\theta^k_s\land\underset{w'\in s'}{\bigdis}(\chi^k_{w'}\land \NE) \\
            &&&\vdash&&\theta^k_s\land\underset{w'\in s'}{\bigdis}\chi^k_{w'} \\
            &&&\vdash&&((\chi^k_w\land \NE)\dis\theta^k_{s\backslash\{w\}})\land\lnot \chi^k_w \tag{Lemma \ref{lemma:not_bisimilar_implies_hintikka_incompatible}}\\
            &&&\vdash&&\Bot \tag{Lemma \ref{lemma:provability_results_Bot}}
    \end{align*}
\end{proof}

We show also that strong Hintikka formulas $\theta_s^{\mathsf{X},k}$ are monotone with respect to the parameters $k$ and $\mathsf{X}$ in the sense of the following lemma.

\begin{lemma} \label{lemma:charf_proves_lower_md_charf}\
\begin{enumerate}
    \item[(i)] 
        \makeatletter\def\@currentlabel{(i)}\makeatother \label{lemma:charf_proves_lower_md_charf_md} If $n \leq k$, then $\theta^{k}_{s}\vdash\theta^{n}_{s}$.
    \item[(ii)] 
        \makeatletter\def\@currentlabel{(ii)}\makeatother \label{lemma:charf_proves_lower_md_charf_prop} If $\mathsf{Y}\subseteq \mathsf{X}$, then $\theta^{\mathsf{X},k}_s\vdash \theta^{\mathsf{Y},k}_s$.
\end{enumerate}
\end{lemma}
\begin{proof}
 It follows from Theorem \ref{theorem:hintikka_bisimulation} that $\chi^{k}_w\vDash \chi^{n}_w$ and $\chi^{\mathsf{X},k}_w\vDash \chi^{\mathsf{Y},k}_w$ for any $w\in s$. Therefore, $\chi^{k}_w\vdash \chi^{n}_w$ and $\chi^{\mathsf{X},k}_w\vdash \chi^{\mathsf{Y},k}_w$ by Proposition \ref{prop:classical_completeness}. Then also $\chi^{k}_w\land \NE \vdash \chi^{n}_w \land \NE$ and $\chi^{\mathsf{X},k}_w\land \NE\vdash \chi^{\mathsf{Y},k}_w\land \NE$, and the results therefore follow by repeated applications of $\dis\R{Mon}$.
\end{proof}

Our final lemma leading to the completeness proof concerns some standard properties of the disjunction elimination rule.

    \begin{lemma}\label{lemma:traceable_deduction_failure}\
    \begin{enumerate}
        \item[(i)] 
        \makeatletter\def\@currentlabel{(i)}\makeatother \label{lemma:traceable_deduction_failure_i} If $\Phi,\bigintd_{j\in J}\phi_j\nvdash \psi$, then $ \Phi,\phi_j \nvdash \psi$ for some $j\in J$.
        \item[(ii)] 
        \makeatletter\def\@currentlabel{(ii)}\makeatother\label{lemma:traceable_deduction_failure_ii} If $\{\bigintd_{j \in J_i} \phi_j\sepp i\in I\}\nvdash \psi$, then $ \{\phi_{i} \sepp i \in I\}\nvdash \psi$ for some collection $\{\phi_i \sepp i \in I\}$ such that for each $i\in I$,  $\phi_i\in\{\phi_{j}\mid j\in J_i\}$.
    \end{enumerate}
    \end{lemma}
    \begin{proof} See, e.g., \cite{ciardelliIemhoffYang}. Item \ref{lemma:traceable_deduction_failure_i} is proved easily by applying  $\intd \mathsf{E}$. 
    For \ref{lemma:traceable_deduction_failure_ii}, choose an enumeration of $I$; the result then follows by an inductive proof making use of \ref{lemma:traceable_deduction_failure_i}.
    \end{proof}

   We are now ready to prove completeness. Our proof is similar to the one given in \cite{yang2017} for the completeness of the system $\PT$, though note that the proof in \cite{yang2017} is for weak completeness; here we directly prove strong completeness.
   
\begin{theorem}[Completeness of \BSMLIBF]\label{theorem:BSMLI_completeness}
    If $\Phi \vDash \psi$, then $\Phi\vdash \psi$.
\end{theorem}
\begin{proof}
    Assume $\Phi\nvdash \psi$. We  show $\Phi\nvDash \psi$.  
    Put $\Phi=\{\phi_i\sepp i \in I\}$. For each $i\in I$,  let 
    $k_i:=max\{md(\phi_i),md(\psi)\}$, and 
$\mathsf{X}_i:=\mathsf{P}(\phi_i)\cup \mathsf{P}(\psi)$. By Lemma \ref{lemma:normal_form_provable_equivalence_BSMLI},
    \begin{align}\label{theorem:BSMLI_completeness_eq1}
    &\quad\quad\quad\phi_i \proveq \bigintdd_{s\in \PPP_i}\theta^{\mathsf{X}_i,k_i}_{s}
    \end{align}
for some properties $\PPP_i$ over $\mathsf{X}_i$. Then also $\{ \bigintd_{s\in \PPP_i}\theta^{\mathsf{X}_i,k_i}_{s}\sepp i\in I\}\nvdash \psi$. By Lemma \ref{lemma:traceable_deduction_failure}, there is a collection of formulas $\Phi'=\{\theta^{\mathsf{X}_i,k_i}_{s_i}\sepp s_i\in \PPP_i$ and $i\in I\}$ such that $\Phi'\nvdash  \psi$.

Observe that $\Phi'\nvDash \Bot$, since otherwise clearly for some $\theta^{\mathsf{X}_i,k_i}_{s_i},\theta^{\mathsf{X}_j,k_j}_{s_j}\in \Phi'$ we have $s_i\nbisim^\mathsf{M}_m s_j$ for $m=min\{k_i,k_j\}$ and $\mathsf{M}=\mathsf{X}_i\cap \mathsf{X}_j$, so that by Lemma \ref{lemma:not_bisimilar_implies_charfs_incompatible}, $\theta^{\mathsf{M},m}_{s_i},\theta^{\mathsf{M},m}_{s_j}\vdash\Bot$. Then by Lemma \ref{lemma:charf_proves_lower_md_charf}, $\theta^{\mathsf{X}_i,k_i}_{s_i},\theta^{\mathsf{X}_j,k_j}_{s_j}\vdash\Bot$, so that $\Phi'\vdash \Bot$, whence $\Phi'\vdash \psi$ by Lemma \ref{lemma:Bot_E}; a contradiction.

Thus, we can let $t$ be such that $t\vDash \Phi'$. By (\ref{theorem:BSMLI_completeness_eq1}) and soundness, we have $\Phi'\vDash \phi_i$ for each $i\in I$. Therefore $t\vDash \Phi$. To show $\Phi\nvDash\psi$, it suffices to show $t\nvDash\psi$.
Assume otherwise. Take an $i\in I$. By Lemma \ref{lemma:normal_form_provable_equivalence_BSMLI},  we have 
\begin{align}\label{theorem:BSMLI_completeness_eq2}
    &\quad\quad\quad\psi \proveq \bigintdd_{r\in \QQQ}\theta^{\mathsf{X}_i,k_i}_r
    \end{align}
for some property $\QQQ$ over $\mathsf{X}_i$. Since $t\vDash \psi$,  by (\ref{theorem:BSMLI_completeness_eq2}) and soundness, we have $t\vDash \theta^{\mathsf{X}_i,k_i}_r$ for some $r\in \QQQ$. Meanwhile, $t\vDash\Phi'$ implies $t\vDash\theta_{s_i}^{\mathsf{X}_i,k_i}$. Thus,  by Proposition \ref{prop:characteristic_formulas_characterize_states}\ref{prop:characteristic_formulas_characterize_states_theta}, we have $r \bisim^{\mathsf{X}_i}_{k_i} t\bisim^{\mathsf{X}_i}_{k_i} s_i$, whence $\theta^{\mathsf{X}_i,k_i}_{s_i}\vdash \theta^{\mathsf{X}_i,k_i}_r$ by Lemma \ref{lemma:charf_provable_equivalence}. Then by $\intd \R{I}$ and (\ref{theorem:BSMLI_completeness_eq2}), $\theta^{\mathsf{X}_i,k_i}_{s_i}\vdash \bigintd_{r\in \QQQ}\theta^{\mathsf{X}_i,k_i}_r\vdash \psi$. But then $\Phi' \vdash \psi$, which is a contradiction.
\end{proof}

We derive the compactness of $\BSMLI$ as well as that of the weaker logics $\BSML$ $\BSMLO$ as a corollary of strong completeness.

%% file: sections/axiomatization/BSMLI/bsmli_normal_form.tex
The remainder of this section is devoted to the proof of Lemma \ref{lemma:normal_form_provable_equivalence_BSMLI} (provable equivalence of the normal form). 
We start with two technical lemmas concerning the behavior of $\NE$ in disjunctions.

\begin{lemma} \label{lemma:provability_results_NE}
$\phi \dis (\psi \land \NE) \proveq (\phi \dis (\psi\land \NE))\land \NE $.
\end{lemma}
\begin{proof}
$\dashv$ follows by $\land \R{E}$. We give a $\BSMLI$-specific derivation for $\vdash$:
        \begin{align*}
            &&&\phi \dis (\psi \land \NE) &&\\
            &\vdash&& (\phi \dis (\psi \land \NE))\land (\bot \intd \NE) \tag{$\NE\R{I}$}\\
            &\vdash&&((\phi \dis (\psi \land \NE))\land \bot) \intd ((\phi \dis (\psi \land \NE))\land\NE) \\
            &\vdash&&((\phi\land \bot )\dis ((\psi \land \NE)\land \bot)) \intd ((\phi \dis (\psi \land \NE))\land\NE) \tag{Prop. \ref{prop:commutativity_associativity_distributivity}}\\
            &\vdash&&((\phi\land \bot )\dis (\psi \land\Bot)) \intd ((\phi \dis (\psi \land \NE))\land\NE) \\
            &\vdash&& ((\phi \dis (\psi \land \NE))\land\NE)\intd ((\phi \dis (\psi \land \NE))\land\NE) \tag{$\Bot\R{Ctr}$}\\
            &\vdash&&(\phi \dis (\psi \land \NE))\land\NE
        \end{align*}
\end{proof}

\begin{lemma} \label{lemma:dis_NE_elimination}
    $\phi\lor \psi\vdash \phi\intd \psi\intd ((\phi \land \NE)\dis ((\psi \land \NE))$.
\end{lemma}
\begin{proof}
We have that
\begin{align*}
    &\phi \dis \psi &&\vdash &&(\phi \land (\bot \intd \NE))\dis \psi\tag{$\NE\R{I}$, $\dis \R{Mon}$}\\
    &&&\vdash &&((\phi \land \bot) \intd (\phi \land \NE)) \dis \psi\\
    &&&\vdash &&((\phi \land \bot) \dis \psi )\intd ((\phi \land \NE) \dis \psi)\tag{$\R{Distr}\dis\intd$}\\
    &&&\vdash && \psi \intd ((\phi \land \NE) \dis \psi)\tag{$\bot \R{E}$} \\
    &&&\vdash && \psi \intd ((\phi \land \NE) \dis (\psi \land (\bot \intd \NE)))\tag{$ \NE\R{I}$, $\dis \R{Mon}$}\\
    &&&\vdash && \psi \intd ((\phi \land \NE) \dis ((\psi \land \bot) \intd (\psi \land  \NE)))\\
    &&&\vdash && \psi \intd ((\phi \land \NE) \dis (\psi \land \bot)) \intd ((\phi \land \NE) \dis (\psi \land  \NE))\tag{$\R{Distr}\dis\intd $}\\
    &&&\vdash && \psi \intd \phi  \intd ((\phi \land \NE) \dis (\psi \land  \NE))\tag{$ \bot \R{E}$}
\end{align*}
\end{proof}

The following lemma further characterizes the interactions between $\NE$ and the disjunctions. From item \ref{lemma:BSMLI_provability_results_nf_classical} below, it follows that formulas in $\ML$-normal form $\bigdis_{(M,w)\in \llbracket \alpha\rrbracket}\chi^k_w$ can be converted into $\BSMLI$-normal form $\bigintd_{\PPP\subseteq \llbracket \alpha \rrbracket}\bigdis_{w\in \PPP}( \chi^k_w\land \NE) $.

\begin{lemma} \label{lemma:BSMLI_provability_results_nf}\
    \begin{enumerate}
        \item[(i)] 
        \makeatletter\def\@currentlabel{(i)}\makeatother \label{lemma:BSMLI_provability_results_nf_NE_general} $\displaystyle{\NE, \underset{i\in I}{\bigdis}\phi_i\vdash \underset{\emptyset\neq J\subseteq I}{\bigintdd}\underset{j\in J}{\bigdis}(\phi_j\land \NE) }$
        \item[(ii)] 
        \makeatletter\def\@currentlabel{(ii)}\makeatother \label{lemma:BSMLI_provability_results_nf_NE_classical} $\displaystyle{\NE \land \underset{i\in I}{\bigdis}\alpha_i\proveq \underset{\emptyset\neq J\subseteq I}{\bigintdd}\underset{j\in J}{\bigdis}(\alpha_j\land \NE) }$
        \item[(iii)] 
        \makeatletter\def\@currentlabel{(iii)}\makeatother \label{lemma:BSMLI_provability_results_nf_general} $\displaystyle{\underset{i\in I}{\bigdis}\phi_i\vdash \underset{J\subseteq I}{\bigintdd}\underset{j\in J}{\bigdis}(\phi_j\land \NE) }$
        \item[(iv)] 
        \makeatletter\def\@currentlabel{(iv)}\makeatother \label{lemma:BSMLI_provability_results_nf_classical} $\displaystyle{\underset{i\in I}{\bigdis}\alpha_i\proveq \underset{J\subseteq I}{\bigintdd}\underset{j\in J}{\bigdis}(\alpha_j\land \NE) }$
    \end{enumerate}
\end{lemma}
\begin{proof}
\ref{lemma:BSMLI_provability_results_nf_NE_general} By induction on $k=|I|$. If $k=0$, then $\bigdis\emptyset=\bot$, and by Lemma \ref{lemma:Bot_E} we have $\NE \land \bot\vdash \psi$ for any $\psi$. 

If $k=1$, then  $\bigdis \{\phi\}=\phi$, and  $\phi,\NE\vdash \phi \land \NE$ by $\land \R{I}$; and $\bigintd \{ \bigdis \{\phi \land \NE\}\}=\phi \land \NE$.

If $k=2$, then by Lemma \ref{lemma:dis_NE_elimination} we have $\phi \dis \psi\vdash \phi\intd \psi \intd ((\phi \land \NE) \dis (\psi \land \NE))$, so that clearly $\NE,\phi \dis \psi\vdash (\phi \land \NE)\intd (\psi \land \NE) \intd ((\phi \land \NE) \dis (\psi \land \NE))$.

         For the case $|I|=k+1$, by the induction hypothesis, \begin{equation}\label{lemma:BSMLI_provability_results_nf_eq1}
         \quad\displaystyle{\NE, \underset{i\in (I\backslash\{x\})}{\bigdis}\phi_i\vdash \underset{\emptyset\neq J\subseteq (I\backslash\{x\})}{\bigintdd}\underset{j\in J}{\bigdis}(\phi_j\land \NE) }
         \end{equation}
            where $x\in I$. Then we have that
            \begin{align*}
                &&&\NE, \underset{i\in I}{\bigdis}\phi_i&\\
                &\vdash&&\NE\land (\phi_x \dis \underset{i\in (I\backslash\{x\})}{\bigdis}\phi_i)\\
                &\vdash&&(\phi_x \land \NE)\intd (\NE \land \underset{i\in (I\backslash\{x\})}{\bigdis}\phi_i )\intd  ( (\phi_x \land \NE) \dis (\NE \land \underset{i\in (I\backslash\{x\})}{\bigdis}\phi_i ) )\tag{Case $k=2$}\\
                &\vdash&&(\phi_x \land \NE)\intd(\underset{\emptyset\neq J\subseteq (I\backslash\{x\})}{\bigintdd}\underset{j\in J}{\bigdis}(\phi_j\land \NE) )\intd  ((\phi_x \land \NE) \dis(\underset{\emptyset\neq J\subseteq (I\backslash\{x\})}{\bigintdd}\underset{j\in J}{\bigdis}(\phi_j\land \NE) )\tag{\ref{lemma:BSMLI_provability_results_nf_eq1}}\\
                &\vdash&&(\phi_x \land \NE)\intd(\underset{\emptyset\neq J\subseteq (I\backslash\{x\})}{\bigintdd}\underset{j\in J}{\bigdis}(\phi_j\land \NE)) \intd  \underset{\emptyset\neq J\subseteq (I\backslash\{x\})}{\bigintdd}( (\phi_x \land \NE) \dis \underset{j\in J}{\bigdis}(\phi_j\land \NE) )\tag{$\R{Distr}\dis \intd$}\\
                &\vdash&&\underset{\emptyset\neq J\subseteq I}{\bigintdd}\underset{j\in J}{\bigdis}(\phi_j\land \NE)
            \end{align*}
    
    \ref{lemma:BSMLI_provability_results_nf_NE_classical} $\vdash$ by \ref{lemma:BSMLI_provability_results_nf_NE_general}. $\dashv$: For any nonempty $J\subseteq I$, we have $\bigdis_{j\in J}(\alpha_j\land \NE)\vdash \NE$ by Lemma \ref{lemma:provability_results_NE}. We also have that for any $j\in J$, $\alpha_j\land \NE \vdash \bigdis_{i\in I}\alpha_i$ by $\land\R{E}$ and $\dis\R{I}$. Therefore $\bigdis_{j\in J}(\alpha_j\land \NE)\vdash \bigdis_{i\in I}\alpha_i$ by $\dis\R{E}$. And so $\bigintd_{\emptyset\neq J\subseteq I}\bigdis_{j\in J}(\alpha_j\land \NE)\vdash \NE\land \bigdis_{i\in I}\alpha_i$ by $\intd\R{E}$.
         
    \ref{lemma:BSMLI_provability_results_nf_general} We have that
        \begin{align*}
            &\underset{i\in I}{\bigdis}\phi_i&&\vdash&&(\bot\intd\NE)\land\underset{i\in I}{\bigdis}\phi_i\tag{$\NE\R{I}$}\\
            &&&\vdash&&(\bot\land\underset{i\in I}{\bigdis}\phi_i)\intd (\NE\land\underset{i\in I}{\bigdis}\phi_i)\\
            &&&\vdash&&\bot\intd \underset{\emptyset\neq J\subseteq I}{\bigintdd}\underset{j\in J}{\bigdis}(\phi_j\land \NE)\tag{\ref{lemma:BSMLI_provability_results_nf_NE_general}}\\
            &&&\vdash&&\underset{J\subseteq I}{\bigintdd}\underset{j\in J}{\bigdis}(\phi_j\land \NE)\tag{$\bigdis\emptyset=\bot$}
        \end{align*}
        
        \ref{lemma:BSMLI_provability_results_nf_classical} $\vdash$ by \ref{lemma:BSMLI_provability_results_nf_general}. $\dashv$: We have that
        \begin{align*}
            &\underset{J\subseteq I}{\bigintdd}\underset{j\in J}{\bigdis}(\alpha_j\land \NE)&&\vdash&&\bot\intd\underset{\emptyset\neq J\subseteq I}{\bigintdd}\underset{j\in J}{\bigdis}(\alpha_j\land \NE)\tag{$\bigdis\emptyset=\bot$}\\
            &&&\vdash&&\bot\intd(\NE\land \underset{i\in I}{\bigdis}\alpha_i)\tag{\ref{lemma:BSMLI_provability_results_nf_NE_classical}}\\
            &&&\vdash&& \underset{i\in I}{\bigdis}\alpha_i\tag{$\intd\R{E}$, Prop. \ref{prop:classical_completeness}}
        \end{align*}
\end{proof}

It now follows that each classical formula is equivalent to one in $\BSMLI$-normal form, which, in our inductive proof for Lemma \ref{lemma:charf_provable_equivalence}, takes care of the case in which the formula is classical.

\begin{lemma} \label{lemma:normal_form_provable_equivalence_classical}
    For each $\alpha\in \ML$, each $k \geq md(\alpha)$, and each finite $\mathsf{X}\supseteq \mathsf{P}(\alpha)$, there is some property $\PPP$ over $\mathsf{X}$ such that: $$\alpha \proveq \bigintdd_{s\in \PPP}\theta^{\mathsf{X,k}}_s.$$
\end{lemma}
\begin{proof}
    $\llbracket\alpha \rrbracket$ (over $\mathsf{X}$) is invariant under $k$-bisimulation so by the proof of Theorem \ref{theorem:world-based_expressive_completeness}, we have $\alpha \equiv \bigdis_{w\in \llbracket \alpha \rrbracket}\chi^k_w$. Then by Proposition \ref{prop:classical_completeness}, $\alpha \proveq \bigdis_{w\in \llbracket \alpha \rrbracket}\chi^k_w$, so that by Lemma \ref{lemma:BSMLI_provability_results_nf} \ref{lemma:BSMLI_provability_results_nf_classical}: 
          $$\alpha \proveq \bigdis_{w\in \llbracket \alpha \rrbracket}\chi^k_w\proveq\underset{\PPP\subseteq \llbracket \alpha \rrbracket}{\bigintdd}\underset{w\in \PPP}{\bigdis}( \chi^k_w\land \NE)= \bigintdd_{s\in \wp (\llbracket \alpha \rrbracket)}\underset{w\in s}{\bigdis}\theta^k_s ,$$
          where $\wp (\llbracket \alpha \rrbracket)$ is the state property $\{\biguplus \{(M,\{w\})\sepp (M,w)\in \PPP\} \sepp \PPP\subseteq \llbracket\alpha \rrbracket\}$.
\end{proof}

Now, for the modality cases in our inductive proof for Lemma \ref{lemma:charf_provable_equivalence}, we need to show normal form provable equivalence for formulas of the  forms $\Di \xi^k_\PPP$ and $\Bo \xi^k_\PPP$. These formulas are flat, so by Proposition \ref{prop:ML_expressive_completeness} we have $\Di \xi^k_\PPP\equiv \alpha$ and $\Bo \xi^k_\PPP\equiv\beta$ for some $\alpha,\beta\in \ML$. Given Lemma \ref{lemma:normal_form_provable_equivalence_classical}, it is thus sufficient to show $\Di \xi^k_\PPP\proveq \alpha$ and $\Bo \xi^k_\PPP\proveq\beta$; that is, the modality cases can be reduced to the case for classical formulas. As will become clear later, this  further reduces to showing that formulas of the form $\Di \theta^k_s$ or $\Bo\theta^k_s$ are provably equivalent to classical formulas.  We now work towards proving this result. First, two technical lemmas in order.

\begin{lemma} \label{lemma:di_NE_I}
    $\Di\phi \vdash \Di(\phi \land \NE)$.
\end{lemma}
\begin{proof}
    We give a $\BSMLI$-specific derivation:
    \begin{align*}
    &\Di \phi &&\vdash && \Di(\phi\land (\bot \intd \NE))\tag{$\NE\R{I}$}\\
    &&&\vdash&&\Di((\phi \land \bot) \intd (\phi \land \NE))\\
    & &&\vdash && \Di(\bot \intd (\phi \land \NE))\\
    & &&\vdash && \Di\bot \dis \Di (\phi \land \NE) \tag{$\Di\intd\dis\R{Conv}$}\\
    & &&\vdash && \bot \dis \Di (\phi \land \NE) \tag{Prop. \ref{prop:classical_completeness}, $\dis \R{Mon}$}\\
    & &&\vdash && \Di (\phi \land \NE)\tag{$\bot\R{E}$}
    \end{align*}
\end{proof}

\begin{lemma} \label{lemma:fc} 
    For all $\phi$ which do not contain $\intd$ and for all $\psi$:
    \begin{enumerate}
        \item[(i)] 
        \makeatletter\def\@currentlabel{(i)}\makeatother \label{lemma:fc_Di}  $\Di ((\phi\land \NE)\dis \psi)\proveq \Di(\phi \dis \psi) \land \Di \phi$
        \item[(ii)] 
        \makeatletter\def\@currentlabel{(ii)}\makeatother \label{lemma:fc_Bo} $\Bo ((\phi\land \NE)\dis \psi )\proveq  \Bo(\phi\dis \psi) \land \Di \phi$
    \end{enumerate}
\end{lemma}
\begin{proof}
    \ref{lemma:fc_Di} $\vdash$: By $\Di\R{Sep}$. $\dashv$: We have that
    \begin{align*}
        &\Di (\phi\dis \psi) \land \Di \phi &&\vdash&&\Di (\phi\dis \psi) \land \Di(\phi\land \NE)&\tag{Lemma \ref{lemma:di_NE_I}}\\
        &&&\vdash&&\Di (\phi \dis (\phi \land \NE)\dis \psi)\tag{$\Di\R{Join}$}\\
        &&&\vdash&&\Di (((\phi\dis(\phi \land \NE)) \land \NE)\dis\psi)\tag{Lemma \ref{lemma:provability_results_NE}, $\dis \R{Mon}$}\\
        &&&\vdash&&\Di ((\phi \land \NE)\dis\psi)\tag{$\dis \R{E}$}
    \end{align*}
    \ref{lemma:fc_Bo} $\vdash$: We have that
    \begin{align*}
        &\Bo ((\phi \land \NE)\dis \psi)&&\vdash&&\Bo (\phi \dis \psi )\land\Bo ((\phi \land \NE)\dis \psi) \\
        &&&\vdash&&\Bo (\phi \dis \psi )\land\Bo (((\phi \land \NE)\dis \psi)\land\NE) \tag{Lemma \ref{lemma:provability_results_NE}}\\
        &&&\vdash&&\Bo (\phi \dis \psi )\land\Di ((\phi \land \NE)\dis \psi) \tag{$\Bo\R{Inst}$}\\
        &&&\vdash&&\Bo (\phi \dis \psi )\land\Di \phi \tag{$\Di\R{Sep}$}
    \end{align*}
    
    $\dashv$: Similar to $\dashv$ of \ref{lemma:fc_Di}, using $\Bo\Di \R{Join}$ instead of $\Di \R{Join}$.
\end{proof}

We now show that formulas of the form $\Di \theta^k_s$ or $\Bo\theta^k_s$ are provably equivalent to classical formulas.

\noindent
\begin{minipage}{\textwidth}
\begin{lemma} \label{lemma:modal_charfs_provably_classical}\
    \begin{enumerate}
        \item[(i)] 
        \makeatletter\def\@currentlabel{(i)}\makeatother \label{lemma:modal_charfs_provably_classical_Di}  $\displaystyle\Di\theta^{k}_{s}\proveq \Di \chi^k_s \land \bigwedge_{w \in s}\Di \chi^{k}_{w}$
        \item[(ii)] 
        \makeatletter\def\@currentlabel{(ii)}\makeatother \label{lemma:modal_charfs_provably_classical_Bo}  $\displaystyle\Bo\theta^{k}_{s}\proveq \Bo \chi^{k}_{s} \land \bigwedge_{w \in s}\Di \chi^{k}_{w}$
    \end{enumerate}
\end{lemma}
\end{minipage}
\begin{proof}
        \ref{lemma:modal_charfs_provably_classical_Di}
        If $s=\emptyset$ then $\Di\theta^{k}_{s}=\Di\bot =\Di\chi^k_s \proveq \Di \chi^k_s \land \Top = \Di \chi^k_s \land \bigwedge_{w \in s}\Di \chi^{k}_{w}$.
                Otherwise if $s\neq \emptyset$ let $\theta^k_{s}=(\chi^k_{w_1}\land \NE) \dis \ldots \dis (\chi^k_{w_n}\land \NE)$. Then by Lemma \ref{lemma:fc} \ref{lemma:fc_Di}, 
        \begin{align*}
            &\Di \theta^k_{s} &&= &&\Di\underset{w\in s}{\bigdis}(\chi^k_w\land \NE)\\
            &&&\proveq &&\Di(\bigdis^{n-1}_{i=1} (\chi^k_{w_{i}}\land \NE)\dis \chi^k_{w_n})\land \Di\chi^k_{w_n}
            \\
            &&&\proveq &&\Di(\bigdis^{n-2}_{i=1} (\chi^k_{w_{i}}\land \NE)\dis \chi^k_{w_{n-1}}\dis \chi^k_{w_n})\land \Di\chi^k_{w_{n-1}}\land \Di\chi^k_{w_n}
            \\
            &&&&&\vdots\\
            &&&\proveq &&\Di \chi^k_{s} \land\bigwedge_{w\in s}\Di\chi^k_w
        \end{align*}
        \ref{lemma:modal_charfs_provably_classical_Bo} Similar to \ref{lemma:modal_charfs_provably_classical_Di}, using Lemma \ref{lemma:fc} \ref{lemma:fc_Bo} instead of \ref{lemma:fc} \ref{lemma:fc_Di}.
\end{proof}

We are finally ready to prove Lemma \ref{lemma:normal_form_provable_equivalence_BSMLI}, the normal form lemma.

\begin{proof}[Proof of Lemma \ref{lemma:normal_form_provable_equivalence_BSMLI}]
    By induction on the complexity of $\phi$. By Proposition \ref{prop:negation_normal_form_provable_BSMLI}, we may assume $\phi$ is in negation normal form. Let $k\geq md(\phi)$.

If $\phi =p$ or $\phi=\lnot p$, we apply Lemma \ref{lemma:normal_form_provable_equivalence_classical}.
     If $\phi =\lnot \NE$, we have $\lnot \NE \proveq \bot$ by $\NE\R{E}$, and the result then follows by Lemma \ref{lemma:normal_form_provable_equivalence_classical}.
        If $\phi=\NE$, we have that
        \begin{align*}
            &\NE &&\proveq &&\NE \land  ( p \dis \lnot p)  \tag{Prop. \ref{prop:classical_completeness}}\\
            &&&\proveq &&\NE\land \underset{s\in \PPP}{\bigintdd} \theta^k_s \tag{Lemma \ref{lemma:normal_form_provable_equivalence_classical}}\\
            &&&\proveq &&\NE \land (\bot \intd \underset{s\neq \emptyset\in  \PPP}{\bigintdd} \theta^k_s) \\
            &&&\proveq &&(\NE\land \bot) \intd \underset{s\neq \emptyset\in  \PPP}{\bigintdd} \theta^k_s \tag{Prop. \ref{prop:commutativity_associativity_distributivity}, Lemma \ref{lemma:provability_results_NE}}\\
            &&&\proveq &&\underset{s\in \PPP}{\bigintdd} \theta^k_s &\tag{Lemma \ref{lemma:Bot_E}}
        \end{align*}

If $\phi =\psi \land \chi$, then $k\geq md(\phi)=max\{md(\psi),md(\chi)\}$, so by induction hypothesis there are $\PPP$ and $\QQQ$ such that $\psi\proveq\bigintd_{s\in \PPP} \theta^k_s$ and $\chi\proveq\bigintd_{t\in \QQQ} \theta^k_{t}$. Let $\RRR:=\{r\sepp \exists s\in \PPP,\exists t\in \QQQ: s\bisim_k r\bisim_k t\}$. We have that
        \begin{align*}
            \psi \land \chi&&&\proveq&&\underset{s\in \PPP}{\bigintdd} \theta^k_s \land \underset{t\in \QQQ}{\bigintdd} \theta^k_t&\\
            &&&\proveq&&\underset{s\in \PPP}{\bigintdd} \underset{t\in \QQQ}{\bigintdd}(\theta^k_s \land  \theta^k_{t})\\
            &&&\proveq&&\underset{s\in \PPP\cap \RRR}{\bigintdd} \,\underset{t\in \QQQ\cap [s]_k}{\bigintdd}(\theta^k_s \land  \theta^k_{t})\intd \underset{s\in \PPP\cap \RRR}{\bigintdd} \,\underset{t\in \QQQ\backslash [s]_k}{\bigintdd}(\theta^k_s \land  \theta^k_{t})\intd\underset{s\in \PPP\backslash \RRR}{\bigintdd} \,\underset{t\in \QQQ}{\bigintdd}(\theta^k_s \land  \theta^k_{t})\tag{where $[s]_k$ denotes the $\bisim_k$-equivalence class of $s$}\\
            &&&\proveq&&\underset{s\in \PPP\cap \RRR}{\bigintdd} \,\underset{t\in \QQQ\cap [s]_k}{\bigintdd}(\theta^k_s \land  \theta^k_{t})\intd (\underset{s\in \PPP\cap \RRR}{\bigintdd} \,\underset{t\in \QQQ\backslash [s]_k}{\bigintdd}\Bot)\intd\underset{s\in \PPP\backslash \RRR}{\bigintdd} \,\underset{t\in \QQQ}{\bigintdd} \Bot\tag{Lemmas \ref{lemma:not_bisimilar_implies_charfs_incompatible} \& \ref{lemma:Bot_E}}\\
            &&&\proveq&&\underset{s\in \PPP\cap \RRR}{\bigintdd} \,\underset{t\in \QQQ\cap [s]_k}{\bigintdd}(\theta^k_s \land  \theta^k_{t})\tag{Lemma \ref{lemma:Bot_E}}\\
            &&&\proveq&&\underset{s\in \RRR}\bigintdd \theta^k_s\tag{Lemma \ref{lemma:charf_provable_equivalence}}
        \end{align*}

    If $\phi =\psi \dis \chi$, then $k\geq md(\phi)=max\{md(\psi),md(\chi)\}$, and by induction hypothesis we have properties $\PPP,\QQQ$ as in the conjunction case. Then:
        \begin{align*}
            &\psi \dis\chi&&\proveq&&\underset{s\in \PPP}{\bigintdd} \theta^k_s \dis \underset{t\in \QQQ}{\bigintdd} \theta^k_{t}\\
            &&&\proveq&&\underset{s\in \PPP}{\bigintdd} \underset{t\in \QQQ}{\bigintdd}(\theta^k_s \dis \theta^k_{t})\tag{Prop. \ref{prop:commutativity_associativity_distributivity}}\\
            &&&\proveq&&\underset{s\in \PPP}{\bigintdd} \underset{t\in \QQQ}{\bigintdd}\theta^k_{s\uplus t}
        \end{align*}

  If $\phi =\psi \intd \chi$, the result follows immediately by the induction hypothesis.
   
   If $\phi= \Di \psi$, then $k-1\geq md(\phi)-1= 
    md(\psi)$, so by the induction hypothesis there is a property $\PPP$ such that $\psi\proveq\bigintd_{s\in \PPP} \theta^{k-1}_s$. Thus,
        \begin{align*}
            &\Di\psi&&\proveq&&\Di \underset{s\in \PPP}{\bigintdd} \theta^{k-1}_s \\
            &&&\proveq  &&\underset{s\in \PPP}{\bigdis}\Di\theta^{k-1}_s \tag{$\mathsf{Conv}\Di\intd\dis$}\\
            &&&\proveq  &&\underset{s\in \PPP}{\bigdis}(\Di \chi^{k-1}_s \land\underset{w\in s}{\bigwedge}\Di\chi^{k-1}_w) \tag{Lemma \ref{lemma:modal_charfs_provably_classical} \ref{lemma:modal_charfs_provably_classical_Di}}
        \end{align*}
        This formula is classical and of modal depth $\leq k$, so we are done by Lemma \ref{lemma:normal_form_provable_equivalence_classical}.

    The case $\phi= \Bo \psi$ is similar to the case for $\Di \psi$, using $\mathsf{Conv}\Bo\intd\dis$ instead of $\mathsf{Conv}\Di\intd\dis$ and Lemma \ref{lemma:modal_charfs_provably_classical} \ref{lemma:modal_charfs_provably_classical_Bo} instead of Lemma \ref{lemma:modal_charfs_provably_classical}\ref{lemma:modal_charfs_provably_classical_Di}.
\end{proof}

%% file: sections/axiomatization/BSMLO/bsmlo_rules.tex
\subsection{$\BSMLOBF$} \label{section:BSMLO}

For the $\BSMLO$-system, we remove the rules concerning $\intd$ from the $\BSMLI$-system and add rules for the emptiness operator $\OC$. Recall that $\OC$ corresponds essentially to one specific type of global disjunction, as $\OC\psi \equiv \psi \intd \bot$. The introduction rules for $\OC$ capture (the right-to-left entailment in) this equivalence by imitating instances of $\intd$-introduction which yield $\psi\vdash \psi \vvee \bot$ and $\bot\vdash \psi \vvee \bot$. The elimination rules for $\OC$ capture something stronger than (the left-to-right entailment in) $\OC\psi \equiv \psi \intd \bot$: we also encode the fact that $\lor$, $\land$ and $\OC$ distribute over $\intd$---and that $\lnot$ and $\Di$ do not---directly into these rules. To that end, we call a formula occurrence $[\psi]$ \emph{$\intd\hspace{-1pt}$-distributive} in $\chi$ if $[\psi]$ is not within the scope of any $\lnot$ or $\Di$ in $\chi$ (where recall that $\Bo$ is an abbreviation of $\neg\Di\neg$). For example, $[p]$ is not $\intd\hspace{-1pt}$-distributive in $\Bo(p \lor q)=\lnot\Di\lnot(p \lor q) $ but it \emph{is} $\intd\hspace{-1pt}$-distributive in the subformula $p \lor q$. Given the pertinent equivalence and distributivity facts, we have that if $[\OC\psi]$ is $\intd$-distributive in $\phi$, then $\phi\equiv\phi[\psi\intd\bot/\OC\psi]\equiv \phi[\psi/\OC\psi]\intd\phi[\bot/\OC\psi]$ and $\Di\phi\equiv\Di\phi[\psi\intd\bot/\OC\psi]\equiv \Di(\phi[\psi/\OC\psi]\intd\phi[\bot/\OC\psi])\equiv \Di\phi[\psi/\OC\psi]\vee\Di\phi[\bot/\OC\psi]$ (and similarly for $\Bo$). The elimination rules for $\OC$ capture (the left-to-right entailments in) these equivalences by imitating the $\intd$-rules as applied to formulas of the form $ \phi[\psi/\OC\psi]\intd\phi[\bot/\OC\psi]$ and $\Di/\Bo(\phi[\psi/\OC\psi]\intd\phi[\bot/\OC\psi])$.

\begin{definition}[Natural deduction system for $\BSMLOBF$] \label{def:BSMLO_system}
The natural deduction system for $\BSMLO$ includes all rules not involving $\intd$ from the system for $\BSMLI$ (boxes (a)--(f)) and the following rules for $\OC$:
    
       \begin{proofbox}[]
       {\small

        \begin{minipage}{.33\textwidth}
         \begin{prooftree}
            \AxiomC{}
            \RightLabel{$\OC\NE\R{I}$}
                \UnaryInfC{$\OC  \NE$}
        \end{prooftree}
        \end{minipage}
        \begin{minipage}{.33\textwidth}
        \begin{prooftree}
                \AxiomC{$D $}
                \noLine
            \UnaryInfC{$\bot$}
            \RightLabel{$\OC\R{I}$}
                \UnaryInfC{$\OC \phi$}
        \end{prooftree}
        \end{minipage}
        \begin{minipage}{.33\textwidth}
        \begin{prooftree}
                \AxiomC{$D $}
                \noLine
            \UnaryInfC{$\phi$}
            \RightLabel{$\OC\R{I}$}
                \UnaryInfC{$\OC \phi$}
        \end{prooftree}
        \end{minipage}
        \vspace{0.2cm}

                \begin{minipage}{.66\textwidth}
                  \begin{prooftree}
                \AxiomC{$D $}
                \noLine
            \UnaryInfC{$\phi$}
                \AxiomC{$[\phi[\psi /\OC\psi]] $}
                \noLine
                \UnaryInfC{$ D_1 $}
                \noLine
            \UnaryInfC{$ \chi$}
                \AxiomC{$[\phi[\bot /\OC\psi]]$}
                \noLine
            \UnaryInfC{$ D_2 $}
                \noLine
            \UnaryInfC{$ \chi$}
            \RightLabel{$\OC\R{E} (*)$}
            \TrinaryInfC{$ \chi $}
        \end{prooftree}
        \end{minipage}
                \begin{minipage}{.33\textwidth}
        \begin{prooftree}
                \AxiomC{$D $}
                \noLine
            \UnaryInfC{$\bnot\OC\phi$}
            \doubleLine
            \RightLabel{$\bnot\OC\R{E}$}
                \UnaryInfC{$\bnot\phi$}
        \end{prooftree}
        \end{minipage}
        \vspace{0.2cm}
            
          \begin{minipage}{.5\textwidth}
        \begin{prooftree}
                \AxiomC{$D $}
                \noLine
            \UnaryInfC{$\Di\phi$}
            \RightLabel{$\Di\OC\R{E} (*)$}
            \UnaryInfC{$ \Di\phi[\psi /\OC\psi]\dis \Di \phi[\bot /\OC\psi] $}
        \end{prooftree}
        \end{minipage}
          \begin{minipage}{.5\textwidth}
                  \begin{prooftree}
                \AxiomC{$D $}
                \noLine
            \UnaryInfC{$\Bo\phi$}
            \RightLabel{$\Bo\OC\R{E} (*)$}
            \UnaryInfC{$ \Bo\phi[\psi /\OC\psi]\dis \Bo \phi[\bot /\OC\psi] $}
        \end{prooftree}
        \end{minipage}
        \vspace{0.2cm}
        
        }
        {\footnotesize
    
        $(*)$ $[\OC\psi]$ is $\intd\hspace{-1pt}$-distributive in $\phi$.}
    \end{proofbox}
\end{definition}
The introduction rules $\OC\NE\R{I}$ and $\OC\R{I}$ imitate the $\BSMLI$-rules $\NE\R{I}$ and $\intd\R{I}$, respectively. The elimination rule $\OC\R{E}$ imitates the rule $\intd\R{E}$; and $\Di\OC\R{E}$ and $\Bo\OC\R{E}$ imitate $\Di\intd\vee\R{Conv}$ and $\Bo\intd\vee\R{Conv}$, respectively. The antisupport clause for $\OC\phi$ is characterized by the elimination rule $\lnot \OC\R{E}$. The soundness of these new rules follows from the equivalences noted above.

\begin{theorem}[Soundness of $\BSMLOBF$] \label{theorem:soundness_BSMLO}
If $\Phi \vdash \psi$, then $ \Phi \vDash \psi$.
\end{theorem}

To provide a simple demonstration of the new rules, we derive the useful fact that $\OC$ cancels out the effects of appending $\land\NE$ to classical formulas:

\begin{lemma} \label{lemma:alpha_proveq_OC_alpha_and_NE}
    $\alpha \proveq \OC(\alpha \land \NE)$.
\end{lemma}
\begin{proof}
$\dashv$: $\bot \vdash \alpha $ by $\lnot \R{I}$ and $\lnot \R{E}$, and $\alpha \land \NE\vdash \alpha$ so $\OC(\alpha \land \NE)\vdash \alpha$ by $\OC\R{E}$. $\vdash$: $\alpha\vdash\alpha\land  \OC \NE$ by $\OC\NE\R{I}$. We have $ \alpha\land \OC \NE[\NE/\OC\NE]= \alpha\land\NE$, and $\alpha\land \NE \vdash \OC(\alpha \land \NE)$ by $\OC\R{I}$. On the other hand, $\alpha\land \OC \NE[\bot/\OC\NE]=\alpha\land \bot$, and $ \alpha \land \bot \vdash \bot \vdash \OC(\alpha \land \NE)$ by $\OC\R{I}$. Therefore $\OC \NE\land\alpha\vdash \OC(\alpha\land \NE)$ by $\OC\R{E}$.
\end{proof}

%% file: sections/axiomatization/BSMLO/bsmlo_completeness.tex
Our strategy for proving completeness is similar to that used in Section \ref{section:BSMLI}. We again require a crucial normal form provable equivalence result. This is stated in the next lemma; the proof is once more withheld until the end of the section.

\begin{lemma} \label{lemma:normal_form_provable_equivalence_BSMLO}
    For each $\phi\in \BSMLO$ and each $k \geq md(\phi)$, there is some property $\PPP$ such that
    \begin{align*}
        \phi \proveq \bigdis_{s\in \PPP}\OC\theta^k_s &&\text{or}&&\phi \proveq \NE \land \bigdis_{s\in \PPP}\OC\theta^k_s
    \end{align*}
\end{lemma}

A number of lemmas which were proved specifically for $\BSMLI$ in Section \ref{section:BSMLI} also hold for $\BSMLO$, and their $\BSMLO$-proofs are similar. In particular, Lemma \ref{lemma:provable_subformula_replacement} (Replacement) follows by a similar inductive argument, where the new case $\theta=\OC\eta$ is proved by applying $\OC\R{E}$ and $\OC\R{I}$. Proposition \ref{prop:negation_normal_form_provable_BSMLI} (Negation normal form) is proved by a similar argument, using the additional equivalence $\lnot \OC \phi \proveq \lnot \phi$ (by $\lnot \OC \R{E}$). 

Now, observe that a formula in $\BSMLO$-normal form can be converted into an equivalent formula in the normal form of the expressively stronger logic $\BSMLI$, as stated in the following lemma (we omit the easy proof).

\begin{lemma} \label{lemma:normal_form_conversion}
$\displaystyle\bigdis_{s\in \PPP}\OC\theta^k_s\equiv \bigintdd_{\QQQ\subseteq \PPP}\theta^k_{\biguplus Q} $.
\end{lemma}

While the global disjunction $\intd$ is not in the language of $\BSMLO$, we show in the next lemma that the formulas $\bigdis_{s\in \PPP}\OC\theta^k_s$ and $\bigintd_{\QQQ\subseteq \PPP}\theta^k_{\biguplus Q}$ are also proof-theoretically interchangeable in the sense that if rules for both $\OC$ and $\intd$ were available, the two formulas would derive the same consequences and be derivable from the same premises. Item \ref{lemma:tensor_normal_form_normal_form_provable_equivalence_i} below corresponds to the entailment $\bigintd_{\QQQ\subseteq \PPP}\theta^k_{\biguplus Q}\vDash \bigdis_{s\in \PPP}\OC\theta^k_s $, and item \ref{lemma:tensor_normal_form_normal_form_provable_equivalence_ii} simulates the other direction $\bigdis_{s\in \PPP}\OC\theta^k_s \vDash \bigintd_{\QQQ\subseteq \PPP}\theta^k_{\biguplus Q}$.

\begin{lemma} \label{lemma:tensor_normal_form_normal_form_provable_equivalence}\
    \begin{enumerate}
        \item[(i)] 
        \makeatletter\def\@currentlabel{(i)}\makeatother\label{lemma:tensor_normal_form_normal_form_provable_equivalence_i} For any $\displaystyle\QQQ\subseteq \PPP$, we have $ \theta^k_{\biguplus \QQQ}\vdash\bigdis_{s\in \PPP} \OC \theta^{k}_{s}$.
        \item[(ii)]         \makeatletter\def\@currentlabel{(ii)}\makeatother\label{lemma:tensor_normal_form_normal_form_provable_equivalence_ii}  Suppose $\displaystyle\bigdis_{s\in \PPP} \OC \theta^{k}_{s}$ is $\intd\hspace{-1pt}$-distributive in $\phi$. If 
        $\displaystyle\Phi, \phi[\theta^k_{\biguplus \QQQ}/\bigdis_{s\in \PPP} \OC \theta^{k}_{s}]\vdash \psi$ for all $\QQQ\subseteq \PPP$, then $\Phi,\phi\vdash \psi$.
    \end{enumerate}
\end{lemma}
\begin{proof}
    \ref{lemma:tensor_normal_form_normal_form_provable_equivalence_i}        Clearly, $\theta^{k}_{\biguplus \QQQ}\proveq \bigdis_{s\in \QQQ}\theta^{k}_{s}$. Next, we derive  $\bigdis_{s\in \QQQ}\theta^{k}_{s}\vdash \bigdis_{s\in \QQQ}\OC\theta^{k}_{s}$ by $\OC\R{I}$ and $\dis \R{Mon}$. Since $ \bigdis_{s\in \QQQ}\OC\theta^{k}_{s}\vdash \bigdis_{s\in \QQQ}\OC\theta^{k}_{s}\vee\bigdis_{t\in \PPP\backslash\QQQ}\bot$ (by $\lor \R{I}$), we obtain $ \bigdis_{s\in \QQQ}\OC\theta^{k}_{s}\vdash \bigdis_{s\in \QQQ}\OC\theta^{k}_{s}\vee\bigdis_{t\in \PPP\backslash\QQQ}\OC\theta_t^k$ by $\OC\R{I}$---that is, $ \bigdis_{s\in \QQQ}\OC\theta^{k}_{s}\vdash \bigdis_{s\in \PPP}\OC\theta^{k}_{s}$. Putting all these together, we conclude $\theta^{k}_{\biguplus \QQQ}\vdash \bigdis_{s\in \PPP}\OC\theta^{k}_{s}$.

         \ref{lemma:tensor_normal_form_normal_form_provable_equivalence_ii} Let $\PPP=\{s_1,\dots,s_n\}$, and thus $\zeta^k_\PPP=\bigdis_{s\in \PPP}\OC\theta^{k}_{s}=\OC\theta^k_{s_1}\lor \ldots\lor \OC\theta^k_{s_n}$. Observe that for every $\QQQ\subseteq\PPP$, we have $\theta_{\biguplus \QQQ}^k\proveq\bigvee_{s_i\in \QQQ}\theta_{s_i}^k\proveq\bigvee_{i=1}^n\tau_i$, where $\tau_i=\theta_{s_i}^k$ if $s_i\in \QQQ$, and $\tau_i=\bot$ if $s_i\notin \QQQ$, by $\dis\R{I}$ and $\bot \R{E}$. Now, by assumption, we have $\Phi,\phi[\theta_{\biguplus\emptyset}^k/\zeta^k_\PPP]\vdash\psi$ and $\Phi,\phi[\theta_{\biguplus\{s_n\}}^k/\zeta^k_\PPP]\vdash\psi$. Since $\theta_{\biguplus\emptyset}^k\proveq\bigvee_{i=1}^n\bot$ and $\theta_{\biguplus\{s_n\}}^k\proveq(\bigvee_{i=1}^{n-1}\bot)\vee\theta_{s_n}^k$, we obtain
\[\Phi,\phi[(\bigvee_{i=1}^{n-2}\bot)\vee\bot\vee\bot/\zeta^k_{\PPP}]\vdash\psi\text{ and }\Phi,\phi[(\bigvee_{i=1}^{n-2}\bot)\vee\bot\vee\theta_{s_n}^k/\zeta^k_{\PPP}]\vdash\psi.
\]
         Then, by $\OC\R{E}$, we obtain    \begin{equation}\label{lemma:tensor_normal_form_normal_form_provable_equivalence_eq1}
         \Phi,\phi[(\bigvee_{i=1}^{n-2}\bot)\vee\bot\vee\OC\theta_{s_n}^k/\zeta^k_{\PPP}]\vdash\psi.
                  \end{equation}
         Next, by assumption again,  $\Phi,\phi[\theta^k_{\biguplus \{s_{n-1}\}}/\zeta_\PPP^k]\vdash\psi$ and $\Phi,\phi[\theta^k_{\biguplus \{s_{n-1},s_n\}}/\zeta_\PPP^k]\vdash\psi$, which imply
         \[\Phi,\phi[(\bigvee_{i=1}^{n-2}\bot)\vee\theta_{s_{n-1}}^k\vee\bot/\zeta^k_{\PPP}]\vdash\psi\text{ and }\Phi,\phi[(\bigvee_{i=1}^{n-2}\bot)\vee\theta_{s_{n-1}}^k\vee\theta_{s_n}^k/\zeta^k_{\PPP}]\vdash\psi.\]
         Again, by $\OC\R{E}$, we obtain \begin{equation}\label{lemma:tensor_normal_form_normal_form_provable_equivalence_eq2}
         \Phi,\phi[(\bigvee_{i=1}^{n-2}\bot)\vee\theta_{s_{n-1}}^k\vee\OC\theta_{s_n}^k/\zeta^k_{\PPP}]\vdash\psi.
         \end{equation}
         Applying $\OC\R{E}$ to (\ref{lemma:tensor_normal_form_normal_form_provable_equivalence_eq1}) and (\ref{lemma:tensor_normal_form_normal_form_provable_equivalence_eq2}), we obtain $\Phi,\phi[(\bigvee_{i=1}^{n-2}\bot)\vee\OC\theta_{s_{n-1}}^k\vee\OC\theta_{s_n}^k/\zeta^k_{\PPP}]\vdash\psi$. Repeating this argument, we eventually get $\Phi,\phi\vdash  \psi$, as required.
    \end{proof}

Given this proof-theoretic correspondence between the normal forms, we can show completeness for $\BSMLO$ using an argument that is similar to that used for $\BSMLI$. This time we first establish weak completeness, and then derive the strong completeness theorem as a corollary of weak completeness and compactness. (The strategy we employed for the direct proof of the strong completeness of $\BSMLI$ does not work for $\BSMLO$ due to the structure of the $\BSMLO$-normal form.)
    
\begin{theorem}[Weak completeness of $\BSMLOBF$] \label{theorem:BSMLO_weak_completeness}
If $\phi\vDash \psi$, then $\phi\vdash \psi$.
\end{theorem}
\begin{proof}
    Let $k=max\{md(\phi),md(\psi)\}$. By Lemma \ref{lemma:normal_form_provable_equivalence_BSMLO},
        \begin{align}\label{theorem:BSMLO_weak_completeness_eq1}
        &&& \phi \proveq \eta_\phi \land \bigdis_{s\in \PPP}\OC\theta^{k}_s &&\text{and}&&\psi \proveq \eta_\psi\land \bigdis_{t\in \QQQ}\OC\theta^{k}_{t}, 
    \end{align}
    where each of $\eta_\phi$ and $\eta_\psi$ is either $\NE$ or (for notational convenience) $\Top$. Since $\phi \vDash \psi$, $\eta_\phi=\Top$ implies $\eta_\psi=\Top$. Each of the remaining possibilities---that is, (a) $\eta_\phi=\eta_\psi=\Top$; (b) $\eta_\phi=\eta_\psi=\NE$; (c) $\eta_\phi=\NE$ and $\eta_\psi=\Top$---implies by (\ref{theorem:BSMLO_weak_completeness_eq1}) and soundness,
    $$\bigdis_{s\in \PPP}\OC\theta^{k}_{s}\vDash  \bigdis_{t\in \QQQ}\OC\theta^{k}_{t},$$
    so that by Lemma \ref{lemma:normal_form_conversion},
\begin{equation}\label{theorem:BSMLO_weak_completeness_eq2} \quad\bigintdd_{\PPP'\subseteq\PPP}\theta^{k}_{\biguplus \PPP'}\vDash\bigintdd_{\QQQ'\subseteq \QQQ}\theta^{k}_{\biguplus \QQQ'}.\end{equation}
    Let $\PPP'\subseteq \PPP$, and $u\vDash \theta^{k}_{\biguplus \PPP'}$. By (\ref{theorem:BSMLO_weak_completeness_eq2}), we have $u\vDash \bigintd_{\QQQ'\subseteq \QQQ}\theta^{k}_{\biguplus \QQQ'}$. Thus, $u\vDash \theta^{k}_{\biguplus \QQQ'}$ for some $\QQQ'\subseteq \QQQ$. By Proposition \ref{prop:characteristic_formulas_characterize_states}\ref{prop:characteristic_formulas_characterize_states_theta}, $\biguplus \PPP'\bisim_ku\bisim_k \biguplus \QQQ'$.
    Then by Lemma \ref{lemma:charf_provable_equivalence}, $\theta^k_{\biguplus \PPP'}\vdash \theta^k_{\biguplus \QQQ'}$. By Lemma \ref{lemma:tensor_normal_form_normal_form_provable_equivalence}\ref{lemma:tensor_normal_form_normal_form_provable_equivalence_i}, $\theta^k_{\biguplus \QQQ'}\vdash \bigdis_{t\in \QQQ} \OC \theta^{k}_{t}$ , so also $\theta^k_{\biguplus \PPP'}\vdash \bigdis_{t\in \QQQ} \OC \theta^{k}_{t}$. Hence, by Lemma \ref{lemma:tensor_normal_form_normal_form_provable_equivalence}\ref{lemma:tensor_normal_form_normal_form_provable_equivalence_ii}, $\bigdis_{s\in \PPP} \OC \theta^{k}_{s}\vdash \bigdis_{t\in \QQQ} \OC \theta^{k}_{t}$, and therefore $\phi\vdash \psi$.
\end{proof}

\begin{corollary}[Strong completeness of $\BSMLOBF$] \label{coro:BSMLO_strong_completeness}
If $\Phi\vDash \psi$, then $\Phi\vdash \psi$.
\end{corollary}
\begin{proof}
    By Theorem \ref{theorem:BSMLO_weak_completeness}  and compactness.
\end{proof}

%% file: sections/axiomatization/BSMLO/bsmlo_normal_form.tex
In the remainder of this section we give the inductive proof of Lemma \ref{lemma:normal_form_provable_equivalence_BSMLO} (normal form provable equivalence). As in Section \ref{section:BSMLI}, we first show normal form provable equivalence for classical formulas.

\begin{lemma} \label{lemma:normal_form_provable_equivalence_BSMLO_classical}
    For each $\alpha\in \ML$ and each $k \geq md(\alpha)$, there is some property $\PPP$ such that
    $$\alpha \proveq \bigdis_{s\in \PPP}\OC\theta^k_s.$$
\end{lemma}
\begin{proof}
    Since $\llbracket\alpha \rrbracket$ is invariant under $k$-bisimulation, by the proof of Theorem \ref{theorem:world-based_expressive_completeness}, we have $\alpha \equiv \bigdis_{w\in \llbracket \alpha \rrbracket}\chi^k_w$. Then by Proposition \ref{prop:classical_completeness}, $\alpha \proveq \bigdis_{w\in \llbracket \alpha \rrbracket}\chi^k_w$, and by Lemma \ref{lemma:alpha_proveq_OC_alpha_and_NE} and $\dis \R{Mon}$, $$\bigdis_{w\in \llbracket \alpha \rrbracket}\chi^k_w \proveq \bigdis_{w\in \llbracket \alpha \rrbracket} \OC(\chi^k_w \land \NE)=\bigdis_{\{w\}\in \wp(\llbracket \alpha\rrbracket)}\OC \theta^k_{\{w\}}.\qedhere$$
\end{proof}

Again following the argument in Section \ref{section:BSMLI}, in our inductive proof of Lemma \ref{lemma:normal_form_provable_equivalence_BSMLO}, another important step consists in showing that formulas of the form $\Di\theta^k_s$ or $\Bo\theta^k_s$ are provably equivalent to classical formulas. In Section \ref{section:BSMLI}, this was proved in Lemma \ref{lemma:modal_charfs_provably_classical}, which made use of Lemma \ref{lemma:fc}. The two lemmas leading to Lemma \ref{lemma:fc} (namely, Lemmas \ref{lemma:provability_results_NE} and \ref{lemma:di_NE_I}) were given, in Section \ref{section:BSMLI}, only $\BSMLI$-specific derivations. We now show that these two lemmas can also be derived in the system for $\BSMLO$; we may then also make use of Lemmas \ref{lemma:fc} and \ref{lemma:modal_charfs_provably_classical} in  $\BSMLO$.

\begin{proof}[Proof of Lemma \ref{lemma:provability_results_NE} (in $\BSMLOBF$)] We show the non-trivial direction $\phi \dis (\psi \land \NE) \vdash\\ (\phi \dis (\psi\land \NE))\land \NE$.
First, $\phi \dis (\psi \land \NE) \vdash (\phi \dis (\psi \land \NE))\land \OC\NE$ by $\OC\NE\R{I}$. Then, $(\phi \dis (\psi \land \NE))\land \bot \vdash (\phi \land \bot )\lor (\psi \land \NE\land \bot)$ by Prop. \ref{prop:commutativity_associativity_distributivity} and $(\phi \land \bot )\lor (\psi \land \NE\land \bot) \vdash\\ (\phi \land \bot )\lor \Bot\vdash (\phi \dis (\psi\land \NE))\land \NE$ by $\Bot\R{Ctr}$. Obviously, $(\phi \dis (\psi\land \NE))\land \NE \vdash (\phi \dis (\psi\\\land \NE))\land \NE $, thus we have $\phi \dis (\psi \land \NE) \vdash (\phi \dis (\psi\land \NE))\land \NE $ by $\OC\R{E}$.
\end{proof}

 \begin{proof}[Proof of Lemma \ref{lemma:di_NE_I} (in $\BSMLOBF$)] We derive $\Di \phi\vdash\Di(\phi \land \NE)$ as follows: 
 \begin{align*}
     &\Di \phi &&\vdash &&\Di(\phi \land \OC\NE) \tag{$\OC\NE\R{I}$}\\
     &&&\vdash &&\Di(\phi \land \NE) \lor \Di(\phi \land \bot) \tag{$\Di\OC\R{E}$}\\
     &&&\vdash &&\Di(\phi \land \NE) \lor \bot \tag{Prop. \ref{prop:classical_completeness}}\\
     &&&\vdash &&\Di(\phi \land \NE)  \tag*{($\bot \R{E}$)}
 \end{align*}
 \end{proof}

With all the pieces at hand, we are now ready to give the full proof of Lemma \ref{lemma:normal_form_provable_equivalence_BSMLO}.

\begin{proof}[Proof of Lemma \ref{lemma:normal_form_provable_equivalence_BSMLO}.]
    By induction on the complexity of $\phi$ (assumed w.l.o.g. to be in negation normal form). Let $k \geq md(\phi)$.
    
  If $\phi =p$ or $\phi =\bnot p$ or $\phi =\lnot \NE$, then the result follows by Lemma \ref{lemma:normal_form_provable_equivalence_BSMLO_classical} and $\lnot \NE\R{E}$.
    If $\phi=\NE$, then we have $\NE \proveq \NE \land  ( p \dis \lnot p)$ by Proposition \ref{prop:classical_completeness}, and  $\NE \land  ( p \dis \lnot p)\proveq \NE\land \underset{s\in \PPP}\bigdis \OC\theta^k_s$ by Lemma \ref{lemma:normal_form_provable_equivalence_BSMLO_classical}.

If $\phi =\psi \land \chi$, then $k\geq md(\phi)=max\{ md(\psi),md(\chi)\}$, so by the induction hypothesis, there are properties $\PPP,\QQQ$ such that 
\begin{equation}\label{lemma:normal_form_provable_equivalence_BSMLO_eq0}
\psi\proveq\eta_\psi \land\bigdis_{s\in \PPP} \OC\theta^k_s\text{ and }\chi\proveq\eta_\chi \land\bigdis_{t\in \QQQ} \OC\theta^k_{t}
\end{equation}
where each of $\eta_\psi$ and $\eta_\chi$ is either $\Top $ or $\NE$. Let
        $$\RRR:=\{r\sepp \biguplus \PPP' \bisim_k r\bisim_k \biguplus \QQQ'\text{ for some $\PPP'\subseteq\PPP$ and some $\QQQ'\subseteq \QQQ$} \}.$$
 We will show        \begin{equation}\label{lemma:normal_form_provable_equivalence_BSMLO_eq1}
\bigdis_{s\in \PPP}\OC\theta^{k}_{s}\land \bigdis_{t\in \QQQ}\OC\theta^{k}_{t}\proveq \bigdis_{r\in \RRR}\OC\theta^k_r,
        \end{equation}
        from which the result will follow, since if $\eta_\psi=\eta_\chi=\Top$, then $\phi \land \psi \proveq \bigdis_{r\in \RRR}\OC\theta^k_r$; and if $\eta_\psi=\NE$ or $\eta_\chi =\NE$, then $\phi \land \psi \proveq \NE\land \bigdis_{s\in \PPP}\OC\theta^{k}_{s}\land \bigdis_{t\in \QQQ}\OC\theta^{k}_{t}\proveq \NE \land \bigdis_{r\in \RRR}\OC\theta^k_r$.

Now, for the direction $\vdash$ of (\ref{lemma:normal_form_provable_equivalence_BSMLO_eq1}), by Lemma \ref{lemma:tensor_normal_form_normal_form_provable_equivalence} \ref{lemma:tensor_normal_form_normal_form_provable_equivalence_ii}, it suffices to show that $\theta^{k}_{\biguplus \PPP'},\theta^{k}_{\biguplus \QQQ'}\vdash \bigdis_{s\in \RRR}\OC\theta^{k}_{s}$ for all $\PPP'\subseteq \PPP$ and $\QQQ'\subseteq \QQQ$. If $\biguplus \PPP'\nbisim_k \biguplus \QQQ'$, then $\theta^{k}_{\biguplus \PPP'},\theta^{k}_{\biguplus \QQQ'}\vdash \Bot\vdash \bigdis_{r\in \RRR}\OC\theta^{k}_{r}$ by Lemmas \ref{lemma:not_bisimilar_implies_charfs_incompatible} and \ref{lemma:Bot_E}. If $\biguplus \PPP'\bisim_k \biguplus \QQQ'\bisim_k u\in \RRR$, then $\theta^{k}_{\biguplus \PPP'} \proveq \theta^{k}_{u}$ by Lemma \ref{lemma:charf_provable_equivalence}. Next, we derive by applying $\OC\R{I}$ and $\lor \R{I}$ that $\theta^{k}_{u}\vdash \OC\theta_u^k\vee\bot\vdash \OC\theta_u^k\vee\bigdis_{u\neq r\in \RRR}\OC\theta_r^k\vdash \bigdis_{r\in \RRR}\OC\theta^{k}_{r}$ .
        
For the converse direction $\dashv$ of (\ref{lemma:normal_form_provable_equivalence_BSMLO_eq1}), we only give the detailed proof for $\bigdis_{r\in \RRR}\OC\theta^{k}_{r}\vdash \bigdis_{s\in \PPP}\OC\theta^{k}_{s} $. Given Lemma \ref{lemma:tensor_normal_form_normal_form_provable_equivalence} \ref{lemma:tensor_normal_form_normal_form_provable_equivalence_ii}, it suffices to show $\theta^{k}_{\biguplus \RRR'}\vdash \bigdis_{s\in \PPP}\OC\theta^{k}_{s}$ for all $\RRR'\subseteq \RRR$. Observe that $\RRR'=\{r\sepp r \bisim_k \biguplus \PPP'$ and $\PPP'\in \mathcal{X}\}$ for some $\mathcal{X}\subseteq \wp(\PPP)$ and $\biguplus \RRR' \bisim_k \biguplus \bigcup \mathcal{X}$. Thus, $\theta^{k}_{\biguplus \RRR'}\proveq \theta^{k}_{\biguplus \bigcup \mathcal{X}}$ by Lemma \ref{lemma:charf_provable_equivalence}. Since $\bigcup \mathcal{X}\subseteq \PPP$, we have by Lemma \ref{lemma:tensor_normal_form_normal_form_provable_equivalence} \ref{lemma:tensor_normal_form_normal_form_provable_equivalence_i} that $\theta^{k}_{\biguplus \bigcup \mathcal{X}}\vdash \bigdis_{s\in \PPP}\OC\theta^{k}_{s}$.
 
   If $\phi =\psi \dis \chi$, then $k\geq md(\phi)= max\{md(\psi),md(\chi)\}$, so by the induction hypothesis we have (\ref{lemma:normal_form_provable_equivalence_BSMLO_eq0}). If $\eta_\psi=\NE$ and $\eta_\chi=\Top$, let
        $$\RRR:=\{  \biguplus \PPP'\uplus \biguplus \QQQ' \sepp \PPP' \subseteq \PPP \text{ s.t. }\biguplus \PPP'\nbisim_k (M,\emptyset), \text{ and }\QQQ'\subseteq \QQQ \},$$
        where $M$ is some arbitrary model. We show that
        $\displaystyle \psi \lor \chi \proveq \NE \land \bigdis_{r\in \RRR}\OC\theta^{k}_{r}$.
        
    For the direction $\vdash$, given Lemma \ref{lemma:tensor_normal_form_normal_form_provable_equivalence} \ref{lemma:tensor_normal_form_normal_form_provable_equivalence_ii}, it suffices to show that for all $\PPP'\subseteq \PPP$ and $\QQQ'\subseteq \QQQ$, $(\NE\land\theta^{k}_{\biguplus \PPP'})\lor \theta^{k}_{\biguplus \QQQ'}\vdash \NE\land \bigdis_{r\in \RRR}\OC\theta^{k}_{r}$. If $\biguplus\PPP'\bisim_k(M,\emptyset)$, then $\theta^k_{\biguplus \PPP'}\proveq\bot$ and we have $(\NE\land\theta^{k}_{\biguplus \PPP'})\lor \theta^{k}_{\biguplus \QQQ'}\vdash \Bot \lor\theta^{k}_{\biguplus \QQQ'}\vdash \NE\land \bigdis_{r\in \RRR}\OC\theta^{k}_{r}$ by $\Bot\R{Ctr}$. Otherwise $\biguplus \PPP'\nbisim_k \emptyset$, and thus $\biguplus \PPP'\uplus \biguplus \QQQ'\in \RRR$. We have \[(\NE\land\theta^{k}_{\biguplus \PPP'})\lor \theta^{k}_{\biguplus \QQQ'}\vdash \theta^{k}_{\biguplus \PPP'}\lor \theta^{k}_{\biguplus \QQQ'}\proveq \theta^{k}_{\biguplus \PPP'\uplus \biguplus \QQQ'}\vdash \bigdis_{r\in \RRR}\OC\theta^{k}_{r},\] 
      by $\OC\R{I}$ and $\lor \R{I}$, and $(\NE\land\theta^{k}_{\biguplus \PPP'})\lor \theta^{k}_{\biguplus \QQQ'}\vdash \NE$ by Lemma \ref{lemma:provability_results_NE}.
    
     For the direction $\dashv$, given Lemma \ref{lemma:tensor_normal_form_normal_form_provable_equivalence} \ref{lemma:tensor_normal_form_normal_form_provable_equivalence_ii}, it suffices to show that for all $\RRR'\subseteq \RRR$,  $\NE \land\theta^{k}_{\biguplus \RRR'}\vdash (\NE\land \bigdis_{s\in \PPP}\OC\theta^{k}_{s}) \dis \bigdis_{t\in \QQQ}\OC\theta^{k}_{t}$. If $\RRR'=\emptyset$, then $\NE \land\theta^{k}_{\biguplus \RRR'}=\NE\land \bot=\Bot$, and the implication follows by Lemma \ref{lemma:Bot_E}. Otherwise 
     \[\RRR'=\{  \biguplus \PPP' \uplus \biguplus\QQQ'\sepp\PPP'\in \mathcal{X}\text{ and }\QQQ'\in \mathcal{Y}\}\] 
     for some nonempty $\mathcal{X}\subseteq \wp(\PPP)$ and nonempty $\mathcal{Y}\subseteq \wp(\QQQ)$ with $\biguplus \PPP' \nbisim_k (M,\emptyset)$  for all $\PPP'\in \mathcal{X}$. Since $\biguplus \RRR' \bisim_k \biguplus \bigcup \mathcal{X} \uplus \biguplus \bigcup \mathcal{Y}$, by Lemma \ref{lemma:charf_provable_equivalence} we have  \[\theta^{k}_{\biguplus \RRR'}\proveq \theta^{k}_{\biguplus \bigcup \mathcal{X}\uplus \biguplus \bigcup \mathcal{Y}}\proveq \theta^{k}_{\biguplus \bigcup \mathcal{X}}\lor \theta^{k}_{\biguplus \bigcup \mathcal{Y}}.\] Observe that $\biguplus \bigcup \mathcal{X}\nbisim_k (M,\emptyset)$, which implies $\theta^{k}_{\biguplus \bigcup \mathcal{X}}\neq \bot$, so that $\theta^{k}_{\biguplus \bigcup \mathcal{X}} \vdash \NE \land \theta^{k}_{\biguplus \bigcup \mathcal{X}} $ by Lemma \ref{lemma:provability_results_NE}. Thus, $\theta^{k}_{\biguplus \bigcup \mathcal{X}}\lor \theta^{k}_{\biguplus \bigcup \mathcal{Y}}\vdash (\NE\land \theta^{k}_{\biguplus \bigcup \mathcal{X}})\lor \theta^{k}_{\biguplus \bigcup \mathcal{Y}}$. Finally, by Lemma \ref{lemma:tensor_normal_form_normal_form_provable_equivalence} \ref{lemma:tensor_normal_form_normal_form_provable_equivalence_i}, we derive $(\NE\land \theta^{k}_{\biguplus \bigcup \mathcal{X}})\lor \theta^{k}_{\biguplus \bigcup \mathcal{Y}}\vdash (\NE\land \bigdis_{s\in \PPP}\OC\theta^{k}_{s}) \dis \bigdis_{t\in \QQQ}\OC\theta^{k}_{t}$.
     
        The other cases are similar; for example, if $\eta_\psi=\eta_\chi=\NE$, one uses
               \begin{align*}
        \RRR:=\{& \biguplus \PPP'\uplus \biguplus \QQQ' \sepp \PPP' \subseteq \PPP \text{ and }\QQQ'\subseteq \QQQ \text{ s.t. }\biguplus \PPP'\nbisim_k \emptyset\text{ and }\biguplus \QQQ'\nbisim_k \emptyset \},
        \end{align*}
        and if $\eta_\psi=\eta_\chi=\Top$ we  clearly have $\phi \lor \psi \proveq \bigdis_{u\in \PPP\cup \QQQ}\OC\theta^{k}_{u}  $.
        
  If $\phi=\OC\psi$, then $k\geq md(\phi)= md(\psi)$, so by the induction hypothesis we have the equivalence for $\psi$ as in (\ref{lemma:normal_form_provable_equivalence_BSMLO_eq0}). We show $\OC (\eta_\psi \land \bigdis_{s\in \PPP}\OC\theta^{k}_{s})   
  \proveq\bigdis_{s\in \PPP}\OC\theta^{k}_{s}$. For the direction $\vdash$, we have $\eta_\psi \land \bigdis_{s\in \PPP}\OC\theta^{k}_{s}\vdash \bigdis_{s\in \PPP} \OC\theta^k_s $ and $\bot\vdash \bigdis_{s\in \PPP}\bot\vdash \bigdis_{s\in \PPP} \OC\theta^k_s$ by $\OC\R{I}$. Thus we conclude $\OC (\eta_\psi \land \bigdis_{s\in \PPP}\OC\theta^{k}_{s})   \vdash\bigdis_{s\in \PPP}\OC\theta^{k}_{s}$ by $\OC\R{E}$. For the converse direction $\dashv$, if $\eta_\psi=\Top$, we have $\bigdis_{s\in \PPP}\OC\theta^{k}_{s} \proveq \Top \land \bigdis_{s\in \PPP}\OC\theta^{k}_{s} \vdash \OC(\Top \land \bigdis_{s\in \PPP}\OC\theta^{k}_{s})$ by $\OC\R{I}$. If $\eta_\psi=\NE$, we have $\bigdis_{s\in \PPP}\OC\theta^{k}_{s}\vdash \bigdis_{s\in \PPP}\OC\theta^{k}_{s} \land\OC\NE$ by $\OC\NE\R{I}$. Then since $\bigdis_{s\in \PPP}\OC\theta^{k}_{s} \land\NE\vdash \OC(\bigdis_{s\in \PPP}\OC\theta^{k}_{s} \land\NE)$ by $\OC\R{I}$ and $\bigdis_{s\in \PPP}\OC\theta^{k}_{s} \land\bot \vdash \bot \vdash \OC(\bigdis_{s\in \PPP}\OC\theta^{k}_{s} \land\NE)$ by $\OC\R{I}$, we have $\bigdis_{s\in \PPP}\OC\theta^{k}_{s} \land\OC\NE\vdash \OC(\bigdis_{s\in \PPP}\OC\theta^{k}_{s} \land\NE)$ by $\OC\R{E}$.

 If $\phi=\Di\psi$, then $k-1\geq md(\phi)-1=md(\psi)$, so by the induction hypothesis, there is a property $\PPP$ such that $\psi\proveq\eta_\psi \land\bigdis_{s\in \PPP} \OC\theta^{k-1}_s$ where $\eta_\psi=\Top$ or $\eta_\psi=\NE$. Let $n=k-1$. By Lemma \ref{lemma:normal_form_provable_equivalence_BSMLO_classical}, it suffices to show that $\psi$ is equivalent to some classical formula of modal depth $\leq k=n+1$. Now, if $\eta_\psi=\NE$, we have $\Di\psi \proveq \Di(\eta_\psi \land \bigdis_{s\in \PPP}\OC\theta^{n}_{s})\proveq \Di\bigdis_{s\in \PPP}\OC\theta^{n}_{s}$ by Lemma \ref{lemma:di_NE_I}; the same equivalence holds also when $\eta_\psi=\Top$. It is then sufficient to show that 
\begin{equation}\label{lemma:normal_form_provable_equivalence_BSMLO_eq2}
  \Di\bigdis_{s\in \PPP}\OC\theta^{n}_{s}\proveq  \bigdis_{\PPP'\subseteq\PPP}\Di\theta^{n}_{\biguplus \PPP'},
  \end{equation}
  for then we have $\displaystyle\bigdis_{\PPP'\subseteq \PPP}\Di \theta^n_{\biguplus \PPP'} \proveq \bigdis_{\PPP'\subseteq \PPP}(\Di\chi^n_{\biguplus \PPP'} \land \bigwedge_{w\in \biguplus \PPP'}\Di \chi^n_w)$ by Lemma \ref{lemma:modal_charfs_provably_classical}\ref{lemma:modal_charfs_provably_classical_Di}.

  To prove the direction $\dashv$ of (\ref{lemma:normal_form_provable_equivalence_BSMLO_eq2}), for any $\PPP'\subseteq \PPP$ we have $\theta^n_{\uplus \PPP'}\vdash \bigdis_{s\in \PPP}\OC\theta^{n}_{s}$ by Lemma \ref{lemma:tensor_normal_form_normal_form_provable_equivalence} \ref{lemma:tensor_normal_form_normal_form_provable_equivalence_i}, so  $\Di\theta^n_{\uplus \PPP'}\vdash \Di\bigdis_{s\in \PPP}\OC\theta^{n}_{s}$ by $\Di \R{Mon}$. The result then follows by $\dis\R{E}$. For the converse direction $\vdash$, let $\PPP=\{s_1,\dots,s_m\}$. Repeatedly applying $\Di\OC\R{E}$ gives
\begin{align*}
\Di\bigvee_{i=1}^m\OC\theta_{s_i}^n
\vdash&\Di(\theta_{s_1}^n\vee\OC\theta_{s_2}^n\vee\dots\vee\OC\theta_{s_m}^n)\vee\Di(\bot\vee\OC\theta_{s_2}^n\vee\dots\vee\OC\theta_{s_m}^n)\\
\vdash&\big(\Di(\theta_{s_1}^n\vee\theta_{s_2}^n\vee\OC\theta_{s_3}^n\vee\dots\vee\OC\theta_{s_m}^n)\vee\Di(\theta_{s_1}^n\vee\bot\vee\OC\theta_{s_3}^n\vee\dots\vee\OC\theta_{s_m}^n)\big)\\
&\vee\Di(\bot\vee\theta_{s_2}^n\vee\OC\theta_{s_3}^n\vee\dots\vee\OC\theta_{s_m}^n)\vee\Di(\bot\vee\bot\vee\OC\theta_{s_3}^n\vee\dots\vee\OC\theta_{s_m}^n)\\
\vdots&\\
\vdash&\bigvee_{\tau_1\in \{s_1,\emptyset\}}\dots\bigvee_{\tau_m\in \{s_m,\emptyset\}}\Di(\theta_{\tau_1}^n\vee\dots\vee\theta_{\tau_m}^n)\tag{recall $\theta_\emptyset^n=\bot$}.
\end{align*}
For each sequence $\tau_1,\dots,\tau_m$ as above, putting $\PPP'=\{\tau_i\mid \tau_i=s_i,~1\leq i\leq m\}$, it is easy to see $\Di(\theta_{\tau_1}^n\vee\dots\vee\theta_{\tau_m}^n)\proveq \Di \bigvee_{\tau_i\in \PPP'}\theta_{\tau_i}^n\proveq\Di \theta^n_{\biguplus \PPP'}$. Thus, \[\bigvee_{\tau_1\in \{s_1,\emptyset\}}\dots\bigvee_{\tau_m\in \{s_m,\emptyset\}}\Di(\theta_{\tau_1}^n\vee\dots\vee\theta_{\tau_m}^n)\vdash \bigdis_{\PPP'\subseteq \PPP}\Di \theta^n_{\biguplus \PPP'}.\] 
        
If  $\phi =\Bo\psi$, then $n=k-1\geq md(\phi)-1=md(\psi)$, so by the induction hypothesis, we have the same normal form equivalence for $\psi$ as in the diamond case. Again, by Lemma \ref{lemma:normal_form_provable_equivalence_BSMLO_classical} it suffices to show that $\Bo \psi$ is provably equivalent to a classical formula of modal depth $\leq k$. For the case $\eta_\psi=\Top$, we have $\Bo\psi \proveq \Bo\bigdis_{s\in \PPP}\OC\theta^{n}_{s}$. One can show that $\Bo\bigdis_{s\in \PPP}\OC\theta^{n}_{s}$ is provably equivalent to a classical formula of modal depth $\leq k$ with a proof analogous to that given for $\Di\bigdis_{s\in \PPP}\OC\theta^{n}_{s}$ in the case for $\phi=\Di\psi$ (using $\Bo \R{Mon}$, $\Bo\OC\R{E}$, and Lemma \ref{lemma:modal_charfs_provably_classical} \ref{lemma:modal_charfs_provably_classical_Bo} instead of $\Di \R{Mon}$, $\Di\OC\R{E}$, and \ref{lemma:modal_charfs_provably_classical} \ref{lemma:modal_charfs_provably_classical_Di}, respectively). For the case $\eta_\psi=\NE$, we have $\Bo\psi \proveq \Bo(\NE \land \bigdis_{s\in \PPP}\OC\theta^{n}_{s})\proveq \Bo\bigdis_{s\in \PPP}\OC\theta^{n}_{s} \land \Di\bigdis_{s\in \PPP}\OC\theta^{n}_{s}$ by Lemma \ref{lemma:fc} \ref{lemma:fc_Bo}. We have already shown that each conjunct in this formula is provably equivalent to a classical formula.
\end{proof}

%% file: sections/axiomatization/BSML/bsml_rules.tex
\subsection{$\BSMLBF$} \label{section:BSML}

In this final section, we axiomatize our base logic $\BSML$. As with $\BSMLI$ and $\PT$, the propositional fragment of $\BSML$ corresponds to the logic $\CPL$ studied and axiomatized in \cite{yang2017} (with the caveat that $\CPL$ does not feature a bilateral negation). To construct the system for $\BSML$,  we remove all rules involving $\intd$ from the system for $\BSMLI$, and add rules which simulate the removed rules.

\begin{definition}[Natural deduction system for $\BSMLBF$] \label{def:BSML_system}

The natural deduction system for $\BSML$ includes all rules not involving $\intd$ from the system for $\BSMLI$ (boxes (a)--(f)) and the following rules simulating the behavior of $\intd$:
 
       \begin{proofbox}[]
       {\small
       
            \begin{prooftree}
                \AxiomC{$D $}
                \noLine
            \UnaryInfC{$\phi$}
                \AxiomC{$[\phi[\psi \land \NE/\psi]] $}
                \noLine
                \UnaryInfC{$ D_1 $}
                \noLine
            \UnaryInfC{$ \chi$}
                \AxiomC{$[\phi[\psi \land \bot/\psi]] $}
                \noLine
            \UnaryInfC{$ D_2 $}
                \noLine
            \UnaryInfC{$ \chi$}
            \RightLabel{$\bot\NE\R{Trs} (*)$}
            \TrinaryInfC{$ \chi $}
            \end{prooftree}
\vspace{0.2cm}
            
        \begin{prooftree}
                \AxiomC{$D $}
                \noLine
            \UnaryInfC{$\Di\phi$}
            \RightLabel{$\Di\bot \NE\R{Trs} (*)$}
            \UnaryInfC{$ \Di\phi[\psi \land \NE /\psi]\dis \Di \phi[\psi \land \bot/\psi] $}
        \end{prooftree}
        \vspace{0.2cm}
        
        \begin{prooftree}
                \AxiomC{$D $}
                \noLine
            \UnaryInfC{$\Bo\phi$}
            \RightLabel{$\Bo\bot \NE\R{Trs} (*)$}
            \UnaryInfC{$ \Bo\phi[\psi \land \NE /\psi]\dis \Bo \phi[\psi \land \bot/\psi] $}
        \end{prooftree}
        \vspace{0.2cm}
        }
        
        {\footnotesize
    
        $(*)$ $[\psi]$ is $\intd\hspace{-1pt}$-distributive in $\phi$.}
    \end{proofbox}
\end{definition}
Recall that $[\psi]$ is $\intd\hspace{-1pt}$-distributive in $\chi$ if $[\psi]$ is not within the scope of any $\lnot$ or $\Di$ in $\chi$. So given $\psi \equiv\psi \land (\NE \intd\bot)$, and given that $\land$ and $\lor$ distribute over $\intd$, if $[\psi]$ is $\intd$-distributive in $\phi$, then $\phi\equiv \phi[\psi \land (\NE \intd\bot)/\psi]\equiv  \phi[\psi\wedge \NE/\psi]\intd\phi[\psi\wedge\bot/\psi]$. The $\bot\NE$-translation rule $\bot\NE \R{Trs}$ captures this equivalence by simulating $\intd$-elimination as applied to $ \phi[\psi\wedge \NE/\psi]\intd\phi[\psi\wedge\bot/\psi]$. The rule $\Di\bot \NE\R{Trs}$ captures the equivalence $\Di\phi\equiv \Di(\phi[\psi\land \NE/\psi]\intd\phi[\psi\land\bot/\psi])\equiv \Di\phi[\psi\land \NE/\psi]\vee\Di\phi[\psi\land\bot/\psi]$ by simulating $\Di\intd\dis \R{Conv}$; similarly for $\Bo\bot \NE\R{Trs}$ and $\Bo\intd\dis \R{Conv}$. The soundness of the new rules follows from these equivalences, and so we have:

\begin{theorem}[Soundness of $\BSMLBF$] \label{theorem:soundness_BSML}
If $\Phi \vdash \psi$, then $ \Phi \vDash \psi$.
\end{theorem}

%% file: sections/axiomatization/BSML/bsml_completeness.tex
Observe that the proofs given in Section \ref{section:BSMLI} for Lemma \ref{lemma:provable_subformula_replacement} (Replacement) and Proposition \ref{prop:negation_normal_form_provable_BSMLI} (Negation normal form) also suffice to establish these results for the $\BSML$-system. In what follows we can thus always assume that formulas are in negation normal form.

We now turn to the proof of the completeness theorem for the system. From our expressivity analysis in Section \ref{section:expressive_power}, we know that every formula $\phi$ in $\BSML$ (being also a formula in the stronger logic $\BSMLI$) is equivalent to a formula in $\BSMLI$-normal form: $\phi \equiv\bigintd_{s\in ||\phi||}\theta_s^k$, where $k=md(\phi)$. While the global disjunction in the normal form is not in the language of $\BSML$, we will still be able to make use of this disjunctive normal form in our completeness proof, as we were also able to do in our proof for $\BSMLO$. Adapting a similar strategy employed in \cite{yang2017} for $\CPL$, we associate with each formula $\phi$ a set of $\BSML$-formulas $\phi^f$ of certain syntactic form, called \emph{realizations} of $\phi$. The realizations of $\phi$ correspond to the disjuncts $\theta_s^k$ in the normal form for $\phi$---we will show that $\phi \equiv\bigintd_{s\in ||\phi||} \theta_s^k\equiv\bigintd\{\phi^f\sepp \phi^f$ is a realization for $\phi\}$ (cf. Lemma \ref{lemma:normal_form_conversion})---and $\phi$ can be simulated in the completeness proof using its realizations: we also show that $\phi$ and $\bigintd\{\phi^f\sepp \phi^f$ is a realization for $\phi\}$ are proof-theoretically interchangeable (cf. Lemma \ref{lemma:tensor_normal_form_normal_form_provable_equivalence}).

Realizations for $\phi$ are defined by replacing $\intd$-distributive occurrences $[\eta]$ in $\phi$ by disjuncts $\theta^k_s$ in their respective normal forms $\bigintd_{||\eta||}\theta^k_s$ (where $k \geq md(\phi)$). The equivalence between $\phi$ and $\bigintd\{\phi^f\sepp \phi^f$ is a realization for $\phi\}$ then follows by the $\intd$-distributivity of these occurrences. In order to secure the proof-theoretic interchangeability result, we require each $\eta$ to be replaced to be either classical or $\NE$---we first essentially show that the required results hold for formulas that can be constructed using only such occurrences (that is, they hold for formulas of the form $\eta_1\medcircle_1 \eta_2\medcircle_2\ldots \medcircle_{n-1} \eta_n$ where each $\eta_i\in \ML$ or $\eta_i=\NE$, and each $\medcircle_i$ is either $\wedge$ or $\vee$), and later show that every formula is provably equivalent to a formula of this form.

Formally, we define realizations relative to $\intd\hspace{-1pt}$-distributive sequences, and we show our results hold with respect to such sequences. A \emph{$\intd\hspace{-1pt}$-distributive sequence} of $\phi$ is a sequence $\mathbf{a}=\langle \eta_1, \ldots, \eta_n\rangle$ of distinct $\intd\hspace{-1pt}$-distributive occurrences of formulas $\eta_i$ in $\phi$ such that each $\eta_i\in \ML$ or $\eta_i=\NE$ for all $1\leq i \leq n$.

Let $\mathbf{a}=\langle \eta_1, \ldots, \eta_n\rangle$ be a $\intd$-$\hspace{-1pt}$-distributive sequence of $\phi$ and let $\mathsf{X}\supseteq \mathsf{P}(\phi)$ be some finite set of propositional variables. A \emph{realizing function over $\mathbf{a}$ and $\mathsf{X}$} is a function $f:\{1,\ldots , n\}\to||\Top||_{\mathsf{X}}$ such that $f_i\vDash \eta_i$ for all $1\leq i \leq n$, where $||\Top||_{\mathsf{X}}=\{(M,s)\sepp M,s\models\Top\}$ is the class of all pointed state models over $\mathsf{X}$. For any $k\geq md(\phi)$, the \emph{$(k,f)$-realization of $\phi$}, written $\phi^{k,f}$ (or simply $\phi^f$), is the $\BSML$-formula
$$\phi[\theta^{k}_{f_1}/\eta_1,\ldots \theta^{k}_{f_n}/\eta_n],$$
where recall that each $\theta^k_{f_i}$ is the strong Hintikka formula for the state $f_i$. We write $\mathcal{F}^{\mathsf{X},\mathbf{a}}$ (or simply $\mathcal{F}^{\mathbf{a}}$) for the set of all realizing functions over $\mathbf{a}$ and $\mathsf{X}$. 

Most results in this section hold for all finite $\mathsf{X}\supseteq \mathsf{P}(\phi)$ and all $k\geq md(\phi)$ where $\phi$ is clear from the context; we usually omit mention of these conditions. As promised, each formula is equivalent to the global disjunction of all its realizations: 
 
\begin{lemma} \label{lemma:realization_equivalence}
Let $\mathbf{a}$ be a $\intd\hspace{-1pt}$-distributive sequence of $\phi$. Then $\phi \equiv \bigintd_{f\in \mathcal{F}^{\mathbf{a}}} \phi^{f}$.
\end{lemma}
\begin{proof} By the proof of Theorem \ref{theorem:characterization_BSMLI}, $\eta_i\equiv \bigintd_{s\in || \eta_i||} \theta^k_{s}$ for each $i$. The result then follows since $\mathbf{a}$ is
$\intd\hspace{-1pt}$-distributive, and $\land$ and $\lor$ distribute over $\intd$.
\end{proof}

As explained above, we will also show that in the system for $\BSML$, every formula is proof-theoretically interchangeable with the global disjunction of its realizations. This is expressed formally in the following lemma.

\begin{lemma} \label{lemma:realization_provable_equivalence}
Let $\mathbf{a}$ be a $\intd\hspace{-1pt}$-distributive sequence of $\phi$. Then:
\begin{enumerate}
    \item[(i)] 
        \makeatletter\def\@currentlabel{(i)}\makeatother\label{lemma:realization_provable_equivalence_i} $\phi^{f}\vdash \phi$ for all $f\in \mathcal{F}^{\mathbf{a}}$ and
    \item[(ii)] 
        \makeatletter\def\@currentlabel{(ii)}\makeatother\label{lemma:realization_provable_equivalence_ii} if $\Phi,\phi^{f}\vdash \psi$ for all $f\in \mathcal{F}^{\mathbf{a}}$, then $\Phi,\phi\vdash \psi$.
\end{enumerate}
\end{lemma}

The proof of the above requires several technical lemmas. For item \ref{lemma:realization_provable_equivalence_i}, we need to show that Lemma \ref{lemma:provability_results_NE} is also derivable in the system for $\BSML$:

\begin{proof}[Proof of Lemma \ref{lemma:provability_results_NE} (in $\BSMLBF$)] We prove $\phi \dis (\psi \land \NE) \proveq (\phi \dis (\psi\land \NE))\land \NE $. Let $\chi:=\phi \dis (\psi \land \NE)$. Then $\chi[\chi \land \bot/\chi]=(\phi \dis (\psi \land \NE))\land \bot \vdash (\phi \land \bot )\lor (\psi \land \NE\land \bot)$ by Proposition \ref{prop:commutativity_associativity_distributivity}, and $(\phi \land \bot )\lor (\psi \land \NE\land \bot) \vdash (\phi \land \bot )\lor \Bot\vdash (\phi \dis (\psi\land \NE))\land \NE$ by $\Bot\R{Ctr}$. And $\chi[\chi \land \NE/\chi]=(\phi \dis (\psi\land \NE))\land \NE \vdash (\phi \dis (\psi\land \NE))\land \NE $, and so we have $\phi \dis (\psi \land \NE) \vdash (\phi \dis (\psi\land \NE))\land \NE $ by $\bot\NE\R{Trs}$.
\end{proof}

\begin{proof} [Proof of Lemma \ref{lemma:realization_provable_equivalence} \ref{lemma:realization_provable_equivalence_i} ]
     By Lemma \ref{lemma:provable_subformula_replacement}, it suffices to show $\theta^k_{f_i}\vdash \eta_i$ for all $1\leq i\leq n$. If $\eta_i=\NE$, then $f_i\neq \emptyset$ so $\theta^k_{f_i}$ is not $\bot$. We have $\theta^k_{f_i}=\bigdis_{w\in f_i}(\chi^k_w\land \NE)\vdash \NE$ by Lemma \ref{lemma:provability_results_NE}. If $\eta_i=\alpha\in \ML$, we clearly have that $\theta^k_{f_i}\vdash \chi^k_{f_i}$. Since $f_i\vDash\alpha$, we have $f_i\subseteq \llbracket \alpha\rrbracket$ by flatness, which implies $\chi_{f_i}^{k}\vDash \chi_{\llbracket \alpha\rrbracket}^k\equiv\alpha$. Therefore, $\chi^k_{f_i}\vdash \alpha$ by Proposition \ref{prop:classical_completeness}.
\end{proof}

To prove item \ref{lemma:realization_provable_equivalence_ii} of Lemma \ref{lemma:realization_provable_equivalence}, we first show that for a $\intd\hspace{-1pt}$-distributive sequence $\mathbf{a}=\langle \eta \rangle$ of length $1$, the entailment $\phi[\eta]\vDash \phi [\Top \land \eta/\eta]\vDash \phi [\bigintd_{s \in ||\Top||_{\mathsf{X}}}\theta^{\mathsf{X},k}_s \land \eta/\eta]$ can be simulated in our system in the sense of the following lemma.

\begin{lemma} \label{lemma:NE_E}
    Let $[\eta]$ be $\intd\hspace{-1pt}$-distributive in $\phi$. Let $k\in \mathbb{N}$ and let $\mathsf{X}\subseteq \mathsf{Prop}$ be finite. If $\Phi,\phi[\theta^{\mathsf{X},k}_s\land \eta/\eta]\vdash \psi$ for all $s\in ||\Top||_{\mathsf{X}}$, then $\Phi,\phi \vdash \psi$.
\end{lemma}
\begin{proof}
Consider $\chi^k_{\llbracket \Top\rrbracket}:=\bigvee_{w\in \llbracket \Top\rrbracket_{\mathsf{X}}}\chi_w^k$, where $\llbracket \Top\rrbracket_{\mathsf{X}}=\{w\sepp w\vDash \Top\}$. Since there are only finitely many non-equivalent $k$-th Hintikka formulas $\chi_w^k$ over $\mathsf{X}$, we may assume w.l.o.g.  $\chi^k_{\llbracket \Top\rrbracket}=\chi^k_{w_1}\lor\ldots\lor \chi^k_{w_n}$, for some worlds $w_1,\dots,w_n$ from some  models $M_1,\dots,M_n$. Clearly, $\chi^k_{\llbracket \Top\rrbracket}\equiv\Top$, which implies $\vdash \chi^k_{\llbracket\Top\rrbracket}$ by Proposition \ref{prop:classical_completeness}. Thus, we have $\phi\vdash \phi[\chi^k_{\llbracket\Top\rrbracket}\land \eta/\eta]$ by Lemma \ref{lemma:provable_subformula_replacement}. 

Now, for any $\tau_1,\dots,\tau_n\in\{\NE,\bot\}$, consider the state $s=\{w_i\mid \tau_i=\NE\}$. By applying $\bot \R{E}$ and $\dis\R{I}$, we derive $\bigvee_{i=1}^n(\chi_{w_i}^k\land\tau_i) \proveq \bigdis_{w_i\in s}(\chi^k_{w_i}\land \NE )= \theta^k_s$. Thus, we have $\phi[\bigvee_{i=1}^n(\chi_{w_i}^k\land\tau_i)\land \eta/\eta]\proveq \phi[\theta_s^k\land \eta/\eta]$. Therefore, by  assumption we have
\[\Phi,\phi[\big((\chi^k_{w_1}\land \bot)\vee\ldots\vee(\chi^k_{w_{n}}\land \bot)\big)\wedge\eta/\eta]\vdash \psi\]
\[\text{ and }\Phi,\phi[\big((\chi^k_{w_1}\land \bot)\vee\ldots\vee(\chi^k_{w_{n}}\land \NE)\big)\wedge \eta/\eta]\vdash \psi,\]
from which $\Phi,\phi[\big((\chi^k_{w_1}\land \bot)\vee\ldots\vee(\chi^k_{w_{n-1}}\land \bot)\vee\chi^k_{w_n}\big)\land\eta/\eta]\vdash \psi$ follows by $\bot\NE\R{Trs}$. Repeating this argument, we finally get $\Phi,\phi[\chi^k_{\llbracket\Top\rrbracket}\land \eta/\eta]\vdash  \psi$, whence $\Phi,\phi\vdash  \psi$.
\end{proof}

Next, in \ref{lemma:theta_and_formula_implies_realization_for_formula_i} of the following lemma (\ref{lemma:theta_and_formula_implies_realization_for_formula_ii} is used later) we show that for a $\intd\hspace{-1pt}$-distributive sequence $\mathbf{a}=\langle \eta_1, \ldots, \eta_n\rangle$ of arbitrary length, we can simulate the entailment $$ \phi [\bigintd_{s \in ||\Top||_{\mathsf{X}}}\theta^{\mathsf{X},k}_s \land \eta_1/\eta_1,\ldots, \bigintd_{s \in ||\Top||_{\mathsf{X}}}\theta^{\mathsf{X},k}_s \land \eta_n/\eta_n]\vDash \bigintd_{f\in \FFF^{\mathsf{X},\mathbf{a}}}\phi^{k,f}.$$

\begin{lemma} \label{lemma:theta_and_formula_implies_realization_for_formula}
    Let $\mathbf{a}$ be a $\intd\hspace{-1pt}$-distributive sequence of $\phi$. Then:
    \begin{enumerate}
        \item [(i)] 
        \makeatletter\def\@currentlabel{(i)}\makeatother \label{lemma:theta_and_formula_implies_realization_for_formula_i} for any $s_1,\ldots ,s_n\in ||\Top||_{\mathsf{X}}$, there is some $f\in \FFF^{\mathbf{a}}$ such that $$\phi[\theta^{\mathsf{X},k}_{s_1}\land \eta_1/\eta_1,\ldots ,\theta^{\mathsf{X},k}_{s_n}\land \eta_n/\eta_n]\vdash \phi^{k,f};$$
        \item[(ii)] 
        \makeatletter\def\@currentlabel{(ii)}\makeatother\label{lemma:theta_and_formula_implies_realization_for_formula_ii} for any $f\in \FFF^{\mathbf{a}}$ there are some $s_1,\ldots , s_n\in ||\Top||_{\mathsf{X}}$ such that $$\phi[\theta^{\mathsf{X},k}_{s_1}\land \eta_1/\eta_1,\ldots ,\theta^{\mathsf{X},k}_{s_n}\land \eta_n/\eta_n]\vdash \phi^{k,f}.$$
    \end{enumerate}
\end{lemma}
\begin{proof}
\ref{lemma:theta_and_formula_implies_realization_for_formula_i} We first give the proof for the case when $\mathbf{a}=\langle \eta_1\rangle$ is a sequence of length $1$. Let $(M_1,s_1)\in ||\Top||_{\mathsf{X}}$. Case 1: $\eta_1= \NE$. If $s_1\neq \emptyset$, let $f_1:=(M_1,s_1)$. Since $s_1\vDash\NE$, we have $f\in \FFF^{\mathbf{a}}$. We derive $\phi[\theta^k_{s_1}\land \NE/\NE]\vdash \phi[\theta^k_{s_1}/\NE]=\phi^f$ by Lemma \ref{lemma:provable_subformula_replacement}. If $s_1= \emptyset$, pick any $(N,t)$ such that $t\neq \emptyset$ and let $f_1:=(N,t)$. Since $t\vDash\NE$, we have $f\in \FFF^{\mathbf{a}}$. Now, we derive $\theta^k_{s_1}\land \NE=\bot\land\NE= \Bot\vdash \theta_{t}^k$ by Lemma \ref{lemma:Bot_E}. Thus, $\phi[\theta^k_{s_1}\land \NE/\NE]\vdash \phi[\theta^k_{t}/\NE]=\phi^f$ follows from Lemma \ref{lemma:provable_subformula_replacement}. 
    
Case 2: $\eta_1 =\alpha\in \ML$. Let $r=\{w\in s_1\sepp w\vDash \alpha\}$. Define $f_1:=(M_1,r)$, which yields a realizing function over $\langle \eta_1\rangle$ (as $r\vDash \alpha$). We show that $\phi[\theta_{s_1}^k\wedge\alpha/\alpha]\vdash\phi^f$, i.e., $\phi[\theta_{s_1}^k\wedge\alpha/\alpha]\vdash\phi[\theta_{r}^k/\alpha]$. By Lemma \ref{lemma:provable_subformula_replacement}, it suffices to show $\theta_{s_1}^k\wedge\alpha\vdash\theta_r^k$. Now, by Proposition \ref{prop:commutativity_associativity_distributivity}, we derive that $$ \theta^k_{s_1}\land \alpha\vdash(\theta^k_{r}\lor \theta^k_{s_1\backslash r} )\land \alpha \vdash \theta^k_{r}\lor (\theta^k_{s_1\backslash r} \land \alpha )\vdash \theta^k_{r}\lor (\chi_{s_1\backslash r}^k\land \alpha).$$
Clearly, $\chi_{s_1\backslash r}^k\vDash\neg\alpha$ and thus $\chi_{s_1\backslash r}^k\vdash\neg\alpha$ by Proposition \ref{prop:classical_completeness}. Thus, we further derive $\theta^k_{r}\lor (\chi_{s_1\backslash r}^k\land \alpha)\vdash\theta^k_{r}\lor (\neg\alpha\land \alpha)\vdash \theta^k_{r}\lor\bot\vdash\theta_r^k$ by $\bot \R{E}$. We thus conclude $\theta_{s_1}^k\wedge\alpha\vdash\theta_r^k$.
      
Now let $\mathbf{a}=\langle \eta_1, \ldots, \eta_n\rangle$ with $n>1$. Let $\phi':= \phi[\theta^k_{s_2}\land \eta_2/\eta_2,\ldots ,\theta^k_{s_n}\land \eta_n/\eta_n]$ and $\mathbf{a}':=\langle \eta_1\rangle$. Applying what we just proved for $\phi'$ and $\mathbf{a'}$ we have $\phi'[\theta^k_{s_1}\land \eta_1/\eta_1]\vdash \phi'^g$ for some $g\in \FFF^{\mathbf{a'}}$. Repeating this argument $n$ times, we can find some $f\in \FFF^{\mathbf{a}}$ such that $\phi[\theta^k_{s_1}\land \eta_1/\eta_1,\ldots ,\theta^k_{s_n}\land \eta_n/\eta_n]\vdash \phi^f$.

\ref{lemma:theta_and_formula_implies_realization_for_formula_ii}  Define $(M_i,s_i):=f_i$.  Since $\theta_{s_i}^k\wedge \eta_i \vdash \theta_{s_i}^k$ for each $i$, we derive by Lemma \ref{lemma:provable_subformula_replacement} that $\phi[\theta^k_{s_1}\land \eta_1/\eta_1,\ldots ,\theta^k_{s_n}\land \eta_n/\eta_n]\vdash \phi[\theta^k_{s_1}/\eta_1,\ldots ,\theta^k_{s_n}/\eta_n]= \phi^f$.
\end{proof}

Putting together Lemmas \ref{lemma:NE_E} and \ref{lemma:theta_and_formula_implies_realization_for_formula} \ref{lemma:theta_and_formula_implies_realization_for_formula_i}, we can simulate $ \phi \vDash \bigintd_{f\in \FFF^{\mathbf{a}}}\phi^{f}$; that is, we can prove Lemma \ref{lemma:realization_provable_equivalence} \ref{lemma:realization_provable_equivalence_ii}.

\begin{proof} [Proof of Lemma \ref{lemma:realization_provable_equivalence} \ref{lemma:realization_provable_equivalence_ii}]
 Assume $\Phi,\phi^f\vdash \psi$ for all $f\in \mathcal{F}^\mathbf{a}$. By multiple applications of Lemma \ref{lemma:NE_E} it suffices to show $\Phi, \phi[\theta^k_{s_1}\land \eta_1/\eta_1,\ldots \theta^k_{s_n}\land \eta_n/\eta_n]\vdash \psi$ for all $s_1,\ldots,s_n\in ||\Top||_{\mathsf{X}}$. This follows by Lemma \ref{lemma:theta_and_formula_implies_realization_for_formula} \ref{lemma:theta_and_formula_implies_realization_for_formula_i} and our assumption. 
\end{proof}

We have now settled that relative to $\intd\hspace{-1pt}$-distributive sequences, each formula is proof-theoretically interchangeable with the global disjunction of all its realizations. For a formula $\phi$ that can be constructed solely using the occurrences in such a sequence $\mathbf{a}$, the realizations over $\mathbf{a}$ correspond with (and are, in fact, provably equivalent to) the disjuncts $\theta^k_s$ in the normal form $\bigintd_{s\in ||\phi||}\theta^k_s$ for $\phi$. We must now show that every formula is provably equivalent to some formula of this form. Formally, we say that a $\intd\hspace{-1pt}$-distributive sequence $\mathbf{a}=\langle \eta_1, \ldots, \eta_n\rangle$ of a formula is in addition a \emph{$\intd\hspace{-1pt}$-distributive partition} of $\phi$ if $\phi$ is of the form $[\eta_1]\medcircle_1 [\eta_2]\medcircle_2\ldots \medcircle_{n-1} [\eta_n]$ where each $\medcircle_i$ is either $\wedge$ or $\vee$. We show:

\begin{lemma} \label{lemma:partitions}
For all $\phi$ there is a $\phi'$ such that $\phi'$ has a partition and $\phi\proveq \phi'$.
\end{lemma}

Before we prove this lemma, let us confirm that realizations over $\intd\hspace{-1pt}$-distributive partitions are indeed provably equivalent to strong Hintikka formulas $\theta_s^k$. Hereafter, to simplify notation, we write $\theta^k_\mathbf{0}:=\Bot$ and view $\mathbf{0}$ as a null state such that $\mathbf{0}\vDash \Bot$.

\begin{lemma} \label{lemma:instatitions_provably_equivalent_to_charfs}
    Let $\mathbf{a}$ be a $\intd\hspace{-1pt}$-distributive partition of $\phi$. Then for all $f\in \mathcal{F}^{\mathsf{X},\mathbf{a}}$, we have that $\phi^{k,f}\proveq \theta^{\mathsf{X},k}_s$ for some $s\in ||\Top||_{\mathsf{X}}$ or $s=\mathbf{0}$.
\end{lemma}
\begin{proof}
It suffices to show that for all  states $s$ and $t$ (including the null state), $\theta^k_s\land \theta^k_t\proveq \theta^k_u$ and $\theta^k_s\lor \theta^k_t\proveq \theta^k_v$ for some (null or non-null) states $u$ and $v$. If one of $s$ and $t$ is $\mathbf{0}$, then $\theta^k_s\land \theta^k_t\proveq \theta^k_{\mathbf{0}}$ and $\theta^k_s\lor\theta^k_t\proveq \theta^k_{\mathbf{0}}$ by Lemma \ref{lemma:Bot_E} and $\Bot \R{Ctr}$. If neither state is $\mathbf{0}$, then $\theta^k_s\lor \theta^k_t \proveq \theta^k_{s\uplus t}$. For $\theta^k_s\land \theta^k_t$, if $s\nbisim_k t$, then $\theta^k_s\land \theta^k_t\proveq \theta^k_{\mathbf{0}}$ by Lemmas \ref{lemma:not_bisimilar_implies_charfs_incompatible} and \ref{lemma:Bot_E}; if $s\bisim_k t$, then $\theta^k_s\proveq \theta^k_t$ by Lemma \ref{lemma:charf_provable_equivalence}, so clearly $\theta^k_s\land\theta^k_t\proveq \theta^k_{s}$.
\end{proof}

We now turn to Lemma \ref{lemma:partitions}. This result is, as it turns out, an easy corollary of the fact that all formulas of the form $\Di \phi$ or $\Bo \phi$ are provably equivalent to classical formulas. To establish this fact, we first recall that in Lemma \ref{lemma:modal_charfs_provably_classical} in Section \ref{section:BSMLI}, we proved that formulas of the form $\Di\theta^k_s$ or $\Bo\theta^k_s$ are provably equivalent to classical formulas. Lemma \ref{lemma:modal_charfs_provably_classical} depends on Lemma \ref{lemma:fc}, which in turns depends on Lemmas \ref{lemma:provability_results_NE} and \ref{lemma:di_NE_I}, which were given $\BSMLI$-specific derivations in Section \ref{section:BSMLI}. We have already shown Lemma \ref{lemma:provability_results_NE} for $\BSML$, and directly below we show Lemma \ref{lemma:di_NE_I}; we may then also make use of Lemmas $\ref{lemma:fc}$ and \ref{lemma:modal_charfs_provably_classical} in $\BSML$.

\begin{proof}[Proof of Lemma \ref{lemma:di_NE_I} (in $\BSMLBF$)]
By $\Di\bot\NE\R{Trs}$, $\Di\phi\vdash \Di(\phi\land \NE)\lor \Di(\phi \land \bot)$. We have $\Di(\phi\land \NE)\lor \Di(\phi \land \bot)\vdash \Di(\phi \land \NE)\lor \Di\bot\vdash \Di(\phi \land \NE)\lor \bot$ by Proposition \ref{prop:classical_completeness}, and finally $\Di(\phi \land \NE)\lor \bot\vdash \Di(\phi \land \NE)$ by $\bot \R{E}$.
\end{proof}

We also need the following modal analogue of Lemma \ref{lemma:NE_E}:

\begin{lemma} \label{lemma:Di_NE_E}
    Let $[\eta]$ be $\intd\hspace{-1pt}$-distributive in $\phi$. For any $k\in \mathbb{N}$ and any finite $\mathsf{X}\subseteq\mathsf{Prop}$:
    \begin{align*}
    &\Di\phi\vdash \bigdis_{s\in ||\Top||_{\mathsf{X}}} \Di\phi[\theta^{\mathsf{X},k}_s\land \eta/\eta]&&\text{ and }&&\Bo \phi\vdash \bigdis_{s\in ||\Top||_{\mathsf{X}}} \Bo\phi[\theta^{\mathsf{X},k}_s\land \eta/\eta].
    \end{align*}
\end{lemma}
\begin{proof}
We prove the result for $\Di \phi$. Consider $\chi^k_{\llbracket\Top\rrbracket}=\chi_{w_1}^k\vee\dots\vee\chi_{w_n}^k$, where $w_1,\dots,w_n$ are some worlds from some  models $M_1,\dots,M_n$. Since $\vdash \chi^k_{\llbracket\Top\rrbracket}$, we derive $\Di\phi\vdash \Di\phi[\chi^k_{\llbracket\Top\rrbracket}\land \eta/\eta]$ by Lemma \ref{lemma:provable_subformula_replacement}. Put $\phi':=\phi[\chi^k_{\llbracket\Top\rrbracket}\land \eta/\eta]$. By applying $\Di\bot\NE\R{Trs}$ repeatedly, we derive 
    \begin{align*}
        &\Di\phi'&&\vdash &&\Di\phi'[\chi^k_{w_1}\land \NE/\chi^k_{w_1}]\lor \Di\phi'[\chi^k_{w_1}\land \bot/\chi^k_{w_1}]\\
        &&&\vdash &&\Di\phi'[\chi^k_{w_1}\land \NE/\chi^k_{w_1}, \chi^k_{w_2}\land \NE/\chi^k_{w_2}]\lor \Di\phi'[\chi^k_{w_1}\land \NE/\chi^k_{w_1}, \chi^k_{w_2}\land \bot/\chi^k_{w_2}]\lor\\
        &&&&& \Di\phi'[\chi^k_{w_1}\land \bot/\chi^k_{w_1}, \chi^k_{w_2}\land \NE/\chi^k_{w_2}]\lor \Di\phi'[\chi^k_{w_1}\land \bot/\chi^k_{w_1}, \chi^k_{w_2}\land \bot/\chi^k_{w_2}]\\
                &&&\dots &&\\
        &&&\vdash &&\bigdis_{\tau_1,\dots,\tau_n\in\{\NE,\bot\}} \Di\phi'[\chi_{w_1}^k\land\tau_1/\chi_{w_1}^k,\dots,\chi_{w_n}^k\land\tau_n/\chi_{w_n}^k]\\
   &&&\vdash && \bigdis_{\tau_1,\dots,\tau_n\in\{\NE,\bot\}} \Di\phi'[(\chi_{w_1}^k\land\tau_1)\vee\dots\vee(\chi_{w_n}^k\land\tau_n)/\chi_{\llbracket\Top\rrbracket}^k]\\
      &&&\vdash && \bigdis_{\tau_1,\dots,\tau_n\in\{\NE,\bot\}} \Di\phi[\bigvee_{i=1}^n(\chi_{w_i}^k\land\tau_i)\wedge\eta/\eta]
    \end{align*}
 Now, consider each disjunct of the above formula, with some arbitrary fixed $\tau_1,\dots,\tau_n\in\{\NE,\bot\}$.  Let $s=\{w_i\mid \tau_i=\NE\}\subseteq\{w_1,\dots,w_n\}$. By applying $\bot \R{E}$ and $\dis\R{I}$, we derive $\bigvee_{i=1}^n(\chi_{w_i}^k\land\tau_i) \proveq \bigdis_{w_i\in s}(\chi^k_{w_i}\land \NE )= \theta^k_s$. It then follows by Lemma \ref{lemma:provable_subformula_replacement} that $\Di\phi[\bigvee_{i=1}^n(\chi_{w_i}^k\land\tau_i)\wedge\eta/\eta]\vdash\Di\phi[\theta_s^k\wedge\eta/\eta]$. We then clearly also have $\bigdis_{\tau_1,\dots,\tau_n\in\{\NE,\bot\}} \Di\phi[\bigvee_{i=1}^n(\chi_{w_i}^k\land\tau_i)\wedge\eta/\eta]\vdash \bigdis_{s\subseteq \{w_1,\ldots,w_n\}}\Di\phi[\theta_s^k\wedge\eta/\eta]$. Observe that $||\Top||_{\mathsf{X}}$ is (modulo $k$-bisimulation) $\{(M_1\uplus\dots\uplus M_n,s)\sepp s\subseteq \{w_1,\ldots, w_n\}\}$.
  Thus, we have $\bigdis_{s\subseteq \{w_1,\ldots,w_n\}}\Di\phi[\theta_s^k\wedge\eta/\eta]\vdash\bigdis_{s\in ||\Top||_{\mathsf{X}}} \Di\phi[\theta^k_s\land \eta/\eta].$ Putting all these together, we finally obtain $\Di\phi\vdash\bigdis_{s\in ||\Top||_{\mathsf{X}}} \Di\phi[\theta^k_s\land \eta/\eta]$.
 
The $\Bo\phi$-result is proved analogously, using $\Bo\bot\NE\R{Trs}$ in place of $\Di\bot\NE\R{Trs}$.
\end{proof}

We can now show that formulas of the form $\Di \phi$ or $\Bo \phi$ are provably equivalent to classical formulas, and then derive the partition lemma as a corollary.

\begin{lemma} \label{lemma:modals_provably_classical}
      $\Di\phi$ and $\Bo\phi$ are provably equivalent to classical formulas.
\end{lemma}
\begin{proof}
We prove the two results simultaneously but only give the details for $\Di\phi$; the details for $\Bo\phi$ are similar. We first prove that the result holds in case $\phi$ has a partition $\mathbf{a}$. Let $k\geq md(\phi)$, and let $\mathsf{X}\supseteq \mathsf{P}(\phi)$ be finite.  We first show that
    $$\Di\phi\proveq \bigdis_{f\in \FFF^{\mathsf{X},\mathbf{a}}} \Di \phi^{k,f}.$$
For the direction $\dashv$, for a given $f\in \mathcal{F}^\mathbf{a}$, we have $\phi^f\vdash \phi$ by Lemma \ref{lemma:realization_provable_equivalence} \ref{lemma:realization_provable_equivalence_i}, which implies  $\Di\phi^f\vdash \Di\phi$; therefore $\bigdis_{f\in \FFF^\mathbf{a}} \Di \phi^f\vdash \Di\phi$ follows by $\dis \R{E}$. For the converse direction $\vdash$, by repeated applications of Lemma \ref{lemma:Di_NE_E}, we have
$$\Di\phi\vdash \bigdis_{s_1\in ||\Top||_{\mathsf{X}}}\ldots \bigdis_{s_n\in ||\Top||_{\mathsf{X}}} \Di\phi[\theta^k_{s_1}\land \eta_1/\eta_1,\ldots, \theta^k_{s_n}\land \eta_n/\eta_n].$$ By Lemma \ref{lemma:theta_and_formula_implies_realization_for_formula} \ref{lemma:theta_and_formula_implies_realization_for_formula_i} and \ref{lemma:theta_and_formula_implies_realization_for_formula_ii} together with $\dis\R{Mon}$ and $\Di\R{Mon}$, we derive
$$\bigdis_{s_1\in ||\Top||_{\mathsf{X}}}\ldots \bigdis_{s_n\in ||\Top||_{\mathsf{X}}}\Di\phi[\theta^k_{s_1}\land \eta_1/\eta_1,\ldots, \theta^k_{s_n}\land \eta_n/\eta_n]\vdash  \bigdis_{f\in \FFF^\mathbf{a}} \Di \phi^f.$$

It now suffices to show that $\bigdis_{f\in \FFF^\mathbf{a}} \Di \phi^f$ is provably equivalent to a classical formula. We show this by showing that each disjunct $\Di\phi^f$ is so. By Lemma \ref{lemma:instatitions_provably_equivalent_to_charfs}, we have $\phi^f\proveq \theta^k_s$ for some  state $s$. Thus, also $\Di\phi^f\proveq \Di\theta^k_s$. If $s=\mathbf{0}$, then $\Di\theta^k_s=\Di\Bot\proveq\Di \bot \proveq \bot$ by Proposition \ref{prop:classical_completeness} and Lemma \ref{lemma:di_NE_I}. If $s\neq \mathbf{0}$, then by Lemma \ref{lemma:modal_charfs_provably_classical} \ref{lemma:modal_charfs_provably_classical_Di}, we know that $\Di\theta^k_s$ is provably equivalent to a classical formula. 

We now show the general case by induction on the modal depth of $\phi$. Note that by $\lnot\NE\R{E}$, we may assume w.l.o.g. that $\phi=\phi(\bot/\lnot\NE)$. If $md(\phi)=0$, there is clearly a partition of $\phi$, so the result follows from what we just proved. If $md(\phi)=n+1$, then for all subformulas of $\phi$ of form $\Di\psi$ or $\Bo\psi$, we have $md(\psi)\leq n$, and thus by induction hypothesis $\Di \psi\proveq \alpha_{\Di\psi}$ and $\Bo \psi\proveq \alpha_{\Bo\psi}$ for some $\alpha_{\Di\psi},\alpha_{\Bo\psi}\in \ML$. By Lemma \ref{lemma:provable_subformula_replacement}, $\phi \proveq \phi':= \phi[\alpha_{\medcircle_1 \psi_1}/\medcircle_1 \psi_1,\ldots , \alpha_{\medcircle_n \psi_n}/\medcircle_n \psi_n]$, where $[\medcircle_1 \psi_1],\ldots , [\medcircle_n \psi_n]$ are all the subformula occurrences of form $\Di\psi$ or $\Bo \psi$ in $\phi$ such that no $[\medcircle_i \psi_i]$ is a suboccurrence of any other occurrence of the form $\Di \psi$ or $\Bo \psi$ in $\phi$. The formula $\phi'$ clearly has a partition, so $\Di\phi'\proveq \alpha \in \ML$; then also $\Di\phi \proveq \alpha$.
\end{proof}
    
\begin{proof}[Proof of Lemma \ref{lemma:partitions}]
    Apply Lemma \ref{lemma:modals_provably_classical}, $\lnot \NE\R{E}$, and Lemma \ref{lemma:provable_subformula_replacement} to replace all subformula occurrences of form $\Di\psi$, $\Bo\psi$, or $\lnot \NE$ in $\phi$ with classical formulas.
\end{proof}

We are now ready to prove the completeness theorem for the $\BSML$-system.

\begin{theorem}[Completeness of \BSMLBF] \label{theorem:BSML_completeness}
    If $\Phi \vDash \psi$, then $\Phi\vdash \psi$.
\end{theorem}
\begin{proof}
    Assume $\Phi\nvdash \psi$. We  show $\Phi\nvDash \psi$. 
    Let $\Phi=\{\phi_i\sepp i \in I\}$. By Lemma \ref{lemma:partitions}, we may assume that each $\phi_i$ has a partition $\mathbf{a}_i$ and $\psi$ has a partition $\mathbf{a}$. For each $i$, let $k_i:=max\{md(\phi_i),md(\psi)\}$ and  $\mathsf{X}_i:=\mathsf{P}(\phi_i)\cup \mathsf{P}(\psi)$. Now, for an arbitrary $i\in I$, since $\Phi\backslash\{\phi_i\},\phi_i\nvdash \psi$, by Lemma \ref{lemma:realization_provable_equivalence} \ref{lemma:realization_provable_equivalence_ii}, we know there is a realizing function $f_i\in \FFF^{\mathsf{X}_i,\mathbf{a}_i}$ such that $\Phi\backslash\{\phi_i\},\phi_i^f\nvdash \psi$. Continuing to argue in a similar way, one can find for each $i\in I$ a realizing function $f_i\in \FFF^{\mathsf{X}_i,\mathbf{a}_i}$ such that  $\{\phi_i^{k_i,f_i} \sepp f_i\in \FFF^{\mathsf{X}_i,\mathbf{a}_i},i\in I\}\nvdash  \psi$. By Lemma \ref{lemma:instatitions_provably_equivalent_to_charfs}, each $\phi_i^{k_i,f_i}$ is provably equivalent to $\theta_{s_i}^{\mathsf{X}_i,k_i}$ for some $s_i$. Thus, $\Phi'=\{\theta_{s_i}^{\mathsf{X}_i,k_i} \sepp i\in I\}\nvdash  \psi$ as well.

    Observe that $\Phi'\nvDash\Bot$, since otherwise either $s_i=\mathbf{0}$ for some $i\in I$, or $s_i\nbisim^\mathsf{M}_m s_j$ for some $i,j\in I$ and for $m=min\{k_i,k_j\}$ and $\mathsf{M}=\mathsf{X}_i\cap \mathsf{X}_j$. In the former case, $\theta_{s_i}^{\mathsf{X}_i,k_i}=\Bot$ and thus $\Phi'\vdash\psi$; a contradiction. In the latter case, $\theta^{{\mathsf{M},m}}_{s_i},\theta^{{\mathsf{M},m}}_{s_j}\vdash\Bot$ by Lemma \ref{lemma:not_bisimilar_implies_charfs_incompatible}, so that by Lemma \ref{lemma:charf_proves_lower_md_charf}, $\theta^{\mathsf{X}_i,k_i}_{s_i},\theta^{\mathsf{X}_j,k_j}_{s_j}\vdash\Bot$. Thus, $\Phi'\vdash\psi$; a contradiction again.

    Now, let $t$ be such that $t\vDash\Phi'$. For each $i\in I$, we have $t\vDash\theta_{s_i}^{k_i}\proveq\phi_i^{f_i}$, and thus $t\vDash\phi_i^{f_i}$ by soundness. By Lemma \ref{lemma:realization_equivalence}, we further conclude $t\vDash \phi_i$, whereby $t\vDash\Phi$. To show $\Phi\nvDash\psi$, it then suffices to show $t\nvDash\psi$. Assume otherwise. Take $i\in I$. By Lemma \ref{lemma:realization_equivalence}, $t\vDash \psi^{k_i,g}$ for some $g\in \FFF^{\mathsf{X}_i,\mathbf{a}}$. In view of Lemma \ref{lemma:instatitions_provably_equivalent_to_charfs}, we have $t\vDash \theta_r^{\mathsf{X}_i,k_i}$ for some (non-null) state $r$ with $\psi^{k_i,g}\proveq\theta_r^{\mathsf{X}_i,k_i}$. On the other hand, since $t\vDash\Phi'$, we have $t\vDash\theta_{s_i}^{\mathsf{X}_i,k_i}$. Then, by Proposition \ref{prop:characteristic_formulas_characterize_states}\ref{prop:characteristic_formulas_characterize_states_theta}, we have $r \bisim^{\mathsf{X}_i}_{k_i} t\bisim^{\mathsf{X}_i}_{k_i} s_i$. Therefore, by Lemma \ref{lemma:charf_provable_equivalence}, $\theta^{\mathsf{X}_i,k_i}_{s_i}\vdash \theta^{\mathsf{X}_i,k_i}_{r}\vdash \psi^{k_i,g}$. As $\psi^{k_i,g}\vdash \psi$ (by Lemma \ref{lemma:realization_provable_equivalence} \ref{lemma:realization_provable_equivalence_i}), we are forced to conclude that $\Phi'\vdash \psi$, which is a contradiction.
\end{proof}

%% file: sections/conclusion.tex
\section{Conclusion} \label{section:conclusion}
In this paper, we presented natural deduction axiomatizations for $\BSML$ and the two extensions we introduced, $\BSMLI$ ($\BSML$ with the global/inquisitive disjunction $\intd$) and $\BSMLO$ ($\BSML$ with the novel emptiness operator $\OC$). We also proved the expressive completeness of the two extensions: $\BSMLI$ is expressively complete for the class of all state properties invariant under bounded bisimulation; $\BSMLO$ for the class of all union-closed state properties invariant under bounded bisimulation. We conclude by noting an additional preliminary result known to us, and listing possible directions for further investigation.

We saw that \BSML is union closed but not expressively complete for union-closed properties. It appears, however, that one can find a different natural class of state properties for which \BSML \emph{is} complete: according to a very recent unpublished result \cite{aknudstorp2024}, \BSML is complete for the class of all properties $\PPP$ such that $\PPP$ is invariant under bounded bisimulation, union-closed, and also \emph{convex}, where $\PPP$ is convex if $[(M,s)\in \PPP$ and $(M,t)\in \PPP$ and $t\subseteq u\subseteq s ]$ implies $(M,u)\in\PPP$.

The semantics of the quantifiers in first-order dependence logic are analogous to the semantics of the modalities used in modal dependence logic, which are distinct from the \BSML-modalities $\Di$ and $\Bo$. Defining first-order quantifiers whose semantics are analogous to those of $\Di$ and $\Bo$ may yield first-order state-based logics with interesting properties.

In inquisitive semantics, the global disjunction $\intd$ is used to model the meanings of questions. Similarly, in the context of the extension $\BSMLI$, this disjunction might allow one to model questions and their interaction with free choice phenomena \ref{ex:qfc}, and to arrive at a full account of examples of overt free choice cancellation \ref{ex:sluice},  which were left open in \cite{aloni2022}.

\ex. \label{ex:qfc} Can I have coffee or tea? \quad $\Diamond (c \vee t) \intd \neg \Diamond (c \vee t)$
\a. Yes $\implicates$ you can have coffee and you can have tea
\b. No $\implicates$ you cannot have either

\ex. \label{ex:sluice} You may have coffee or tea, I don't know which.

Finally, given that the emptiness operator $\OC$ is a natural counterpart to $\NE$ and that it can be used to cancel out the effects of pragmatic enrichment in the sense that $\OC(\alpha \land \NE)\equiv \alpha$, investigating the applications of this operator in formal semantics may yield some interesting results.

%% file: sections/bibliography.tex
\bibliography{bibb}

%% file: main.bbl
\begin{thebibliography}{57}
\expandafter\ifx\csname natexlab\endcsname\relax\def\natexlab#1{#1}\fi
\def\docolon{:}
\def\eatcomma#1{}
\def\onlyone#1{\gdef\oneletter{#1}}
\def\sphref#1#2{{\let\#=\docolon\xdef\one{#1}}\href{\one}{#2}}
\def\zhref#1,#2{{\let\#=\docolon\xdef\one{#1}}\href{\one}{#2}}
\expandafter\ifx\csname url\endcsname\relax
  \def\url#1{{\tt #1}}\fi
\newcommand{\enquote}[2]{``#1,''}

\bibitem[Aher(2011)]{aher2011}
Aher, M.,
\newblock \enquote{Free choice in deontic inquisitive semantics ({DIS})},
\newblock in {\em Proceedings of the 18th Amsterdam Colloquim Conference on Logic, Language and Meaning}, AC'11, pp.~22--–31, Berlin, Heidelberg, 2011. Springer-Verlag,
\newblock URL \url{https://doi.org/10.1007/978-3-642-31482-7_3}.

\bibitem[Aloni(2007)]{aloni2007}
Aloni, M.,
\newblock \enquote{Free choice, modals and imperatives},
\newblock {\em Natural Language Semantics}, vol.~15 (2007), pp.~65--94.\eatcomma,
\newblock URL \url{https://doi.org/10.1007/s11050-007-9010-2}.

\bibitem[Aloni(2022)]{aloni2022}
Aloni, M.,
\newblock \enquote{Logic and conversation: the case of free choice},
\newblock {\em Semantics and Pragmatics}, vol.~15 (2022).\eatcomma,
\newblock URL \url{https://doi.org/10.3765/sp.15.5}.

\bibitem[Aloni and Ciardelli(2013)]{AloniCiardelli2013}
Aloni, M. \unskip, and  I.~Ciardelli,
\newblock \enquote{A logical account of free choice imperatives},
\newblock pp.~1--17\eatcomma,  in {\em The Dynamic, Inquisitive, and Visionary Life of $\phi$, $?\phi$, and $\Diamond\phi$},
\newblock Institute for Logic, Language and Computation, 2013,
\newblock URL \url{https://festschriften.illc.uva.nl/Festschrift-JMF/}.

\bibitem[Alonso-Ovalle(2006)]{alonso2006}
Alonso-Ovalle, L.,
\newblock {\em Disjunction in Alternative Semantics},
\newblock PhD thesis, University of Massachusetts, Amherst, 2006,
\newblock URL \url{https://scholarworks.umass.edu/dissertations/AAI3242324/}.

\bibitem[Anderson and Belnap(1975)]{anderson1975}
Anderson, A.~R. \unskip, and  N.~D. Belnap, Jr.,
\newblock {\em Entailment},
\newblock Princeton University Press, Princeton, N.J.-London, 1975.\eatcomma,
\newblock Volume I: The logic of relevance and necessity, With contributions by J. Michael Dunn and Robert K. Meyer, and further contributions by John R. Chidgey, J. Alberto Coffa, Dorothy L. Grover, Bas van Fraassen, Hugues LeBlanc, Storrs McCall, Zane Parks, Garrel Pottinger, Richard Routley, Alasdair Urquhart and Robert G. Wolf. \sphref{http://www.ams.org/mathscinet-getitem?mr=406756}{\hbox{MR 406756}}.

\bibitem[Anttila(2021)]{anttila2021}
Anttila, A.,
\newblock \enquote{The logic of free choice. axiomatizations of state-based modal logics},
\newblock M{S}c thesis, University of Amsterdam, 2021,
\newblock URL \url{https://eprints.illc.uva.nl/id/eprint/1788/}.

\bibitem[Anttila(2024)]{anttila2024}
Anttila, A.,
\newblock \enquote{Further remarks on the dual negation in team logics}, 2024. Manuscript.

\bibitem[Anttila et~al.(2023)Anttila, Häggblom, and Yang]{anttila2023axiomatizingmodalinclusionlogic}
Anttila, A., M.~Häggblom \unskip, and  F.~Yang,
\newblock \enquote{Axiomatizing modal inclusion logic and its variants}, 2023,
\newblock URL \url{https://arxiv.org/abs/2312.02285}.

\bibitem[Belnap(1977{\natexlab{a}})]{belnap19772}
Belnap, N.~D., Jr.,
\newblock \enquote{How a computer should think}\eatcomma,
\newblock  in {\em Contemporary aspects of philosophy},
\newblock Oriel Press, 1977{\natexlab{a}}.

\bibitem[Belnap(1977{\natexlab{b}})]{belnap1977}
Belnap, N.~D., Jr.,
\newblock \enquote{A useful four-valued logic},
\newblock pp.~5--37\eatcomma,  in {\em Modern uses of multiple-valued logic ({F}ifth {I}nternat. {S}ympos., {I}ndiana {U}niv., {B}loomington, {I}nd., 1975)}, volume~Vol. 2 of {\em Episteme},
\newblock Reidel, Dordrecht-Boston, Mass., 1977{\natexlab{b}},
\newblock URL \url{https://doi.org/10.1007/978-94-010-1161-7_2}. \sphref{http://www.ams.org/mathscinet-getitem?mr=485167}{\hbox{MR 485167}}.

\bibitem[Blackburn et~al.(2001)Blackburn, de~Rijke, and Venema]{blackburn2001}
Blackburn, P., M.~de~Rijke \unskip, and  Y.~Venema,
\newblock {\em Modal logic}, volume~53 of {\em Cambridge Tracts in Theoretical Computer Science},
\newblock Cambridge University Press, Cambridge, 2001,
\newblock URL \url{https://doi.org/10.1017/CBO9781107050884}.\eatcomma,
\newblock \sphref{http://www.ams.org/mathscinet-getitem?mr=1837791}{\hbox{MR 1837791}}.

\bibitem[Bott et~al.(2019)Bott, Schlotterbeck, and Klein]{Bott-Schlotterbeck-Klein:19}
Bott, O., F.~Schlotterbeck \unskip, and  U.~Klein,
\newblock \enquote{Empty-set effects in quantifier interpretation},
\newblock {\em Journal of Semantics}, vol.~36 (2019), pp.~99--163.\eatcomma,
\newblock URL \url{https://doi.org/10.1093/jos/ffy015}.

\bibitem[Burgess(2003)]{burgess2003}
Burgess, J.~P.,
\newblock \enquote{A remark on {H}enkin sentences and their contraries},
\newblock {\em Notre Dame J. Formal Logic}, vol.~44 (2003), pp.~185--188.\eatcomma,
\newblock URL \url{https://doi.org/10.1305/ndjfl/1091030856}. \sphref{http://www.ams.org/mathscinet-getitem?mr=2130790}{\hbox{MR 2130790}}.

\bibitem[Ciardelli(2016)]{ciardelli2016}
Ciardelli, I.,
\newblock {\em Questions in Logic},
\newblock PhD thesis, University of Amsterdam, 2016,
\newblock URL \url{https://hdl.handle.net/11245/1.518411}.

\bibitem[Ciardelli(2022)]{ciardellibook}
Ciardelli, I.,
\newblock {\em Inquisitive logic---consequence and inference in the realm of questions}, volume~60 of {\em Trends in Logic---Studia Logica Library},
\newblock Springer, Cham, 2022,
\newblock URL \url{https://doi.org/10.1007/978-3-031-09706-5}.\eatcomma,
\newblock \sphref{http://www.ams.org/mathscinet-getitem?mr=4719694}{\hbox{MR 4719694}}.

\bibitem[Ciardelli et~al.(2020)Ciardelli, Iemhoff, and Yang]{ciardelliIemhoffYang}
Ciardelli, I., R.~Iemhoff \unskip, and  F.~Yang,
\newblock \enquote{Questions and dependency in intuitionistic logic},
\newblock {\em Notre Dame J. Form. Log.}, vol.~61 (2020), pp.~75--115.\eatcomma,
\newblock URL \url{https://doi.org/10.1215/00294527-2019-0033}. \sphref{http://www.ams.org/mathscinet-getitem?mr=4054246}{\hbox{MR 4054246}}.

\bibitem[Ciardelli and Otto(2021)]{ciardelli2021}
Ciardelli, I. \unskip, and  M.~Otto,
\newblock \enquote{Inquisitive bisimulation},
\newblock {\em J. Symb. Log.}, vol.~86 (2021), pp.~77--109.\eatcomma,
\newblock URL \url{https://doi.org/10.1017/jsl.2020.77}. \sphref{http://www.ams.org/mathscinet-getitem?mr=4282699}{\hbox{MR 4282699}}.

\bibitem[Ciardelli and Roelofsen(2011)]{CiardelliRoelofsen2011}
Ciardelli, I. \unskip, and  F.~Roelofsen,
\newblock \enquote{Inquisitive logic},
\newblock {\em J. Philos. Logic}, vol.~40 (2011), pp.~55--94.\eatcomma,
\newblock URL \url{https://doi.org/10.1007/s10992-010-9142-6}. \sphref{http://www.ams.org/mathscinet-getitem?mr=2764121}{\hbox{MR 2764121}}.

\bibitem[D'Agostino(2019)]{dagostino}
D'Agostino, G.,
\newblock \enquote{Uniform interpolation for propositional and modal team logics},
\newblock {\em J. Logic Comput.}, vol.~29 (2019), pp.~785--802.\eatcomma,
\newblock URL \url{https://doi.org/10.1093/logcom/exz006}. \sphref{http://www.ams.org/mathscinet-getitem?mr=4009525}{\hbox{MR 4009525}}.

\bibitem[Dunn(1976)]{dunn1976}
Dunn, J.~M.,
\newblock \enquote{Intuitive semantics for first-degree entailments and `coupled trees'},
\newblock {\em Philos. Studies}, vol.~29 (1976), pp.~149--168.\eatcomma,
\newblock URL \url{https://doi.org/10.1007/bf00373152}. \sphref{http://www.ams.org/mathscinet-getitem?mr=491026}{\hbox{MR 491026}}.

\bibitem[Fine(2017)]{fine2017}
Fine, K.,
\newblock \enquote{Truthmaker semantics},
\newblock pp.~556--577\eatcomma,  in {\em A Companion to the Philosophy of Language},
\newblock Wiley, 2017,
\newblock URL \url{https://doi.org/10.1002/9781118972090.ch22}.

\bibitem[Goranko and Otto(2007)]{goranko2007}
Goranko, V. \unskip, and  M.~Otto,
\newblock \enquote{Model theory of modal logic},
\newblock pp.~249--329\eatcomma,  in {\em Handbook of modal logic}, volume~3 of {\em Stud. Log. Pract. Reason.},
\newblock Elsevier B. V., Amsterdam, 2007,
\newblock URL \url{https://doi.org/10.1016/S1570-2464(07)80008-5}. \sphref{http://www.ams.org/mathscinet-getitem?mr=3618507}{\hbox{MR 3618507}}.

\bibitem[Hella et~al.(2014)Hella, Luosto, Sano, and Virtema]{hella2014}
Hella, L., K.~Luosto, K.~Sano \unskip, and  J.~Virtema,
\newblock \enquote{The expressive power of modal dependence logic},
\newblock pp.~294--312\eatcomma,  in {\em Advances in modal logic. {V}ol. 10},
\newblock Coll. Publ., London, 2014,
\newblock URL \url{http://www.aiml.net/volumes/volume10/Hella-Luosto-Sano-Virtema.pdf}. \sphref{http://www.ams.org/mathscinet-getitem?mr=3329828}{\hbox{MR 3329828}}.

\bibitem[Hella and Stumpf(2015)]{hella2015}
Hella, L. \unskip, and  J.~Stumpf,
\newblock \enquote{The expressive power of modal logic with inclusion atoms},
\newblock in {\em Proceedings {S}ixth {I}nternational {S}ymposium on {G}ames, {A}utomata, {L}ogics and {F}ormal {V}erification}, volume~193 of {\em Electron. Proc. Theor. Comput. Sci. (EPTCS)}, pp.~129--143. EPTCS, 2015,
\newblock URL \url{https://doi.org/10.4204/EPTCS.193.10},
\newblock \sphref{http://www.ams.org/mathscinet-getitem?mr=3591917}{\hbox{MR 3591917}}.

\bibitem[Hintikka(1996)]{hintikka1996}
Hintikka, J.,
\newblock {\em The principles of mathematics revisited},
\newblock Cambridge University Press, Cambridge, 1996,
\newblock URL \url{https://doi.org/10.1017/CBO9780511624919}.\eatcomma,
\newblock With an appendix by Gabriel Sandu. \sphref{http://www.ams.org/mathscinet-getitem?mr=1410063}{\hbox{MR 1410063}}.

\bibitem[Hintikka and Sandu(1989)]{hintikka1989}
Hintikka, J. \unskip, and  G.~Sandu,
\newblock \enquote{Informational independence as a semantical phenomenon},
\newblock pp.~571--589\eatcomma,  in {\em Logic, methodology and philosophy of science, {VIII} ({M}oscow, 1987)}, volume~126 of {\em Stud. Logic Found. Math.},
\newblock North-Holland, Amsterdam, 1989,
\newblock URL \url{https://doi.org/10.1016/S0049-237X(08)70066-1}. \sphref{http://www.ams.org/mathscinet-getitem?mr=1034575}{\hbox{MR 1034575}}.

\bibitem[Hodges(1997{\natexlab{a}})]{hodges1997}
Hodges, W.,
\newblock \enquote{Compositional semantics for a language of imperfect information},
\newblock {\em Log. J. IGPL}, vol.~5 (1997), pp.~539--563.\eatcomma,
\newblock URL \url{https://doi.org/10.1093/jigpal/5.4.539}. \sphref{http://www.ams.org/mathscinet-getitem?mr=1465612}{\hbox{MR 1465612}}.

\bibitem[Hodges(1997{\natexlab{b}})]{hodges1997b}
Hodges, W.,
\newblock \enquote{Some strange quantifiers},
\newblock pp.~51--65\eatcomma,  in {\em Structures in logic and computer science}, volume~1261 of {\em Lecture Notes in Comput. Sci.},
\newblock Springer, Berlin, 1997{\natexlab{b}},
\newblock URL \url{https://doi.org/10.1007/3-540-63246-8_4}. \sphref{http://www.ams.org/mathscinet-getitem?mr=1638352}{\hbox{MR 1638352}}.

\bibitem[Humberstone(1981)]{humberstone1981}
Humberstone, I.~L.,
\newblock \enquote{From worlds to possibilities},
\newblock {\em J. Philos. Logic}, vol.~10 (1981), pp.~313--399.\eatcomma,
\newblock URL \url{https://doi.org/10.1007/BF00293423}. \sphref{http://www.ams.org/mathscinet-getitem?mr=628093}{\hbox{MR 628093}}.

\bibitem[Incurvati and Schl\"{o}der(2017)]{incurvati2017}
Incurvati, L. \unskip, and  J.~J. Schl\"{o}der,
\newblock \enquote{Weak rejection},
\newblock {\em Australasian Journal of Philosophy}, vol.~95 (2017), pp.~741--760.\eatcomma,
\newblock URL \url{https://doi.org/10.1080/00048402.2016.1277771}.

\bibitem[Kamp(1973/74)]{kamp1973}
Kamp, H.,
\newblock \enquote{Free choice permission},
\newblock {\em Proc. Aristotelian Soc. (N.S.)}, vol.~74 (1973/74), pp.~57--74.\eatcomma,
\newblock URL \url{https://doi.org/10.1093/aristotelian/74.1.57}. \sphref{http://www.ams.org/mathscinet-getitem?mr=419161}{\hbox{MR 419161}}.

\bibitem[Knudstorp(2023)]{aknudstorp2024}
Knudstorp, S.~B.,
\newblock \enquote{The expressive completeness of {B}ilateral {S}tate-based {M}odal {L}ogic}, 2023. Manuscript.

\bibitem[Kontinen et~al.(2015)Kontinen, M\"{u}ller, Schnoor, and Vollmer]{kontinen2014}
Kontinen, J., J.-S. M\"{u}ller, H.~Schnoor \unskip, and  H.~Vollmer,
\newblock \enquote{A van {B}enthem theorem for modal team semantics},
\newblock pp.~277--291\eatcomma,  in {\em 24th {EACSL} {A}nnual {C}onference on {C}omputer {S}cience {L}ogic}, volume~41 of {\em LIPIcs. Leibniz Int. Proc. Inform.},
\newblock Schloss Dagstuhl. Leibniz-Zent. Inform., Wadern, 2015,
\newblock URL \url{https://doi.org/10.4230/LIPIcs.CSL.2015.277}. \sphref{http://www.ams.org/mathscinet-getitem?mr=3441768}{\hbox{MR 3441768}}.

\bibitem[Kontinen et~al.(2017)Kontinen, M\"{u}ller, Schnoor, and Vollmer]{kontinen20142}
Kontinen, J., J.-S. M\"{u}ller, H.~Schnoor \unskip, and  H.~Vollmer,
\newblock \enquote{Modal independence logic:},
\newblock {\em J. Logic Comput.}, vol.~27 (2017), pp.~1333--1352.\eatcomma,
\newblock URL \url{https://doi.org/10.1093/logcom/exw019}. \sphref{http://www.ams.org/mathscinet-getitem?mr=3802105}{\hbox{MR 3802105}}.

\bibitem[Kontinen and V\"{a}\"{a}n\"{a}nen(2011)]{kontinen2011}
Kontinen, J. \unskip, and  J.~V\"{a}\"{a}n\"{a}nen,
\newblock \enquote{A remark on negation in dependence logic},
\newblock {\em Notre Dame J. Form. Log.}, vol.~52 (2011), pp.~55--65.\eatcomma,
\newblock URL \url{https://doi.org/10.1215/00294527-2010-036}. \sphref{http://www.ams.org/mathscinet-getitem?mr=2747162}{\hbox{MR 2747162}}.

\bibitem[Lück(2020)]{luck2020}
Lück, M.,
\newblock {\em Team Logic: Axioms, Expressiveness, Complexity},
\newblock PhD thesis, Leibniz University Hannover, 2020,
\newblock URL \url{https://doi.org/10.15488/9376}.

\bibitem[Müller(2014)]{muller2014}
Müller, J.-S.,
\newblock {\em Satisfiability and Model Checking in Team Based Logics},
\newblock PhD thesis, Leibniz University Hannover, 2014.

\bibitem[Nygren(2023)]{nygren}
Nygren, K.,
\newblock \enquote{Free choice in modal inquisitive logic},
\newblock {\em J. Philos. Logic}, vol.~52 (2023), pp.~347--391.\eatcomma,
\newblock URL \url{https://doi.org/10.1007/s10992-022-09674-4}. \sphref{http://www.ams.org/mathscinet-getitem?mr=4568863}{\hbox{MR 4568863}}.

\bibitem[Price(1983)]{price1983}
Price, H.,
\newblock \enquote{Sense, assertion, {D}ummett and denial},
\newblock {\em Mind}, vol.~92 (1983), pp.~161--173.\eatcomma,
\newblock URL \url{https://doi.org/10.1093/mind/XCII.366.161}. \sphref{http://www.ams.org/mathscinet-getitem?mr=702583}{\hbox{MR 702583}}.

\bibitem[Restall(2005)]{restall2005}
Restall, G.,
\newblock \enquote{Multiple conclusions},
\newblock pp.~189--205\eatcomma,  in {\em Logic, Methodology and Philosophy of Science},
\newblock College Publications, 2005,
\newblock URL \url{https://consequently.org/papers/multipleconclusions.pdf}.

\bibitem[Rumfitt(2000)]{rumfitt2000}
Rumfitt, I.,
\newblock \enquote{``{Y}es'' and ``{N}o''},
\newblock {\em Mind}, vol.~109 (2000), pp.~781--823.\eatcomma,
\newblock URL \url{https://doi.org/10.1093/mind/109.436.781}. \sphref{http://www.ams.org/mathscinet-getitem?mr=2131446}{\hbox{MR 2131446}}.

\bibitem[Schroeder-Heister(2023)]{sep-proof-theoretic-semantics}
Schroeder-Heister, P.,
\newblock \enquote{{Proof-Theoretic Semantics}}\eatcomma,
\newblock  in {\em The {Stanford} Encyclopedia of Philosophy},
\newblock Metaphysics Research Lab, Stanford University, {F}all 2023 edition, 2023,
\newblock URL \url{https://plato.stanford.edu/archives/fall2023/entries/proof-theoretic-semantics/}.

\bibitem[Sevenster(2009)]{sevenster2009}
Sevenster, M.,
\newblock \enquote{Model-theoretic and computational properties of modal dependence logic},
\newblock {\em J. Logic Comput.}, vol.~19 (2009), pp.~1157--1173.\eatcomma,
\newblock URL \url{https://doi.org/10.1093/logcom/exn102}. \sphref{http://www.ams.org/mathscinet-getitem?mr=2565920}{\hbox{MR 2565920}}.

\bibitem[Smiley(1996)]{smiley1996}
Smiley, T.,
\newblock \enquote{Rejection},
\newblock {\em Analysis (Oxford)}, vol.~56 (1996), pp.~1--9.\eatcomma,
\newblock URL \url{https://doi.org/10.1111/j.0003-2638.1996.00001.x}. \sphref{http://www.ams.org/mathscinet-getitem?mr=1379622}{\hbox{MR 1379622}}.

\bibitem[Van~Fraassen(1969)]{vfraassen1969}
Van~Fraassen, B.~C.,
\newblock \enquote{Facts and tautological entailments},
\newblock {\em The Journal of Philosophy}, vol.~66 (1969), pp.~477--487.\eatcomma,
\newblock URL \url{https://doi.org/10.2307/2024563}.

\bibitem[von Wright(1968)]{vwright1968}
von Wright, G.~H.,
\newblock \enquote{An essay in deontic logic and the general theory of action. {W}ith a bibliography of deontic and imperative logic},
\newblock {\em Acta Philos. Fenn.},  (1968), p.~110.\eatcomma. \sphref{http://www.ams.org/mathscinet-getitem?mr=290940}{\hbox{MR 290940}}.

\bibitem[Väänänen(2007)]{vaananen2007}
Väänänen, J.,
\newblock {\em Dependence logic}, volume~70 of {\em London Mathematical Society Student Texts},
\newblock Cambridge University Press, Cambridge, 2007,
\newblock URL \url{https://doi.org/10.1017/CBO9780511611193}.\eatcomma,
\newblock A new approach to independence friendly logic. \sphref{http://www.ams.org/mathscinet-getitem?mr=2351449}{\hbox{MR 2351449}}.

\bibitem[Väänänen(2008)]{vaananen2008}
Väänänen, J.,
\newblock \enquote{Modal dependence logic},
\newblock pp.~237--254\eatcomma,  in {\em New perspectives on games and interaction}, volume~4 of {\em Texts Log. Games},
\newblock Amsterdam Univ. Press, Amsterdam, 2008,
\newblock URL \url{https://hdl.handle.net/11245/1.300509}. \sphref{http://www.ams.org/mathscinet-getitem?mr=2985117}{\hbox{MR 2985117}}.

\bibitem[Väänänen(2014)]{vaananen2014}
Väänänen, J.,
\newblock \enquote{Multiverse set theory and absolutely undecidable propositions},
\newblock pp.~180--208\eatcomma,  in {\em Interpreting {G}\"{o}del},
\newblock Cambridge Univ. Press, Cambridge, 2014,
\newblock URL \url{https://doi.org/10.1017/CBO9780511756306.013}. \sphref{http://www.ams.org/mathscinet-getitem?mr=3468187}{\hbox{MR 3468187}}.

\bibitem[Wansing(2010)]{wansing2010}
Wansing, H.,
\newblock \enquote{Proofs, disproofs, and their duals},
\newblock pp.~483--505\eatcomma,  in {\em Advances in modal logic. {V}olume 8},
\newblock Coll. Publ., London, 2010,
\newblock URL \url{http://www.aiml.net/volumes/volume8/Wansing.pdf}. \sphref{http://www.ams.org/mathscinet-getitem?mr=2808459}{\hbox{MR 2808459}}.

\bibitem[Wansing and Ayhan(2023)]{wansingayhan23}
Wansing, H. \unskip, and  S.~Ayhan,
\newblock \enquote{Logical multilateralism},
\newblock {\em J. Philos. Logic}, vol.~52 (2023), pp.~1603--1636.\eatcomma,
\newblock URL \url{https://doi.org/10.1007/s10992-023-09720-9}. \sphref{http://www.ams.org/mathscinet-getitem?mr=4670072}{\hbox{MR 4670072}}.

\bibitem[Willer(2018)]{willer2018}
Willer, M.,
\newblock \enquote{Simplifying with free choice},
\newblock {\em Topoi}, vol.~37 (2018), pp.~379--392.\eatcomma,
\newblock URL \url{https://doi.org/10.1007/s11245-016-9437-5}. \sphref{http://www.ams.org/mathscinet-getitem?mr=3845804}{\hbox{MR 3845804}}.

\bibitem[Yang(2017)]{yang20172}
Yang, F.,
\newblock \enquote{Modal dependence logics: axiomatizations and model-theoretic properties},
\newblock {\em Log. J. IGPL}, vol.~25 (2017), pp.~773--805.\eatcomma,
\newblock URL \url{https://doi.org/10.1093/jigpal/jzx023}. \sphref{http://www.ams.org/mathscinet-getitem?mr=3798986}{\hbox{MR 3798986}}.

\bibitem[Yang(2022)]{yang2022}
Yang, F.,
\newblock \enquote{Propositional union closed team logics},
\newblock {\em Ann. Pure Appl. Logic}, vol.~173 (2022), pp.~Paper No. 103102, 35.\eatcomma,
\newblock URL \url{https://doi.org/10.1016/j.apal.2022.103102}. \sphref{http://www.ams.org/mathscinet-getitem?mr=4385148}{\hbox{MR 4385148}}.

\bibitem[Yang and V\"{a}\"{a}n\"{a}nen(2016)]{yangvaananen2016}
Yang, F. \unskip, and  J.~V\"{a}\"{a}n\"{a}nen,
\newblock \enquote{Propositional logics of dependence},
\newblock {\em Ann. Pure Appl. Logic}, vol.~167 (2016), pp.~557--589.\eatcomma,
\newblock URL \url{https://doi.org/10.1016/j.apal.2016.03.003}. \sphref{http://www.ams.org/mathscinet-getitem?mr=3488885}{\hbox{MR 3488885}}.

\bibitem[Yang and V\"{a}\"{a}n\"{a}nen(2017)]{yang2017}
Yang, F. \unskip, and  J.~V\"{a}\"{a}n\"{a}nen,
\newblock \enquote{Propositional team logics},
\newblock {\em Ann. Pure Appl. Logic}, vol.~168 (2017), pp.~1406--1441.\eatcomma,
\newblock URL \url{https://doi.org/10.1016/j.apal.2017.01.007}. \sphref{http://www.ams.org/mathscinet-getitem?mr=3638896}{\hbox{MR 3638896}}.

\end{thebibliography}
